\documentstyle[amsfonts, txfonts, amssymb, epsfig, latexsym,12pt]{amsart}

\newtheorem{theorem}{Theorem}[section]
\newtheorem{lemma}[theorem]{Lemma}
\newtheorem{proposition}[theorem]{Proposition}
\newtheorem{corollary}[theorem]{Corollary}
\newtheorem{definition}[theorem]{Definition}

\newtheorem{question}[theorem]{Question}

\def\QSet{\mbox{\rm\kern.24em
\vrule width.03em height1.48ex depth-.051ex \kern-.26em Q}}
\def\Z{{\mbox{\rm\kern.25em
\vrule width.03em height0.57ex depth0ex
\kern.033em
\vrule width.03em height1.52ex depth-0.96ex \kern-.338em Z}}}
\def\Z{{\bf Z}}
\def\N{{{\mbox{\rm I\kern-.2em N}}}}
\def\R{{\mbox{\rm I\kern-.22em R}}}
\def\P{{\bf P}}
\def\Q{{\bf Q}}

\def\T{{\mbox{\rm I\kern-.22em T}}}
\def\supp{{\rm supp}}
\def\size{{\rm size}}
\def\diam{{\rm diam}}
\def\energy{{\rm energy}}
\def\sgn{{\rm sgn}}

\def\n{{\bf n}}

\def\C{{\cal{C}}}

\def\G{{\cal{G}}}
\def\I{{\cal{I}}}
\def\l{{\bf l}}

\def\dist{{\rm dist}}

\def\111{\gamma}

\def\be#1{\begin{equation}\label{#1}}
\def\bas{\begin{align*}}
\def\eas{\end{align*}}
\def\bi{\begin{itemize}}
\def\ei{\end{itemize}}
\newenvironment{proof}{\noindent {\bf Proof} }{\endprf\par}
\def \endprf{\hfill  {\vrule height6pt width6pt depth0pt}\medskip}
\def\emph#1{{\it #1}}

\title{Multi-linear multipliers associated to simplexes of arbitrary length}

\author{Camil Muscalu}
\address{Department of Mathematics, Cornell University, Ithaca, NY 14853}
\email{camil@@math.cornell.edu}

\author{Terence Tao}
\address{Department of Mathematics, UCLA, Los Angeles, CA 90024}
\email{tao@@math.ucla.edu}

\author{Christoph Thiele}
\address{Department of Mathematics, UCLA, Los Angeles, CA 90024}
\email{thiele@@math.ucla.edu}

\begin{document}

\begin{abstract}
In this article we prove that the $n$ - linear operator whose symbol is the characteristic function of the simplex
$\Delta_n = \xi_1<...<\xi_n$ is bounded from $L^2\times...\times L^2$ into $L^{2/n}$, generalizing in this way our previous work on the ``bi-est'' operator
\cite{mtt-4}, \cite{mtt-5} (which corresponds to the case $n=3$) as well as the Lacey - Thiele theorem on the bi-linear Hilbert transform \cite{laceyt1}, \cite{laceyt2} 
(which corresponds to the case $n=2$).
\end{abstract}

\maketitle

\section{Introduction}

The present paper is a natural continuation of our previous work in \cite{mtt-4} and \cite{mtt-5}. In those articles we studied the
$L^p$ boundedness properties of a tri-linear operator $T_3$ defined by the formula

\begin{equation}\label{the biest}
T_3(f_1, f_2, f_3)(x) = \int_{\xi_1<\xi_2<\xi_3}
\widehat{f_1}(\xi_1)
\widehat{f_2}(\xi_2)
\widehat{f_3}(\xi_3)
e^{2\pi i x (\xi_1+\xi_2+\xi_3)}
d\xi_1 d\xi_2 d\xi_3
\end{equation}
for $f_1, f_2, f_3$ Schwartz functions on the real line. A particular case of our main theorem there is the following
\footnote{The reader familiar with our earlier work in \cite{mtt-1} should notice that our main result in that paper could not handle the case of $T_3$ 
since the singularity of the symbol $\chi_{\xi_1<\xi_2<\xi_3}$ being two dimensional, is too big; roughly speaking, Theorem 1.1. in \cite{mtt-1} allows
one to prove estimates for $n$-linear multipliers as long as the dimension $k$ of the singularity of the symbol satisfies the inequality $k <\frac{n+1}{2}$.}

\begin{theorem}
$T_3$ extends to a bounded tri-linear operator from $L^2\times L^2\times L^2$ into $L^{2/3}$. 
\end{theorem}

Related to $T_3$ is the well known bi-linear Hilbert transform $T_2$ essentially defined by

\begin{equation}\label{bht}
T_2(f_1, f_2)(x) = \int_{\xi_1<\xi_2}
\widehat{f_1}(\xi_1)
\widehat{f_2}(\xi_2)
e^{2\pi i x(\xi_1+\xi_2)}
d \xi_1 d \xi_2.
\end{equation}

From the work of Lacey and Thiele \cite{laceyt1}, \cite{laceyt2} we know in particular the following result

\begin{theorem}
$T_2$ extends to a bounded bi-linear operator from $L^2\times L^2$ into $L^1$. 
\end{theorem}

The main task of the current paper is to generalize these theorems and prove similar estimates for multi-linear multipliers whose symbols are given
by characteristic functions of simplexes of arbitrary length. More precisely, for any $n\geq 2$ denote by $T_n$ the $n$-linear operator defined by

\begin{equation}\label{multi-est}
T_n(f_1,...,f_n)(x) =
\int_{\xi_1<...<\xi_n}
\widehat{f_1}(\xi_1)...
\widehat{f_n}(\xi_n)
e^{2\pi i x(\xi_1+...+\xi_n)}
d \xi_1...d\xi_n
\end{equation}
where as before $f_1,...,f_n$ are Schwartz functions on $\R$. Our main result is the following
\footnote{Clearly, the correct estimates one is looking for are those of H\"{o}lder type, since when one ignores the symbol
$\chi_{\xi_1<...<\xi_n}$, the corresponding formula in (\ref{multi-est}) becomes the product of the $n$ functions.}

\begin{theorem}\label{mainth}
$T_n$ extends to a bounded $n$-linear operator from $L^2\times...\times L^2$ into $L^{2/n}$.
\end{theorem}

The initial motivation to consider and study such multi-linear operators came from the work of Christ and Kiselev on eigenfunctions of Schr\"{o}dinger operators
\cite{ck1}, \cite{ck2}. Since their theory is quite relevant to our discussion here, we should pause and recall several aspects of it.

For every positive real number $\lambda$ consider the following eigenfunction differential equation

\begin{equation}\label{scheq}
-\Delta u(x) + V(x) u(x) = \lambda u(x)
\end{equation}
for $x$ on the real line, where $V$ is a real valued potential function. Clearly, when $V$ is identically equal to zero every solution of the equation
(\ref{scheq}) can be written as a linear combination of the fundamental solutions $u^{+}_{\lambda}(x) = e^{i\sqrt{\lambda}x}$
and $u^{-}_{\lambda}(x) = e^{-i\sqrt{\lambda}x}$ and as a consequence, it is a bounded function.
The main question addressed in the papers \cite{ck1}, \cite{ck2} was how much can one perturb the free Laplacian $-\Delta$ by a potential function $V$ and still get
bounded corresponding solutions $u_{\lambda}$ of (\ref{scheq}) for almost every $\lambda > 0$ ?

It was known (and not difficult to prove) that the case $V\in L^1(\R)$ is true and it was also known that the case $V\in L^p(\R)$ for $p>2$ is in general false, \cite{simon}.
Christ and Kiselev showed in their papers that the answer to the above question is still affirmative if one considers potential functions in the class $L^p(\R)$ for
any $1\leq p <2$. Roughly speaking (and oversimplifying a lot) the starting point of their proof was to realize that every solution $u_{\lambda}$ of the Schr\"{o}dinger
equation (\ref{scheq}) can be essentially written as a series of expressions of the type

\begin{equation}\label{minus}
\int_{x_1<...<x_n<x} V(x_1)...V(x_n) e^{i\sqrt{\lambda}(x_1-x_2+x_3-...+(-1)^n x_n)} dx_1... dx_n.
\end{equation}
Clearly, in order to prove that every such a formula is bounded (as a function of $x$) for almost every $\lambda > 0$, it is enough to prove $L^p$ bounds for the corresponding
maximal operator $\widetilde{M_n}$ defined by

\begin{equation}\label{maxminus}
\widetilde{M_n}(f_1,...,f_n)(x) = \sup_N
\left|
\int_{x_1<...<x_n<N} f_1(x_1)...f_n(x_n) e^{ix (x_1-x_2+x_3-...+(-1)^n x_n)} dx_1... dx_n.
\right|
\end{equation}
One of the main results in \cite{ck1}, \cite{ck2} says that $\widetilde{M_n}$ is indeed a bounded operator from
$L^p\times...\times L^p$ into $L^{p'/n}$ where $1/p+1/p'=1$ and $1\leq p < 2$, which means that the following inequalities hold

\begin{equation}\label{ineq1}
\|\widetilde{M_n}(f_1,...,f_n)\|_{p'/n} \leq C_{n,p}
\|f_1\|_p...\|f_n\|_p.
\end{equation}
One should also observe that in the simplest case of $L^1$ potentials, one has the trivial pointwise bound

\begin{equation}\label{trivial}
\|\widetilde{M_n}(f_1,...,f_n)\|_{\infty} \leq \frac{1}{n!}\|V\|_1.
\end{equation}
As it turned out \cite{ck1} such a small constant appears also in (\ref{ineq1}) in the place of $C_{n,p}$ in the particular case when $f_1=...=f_n = V$ and this 
essentially allowed the authors of \cite{ck1} to carefully sum up the contributions of all these expressions in (\ref{minus}) and prove the boundedness of the
eigenfunctions.

The case $p=2$ remained open and it still is today.
\footnote{This also explains our predilection for proving $L^2$ estimates only for our operators, even though one can in principle prove many other $L^p$ estimates, 
as we did in \cite{mtt-4} and \cite{mtt-5}.}
If one would like to follow the same strategy, one is naturally led to considering (after using Plancherel) the following sequence of maximal operators
(still denoted by $\widetilde{M_n}$) defined by

\begin{equation}\label{maxminus1}
\widetilde{M_n}(f_1,...,f_n)(x) = \sup_N
\left|
\int_{\xi_1<...<\xi_n<N}\widehat{f_1}(\xi_1)...\widehat{f_n}(\xi_n) e^{2\pi i x(\xi_1-\xi_2+\xi_3-...+(-1)^n\xi_n)} d\xi_1... d\xi_n
\right|
\end{equation}
and their simplified multi-linear variants $\widetilde{T_n}$ given by

\begin{equation}\label{minus1}
\widetilde{T_n}(f_1,...,f_n)(x) = 
\int_{\xi_1<...<\xi_n}\widehat{f_1}(\xi_1)...\widehat{f_n}(\xi_n) e^{2\pi i x(\xi_1-\xi_2+\xi_3-...+(-1)^n\xi_n)} d\xi_1... d\xi_n
\end{equation}
and proving at least $L^2\times... L^2\rightarrow L^{2/n,\infty}$ bounds for each of them.

This was precisely our initial attempt of understanding the $L^2$ question, but before doing anything else we first ``fixed'' their phases and replaced
the $(\widetilde{T_n})_n$ and $(\widetilde{M_n})_n$ with $(T_n)_n$ and $(M_n)_n$ respectively, defined by

\begin{equation}\label{tn}
T_n(f_1,...,f_n)(x) = 
\int_{\xi_1<...<\xi_n}\widehat{f_1}(\xi_1)...\widehat{f_n}(\xi_n) e^{2\pi i x(\xi_1+...+\xi_n)} d\xi_1... d\xi_n
\end{equation}
and 

\begin{equation}\label{mn}
M_n(f_1,...,f_n)(x) = \sup_N
\left|
\int_{\xi_1<...<\xi_n<N}\widehat{f_1}(\xi_1)...\widehat{f_n}(\xi_n) e^{2\pi i x(\xi_1+...+\xi_n)} d\xi_1... d\xi_n
\right|
\end{equation}
since these new operators looked ``more symmetric'' to us and at the time, we believed that their $L^2$ boundedness properties should be similar to the $L^2$ boundedness
properties of the original operators.

In the series of papers \cite{mtt-4}, \cite{mtt-5}, \cite{mtt-7}, \cite{mtt-8} we understood completely the cases of $T_3$ and $M_2$. However, later on when we 
returned to the study of $\widetilde{T_3}$ and $\widetilde{M_2}$ we surprisingly realized that not only they do not satisfy the necessary weak-$L^2$ estimates,
but they don't satisfy any $L^p$ estimates whatsoever \cite{mtt-6}; and the same is true for all the operators $(\widetilde{T_n})_n$ and
$(\widetilde{M_n})_n$ with the only exceptions of $\widetilde{M_1}$ (which is the Carleson operator \cite{carleson}) and $\widetilde{T_2}$ (which essentially
coincides with $H(f_1, f_2)$ where $H$ is the Hilbert transform \cite{stein}).
\footnote{We also showed in \cite{mtt-6} that in spite of all of these, the corresponding eigenfunctions are still bounded functions! After observing this, our strategy
towards proving the $L^2$ Schr\"{o}dinger conjecture changed and we eventually proved a discrete Cantor group model of it, by completely different means.}

Since at least heuristically, the corresponding counterexample for $\widetilde{T_3}$ is not difficult to explain, we will briefly describe it in what follows.

First, let us recall that if one replaces the symbol $\chi_{\xi_1<\xi_2}$ by $\sgn(\xi_1-\xi_2)$ one obtains the kernel representation of the bi-linear Hilbert transform
\cite {laceyt1} given by

\begin{equation}
T_2(f_1, f_2)(x) = \int_{\R}f_1(x-t) f_2(x+t) \frac{dt}{t}.
\end{equation}
Similarly, if one replaces the symbol $\chi_{\xi_1<\xi_2<\xi_3}$ by $\sgn(\xi_1-\xi_2)\cdot \sgn(\xi_2-\xi_3)$ in (\ref{minus1}) and (\ref{tn}) when $n=3$,
one can rewrite the modified $T_3$ and $\widetilde{T_3}$ as

\begin{equation}\label{t3+}
T_3(f_1, f_2, f_3)(x) = \int_{\R^2} f_1(x-t_1) f_2(x+t_1+t_2) f_3(x-t_2)\frac{dt_1}{t_1} \frac{dt_2}{t_2}
\end{equation}
and

\begin{equation}\label{t3-}
\widetilde{T_3}(f_1, f_2, f_3)(x) = \int_{\R^2} f_1(x-t_1) f_2(x-t_1-t_2) f_3(x-t_2)\frac{dt_1}{t_1} \frac{dt_2}{t_2}.
\end{equation}
This time, these are all harmless modifications, since the new resulted operators behave similarly. Now, if one takes $f_1(x)=f_3(x) = e^{ix^2}$ and
$f_2(x)=e^{-ix^2}$ one observes that formally,

$$
\widetilde{T_3}(f_1, f_2, f_3)(x) =
e^{ix^2}\int_{\R^2}e^{it_1 t_2}\frac{dt_1}{t_1}\frac{dt_2}{t_2}
=C e^{ix^2}\int_{\R}\frac{dt}{|t|}.
$$
In other words, we have $\widetilde{T_3}(f_1, f_2, f_3) = C f_1\cdot f_2\cdot f_3 \cdot \int_{\R}\frac{dt}{|t|}$. One can then quantify this equality by restricting the
functions $f_1, f_2, f_3$ to an interval of the form $[-N, N]$. Roughly speaking, one obtains in this way that

\begin{equation}
\widetilde{T_3}(f_1\chi_{[-N, N]}, f_2\chi_{[-N, N]}, f_3\chi_{[-N, N]}) \sim f_1 f_2 f_3 \chi_{[-N, N]} \log N
\end{equation}
and it is precisely this logarithmic factor which determines the failure of any attempt of proving $L^p$ estimates for $\widetilde{T_3}$
(see \cite{mtt-6} for details).

It is also interesting and worth mentioning the fact that if one replaces the bi-parameter kernel $\frac{1}{t_1}\frac{1}{t_2}$ with a classical
Calder\'{o}n-Zygmund kernel $K(t_1, t_2)$ of two variables \cite{stein}, the corresponding trilinear operators (\ref{t3+}) and (\ref{t3-}) behave quite similarly
and they both satisfy many $L^p$ estimates, including the $L^2\times L^2\times L^2\rightarrow L^{2/3}$ one (see \cite{mtt-1}).

Because of these counterexamples, we stopped for a while our study of $(T_n)_n$ and $(M_n)_n$ thinking that maybe their invention was a bit artificial
($(\widetilde{T_n})_n$ and $(\widetilde{M_n})_n$ were, after all, the operators which appeared ``naturally'').
However, more recently, our interest in them has been rekindled by the discovery that they really do appear in connection to a very similar but more general problem related
to the  behaviour of solutions of the so-called AKNS systems which play an important role in nuclear physics \cite{ablowitzsegur}.
We will explain all these connections in detail, later on. The desired $L^2$ boundedness properties for $T_n$ will be described in this paper, while
the corresponding theorem for $M_n$ will be postponed and presented in a future, forthcoming work.
\footnote{ We thus decided to continue our initial program and study not only the sequence $(M_n)_n$ but also the sequence $(T_n)_n$ since they are all very interesting
objects from a purely Fourier analytic point of view (after all, the simplest operator in the $(M_n)_n$ sequence is the Carleson operator \cite{carleson}
while the simplest operator in the $(T_n)_n$ sequence is the bi-linear Hilbert transform.)  }

We still don't have any news regarding the ``$L^2$- Schr\"{o}dinger conjecture'' but we have some interesting (we think) results related to the analogous
``$L^2$-AKNS conjecture'' which we now can prove in the case of upper (and lower) triangular matrices.

The article is organized as follows. In the next section we introduce the AKNS systems and describe their connection with our operators $(T_n)_n$ and $(M_n)_n$.
Then, the rest of the paper is devoted to the proof of the main Theorem \ref{mainth}. We should warn the reader already familiar with our previous $T_3$ papers
\cite{mtt-4} and \cite{mtt-5} that Theorem \ref{mainth} is not a routine generalization of our previous work, since the complexity of $T_n$ for $n\geq 4$ adds some 
fundamentally new features which did not appear in the $T_3$ case. We will unravel them as we move along.

In Section 3, we present a way of decomposing the symbol $\chi_{\xi_1<...<\xi_n}$ naturally, into finitely many slightly smoother pieces. These pieces
are intimately connected with several subregions of the simplex $\Delta_n = \xi_1<...<\xi_n$ which will be described with the help of certain combinatorial
$rooted$ $trees$. As a consequence, our operator $T_n$ will be decomposed as

\begin{equation}\label{dec}
T_n = \sum_G T_n^G
\end{equation}
where the sum in (\ref{dec}) runs over a certain subclass of rooted trees having precisely $n$ leaves. Then, in Section 4 we show how to discretize all 
these operators $T_n^G$ and also show that in order to prove our main theorem it is enough to prove it in the case of these discretized model operators.

Section 5 contains the main part of the actual proof of the theorem and the paper ends with Section 6, in which we prove the ``delicate Bessel'' Lemma
which plays an important role in the argument.

{\bf Acknowledgements:} The first author has  been partially supported by the NSF Grant DMS 0653519 and by an Alfred P. Sloan Research Fellowship.
The second author has been partially supported the NSF Grant DMS 0701302. The third author has been partially supported by the NSF Grant
CCF   0649473 and by a McArthur Fellowship.

\section{AKNS systems and Fourier analysis}

Let $\lambda\in\R$, $\lambda\neq 0$ and consider the system of differential equations

\begin{equation}\label{akns}
u' = i \lambda D u + N u
\end{equation}
where $u = [u_1,...,u_n]^t$ is a vector valued function defined on the real line,
$D$ is a diagonal $n\times n$ constant matrix with real and distinct entries $d_1,...,d_n$
and $N = (a_{ij})_{i,j=1}^n$ is a matrix valued function defined also on the real line and
having the property that $a_{ii}\equiv 0$ for every $i=1,...,n$. These systems play a 
fundamental role in nuclear physics and they are called AKNS systems \cite{ablowitzsegur}.  The particular case
$n=2$ is also known to be deeply connected to the classical theory of Schr\"{o}dinger
operators \cite{ck1}, \cite{ck2}.

If $N\equiv 0$ it is easy to see that our system (\ref{akns}) becomes a union of independent
single equations

$$u'_k = i\lambda d_k u_k$$
for $k=1,...,n$ whose solutions are

$$u_k^{\lambda}(x) = C_{k,\lambda} e^{ i\lambda d_k x}$$
and they are all $L^{\infty}(\R)$-functions. 

As in the case of the Schr\"{o}dinger equation mentioned in the introduction, it is natural to ask how much can one perturb the $N\equiv 0$ case and still obtain 
bounded solutions $(u_k^{\lambda})_{k=1}^n$ for almost every real $\lambda$. As before, the answer is affirmative and easy for $L^1$ entries, very likely to hold true
for $L^p$ entries when $1\leq p <2$ (we have not checked this carefully but we believe that the arguments of \cite{ck1}, \cite{ck2} should be able to be adapted in this 
setting also) and is false for $L^p$ entries if $p>2$ \cite{simon}.

Thus, one is left with the following

\begin{question}
Is it true that as long as the entries of the potential matrix $N$ are $L^2(\R)$ functions, the corresponding solutions $(u_k^{\lambda})_{k=1}^n$
of the AKNS system (\ref{akns}) are all bounded functions for almost every real number $\lambda$ ?
\end{question}

When $N\nequiv 0$ one can use a simple variation of constants argument and write $u_k(x)$
as

$$u_k(x) := e^{ i\lambda d_k x} v_k(x)$$
for $k=1,...,n$. As a consequence, the column vector $v=[v_1,...,v_n]^t$ becomes the solution
of the following system

\begin{equation}\label{111}
v' = W v
\end{equation}
where the entries of $W$ are given by $w_{lm}(x):= a_{lm}(x)e^{ i\lambda (d_l-d_m) x}$.
It is therefore enough to prove that the solutions of (\ref{111}) are bounded as long as the
entries $a_{lm}$ are square integrable.

To get a feeling of the difficulties of the problem, let us first consider the easiest possible
case, that of $2\times 2$ upper triangular matrices. This means that $n=2$ and 
$a_{11}=a_{22}= a_{21}\equiv 0$ while $a_{12}(x):= f(x)$ is an arbitrary $L^2(\R)$ function.
The system (\ref{111}) then becomes

\begin{equation}
\left[
\begin{array}{cc}
v'_1\\
v'_2
\end{array}\right] =
\left[
\begin{array}{cccc}
0 & f(x) e^{i\lambda (d_1-d_2)x}\\
0 & 0
\end{array}
\right]
\left[
\begin{array}{cc}
v_1\\
v_2
\end{array}\right]
\end{equation}
which implies that

\begin{eqnarray*}
v'_1 & = & v_2(x) f(x) e^{i\lambda (d_1-d_2)x}\\
v'_2 & = & 0.
\end{eqnarray*}
Clearly, $v_2$ is bounded since it is constant (which we call $C_{\lambda}$), while $v_1(=v_1^{\lambda})$ can be written as

$$v_1^{\lambda}(x) = C_{\lambda} \int_{-\infty}^x f(y) e^{i\lambda (d_1-d_2)y} dy
+ \widetilde{C}_{\lambda}$$
for some other constant $\widetilde{C}_{\lambda}$. In particular, we have

\begin{equation}\label{122}
\|v_1^{\lambda}\|_{\infty}\leq |C_{\lambda}|\sup_x
\left|\int_{-\infty}^x f(y) e^{i\lambda (d_1-d_2)y} dy
\right| + |\widetilde{C}_{\lambda} |.
\end{equation}
We now recall the Carleson operator $\C$ defined by

$$\C f(x):= \sup_N \left|
\int_{\xi < N}
\widehat{f}(\xi) e^{2\pi i x\xi} d\xi
\right|.
$$
A celebrated theorem of Carleson \cite{carleson} says that $\C$ maps $L^2(\R)$ into $L^2(\R)$
boundedly and in particular this means that $\C f(x) <\infty$ for almost every $x\in\R$, as long
as $f$ is an $L^2(\R)$-function. Using this fact and Plancherel we see from (\ref{122}) that
indeed $\|v_1^{\lambda}\|_{\infty}$ is finite for almost every $\lambda$ which means that the 
conjecture is true in this particular case.

Let us similarly consider now the case of $3\times 3$ upper triangular systems. So this time
$n=3$ and $a_{12}(x):= f_1(x)$, $a_{13}(x):= f_2(x)$, $a_{23}(x):= f_3(x)$ and all the other
entries are identically equal to zero. Our system (\ref{111}) becomes

\begin{equation}
\left[
\begin{array}{cc}
v'_1\\
v'_2\\
v'_3
\end{array}\right] =
\left[
\begin{array}{cccc}
0 & f_1(x) e^{i\lambda (d_1-d_2)x} & f_2(x) e^{i\lambda (d_1-d_3)x}\\
0 & 0 & f_3(x) e^{i\lambda (d_2-d_3)x}\\
0 & 0 & 0
\end{array}
\right]
\left[
\begin{array}{cc}
v_1\\
v_2\\
v_3
\end{array}\right]
\end{equation}
which implies that

\begin{eqnarray*}
v'_1 & = & v_2(x) f_1(x) e^{i\lambda (d_1-d_2)x} + v_3(x) f_2(x) e^{i\lambda (d_1-d_3)x}    \\
v'_2 & = & v_3(x) f_3(x) e^{i\lambda (d_2-d_3)x}\\
v'_3 & = & 0.
\end{eqnarray*}
Clearly, $v_3$ is bounded since it is constant (say $C_{\lambda}$) and exactly as before
$v_2 (= v_2^{\lambda})$ is also bounded for almost every $\lambda$, as a consequence
of the same theorem of Carleson. Since

$$v_2(x) = C_{\lambda}\int_{-\infty}^x f_3(y) e^{i\lambda (d_2-d_3)y} dy
+ \widetilde{C}_{\lambda}$$
it follows that

$$v'_1(x)= C_{\lambda}
\left(\int_{-\infty}^x f_3(y) e^{i\lambda (d_2-d_3)y} dy\right)
f_1(x) e^{i\lambda (d_1-d_2)x} +$$

$$\widetilde{C}_{\lambda}f_1(x) e^{i\lambda (d_1-d_2)x} +
C_{\lambda}f_2(x) e^{i\lambda (d_1-d_3)x}.$$
By taking one more antiderivative, $v_1( = v_1^{\lambda})$ becomes

$$v_1^{\lambda}(x) =
C_{\lambda}
\int_{-\infty}^x f_1(y) e^{i\lambda (d_1-d_2)y}
\left(
\int_{-\infty}^y f_3(z) e^{i\lambda (d_2-d_3)z} dz\right) dy +$$

$$\widetilde{C}_{\lambda}
\int_{-\infty}^x f_1(y) e^{i\lambda (d_1-d_2)y} dy + 
C_{\lambda}
\int_{-\infty}^x f_2(y) e^{i\lambda (d_1-d_3)y} dy + \widetilde{\widetilde{C_{\lambda}}}:= I + II + III + \widetilde{\widetilde{C_{\lambda}}}.
$$
The terms $II$ and $III$ are bounded for almost every $\lambda$ as before, while the first one
can be rewritten as

\begin{equation}\label{133}
C_{\lambda}
\int_{z < y < x}
f_3(z) f_1(y)
e^{i \lambda [(d_2-d_3) z + (d_1-d_2) y]} dz dy.
\end{equation}
Let us now recall the bi-Carleson operator operator $M_2^{\alpha}$ introduced in \cite{mtt-8} and defined by
\footnote{$\alpha=(\alpha_1,\alpha_2)$ is a fixed vector in $\R^2$ so that 
$\alpha_1\neq 0$ and $\alpha_2\neq 0$}

$$M_2^{\alpha}(f, g)(x) =
\sup_N\left|
\int_{\xi_1 < \xi_2 < N}
\widehat{f}(\xi_1)
\widehat{g}(\xi_2)
e^{2\pi i x(\alpha_1\xi_1 + \alpha_2\xi_2)} d\xi_1 d\xi_2
\right|.
$$
A recent theorem in \cite{mtt-8} says that if $\alpha_1 + \alpha_2 \neq 0$ then
$M_2^{\alpha}$ maps $L^2(\R)\times L^2(\R)$ into $L^1(\R)$ and as a consequence 
\footnote{It is also known that if $\alpha_1+\alpha_2=0$ the $M_2^{\alpha}$ does not satisfy any $L^p$ estimates \cite{mtt-6}.   }
this means
that $M_2^{\alpha}(f, g)(x) < \infty$ for almost every $x\in\R$, as long as $f$ and $g$
are $L^2(\R)$ functions. Using this fact and Plancherel again, we see from (\ref{133}) (since $d_2-d_3 + d_1-d_2 = d_1-d_3 \neq 0$)
that $v_1^{\lambda}$ is also bounded for almost every $\lambda$, which means that the
conjecture is also true for upper triangular $3\times 3$ potential matrices $N$.

The case of general upper triangular $n\times n$ matrices for $n\geq 2$ is similar and can 
be reduced to proving $L^2(\R)$ estimates for maximal operators of the form

$$M_k^{\alpha}(f_1,...f_k)(x):=
\sup_N \left|
\int_{\xi_1 <...<\xi_k < N}
\widehat{f_1}(\xi_1)...
\widehat{f_2}(\xi_2)
e^{2\pi i x(\alpha_1 \xi_1 +...+ \alpha_k \xi_k)} d\xi_1...d\xi_k
\right|,$$
where $\alpha_1,...,\alpha_k$ satisfy the nondegeneracy condition

\begin{equation}\label{nondeg}
\sum_{j=j_1}^{j_2} \alpha_j \neq 0
\end{equation}
for every $1\leq j_1<j_2\leq k$. It will be clear from the method of proof that our Theorem \ref{mainth} holds not only for the operators $T_n$ but also for the
operators $T_k^{\alpha}$ defined by

\begin{equation}
T_k^{\alpha}(f_1,...f_k)(x):=
\int_{\xi_1 <...<\xi_k}
\widehat{f_1}(\xi_1)...
\widehat{f_2}(\xi_2)
e^{2\pi i x(\alpha_1 \xi_1 +...+ \alpha_k \xi_k)} d\xi_1...d\xi_k,
\end{equation}
as long as $\alpha_1,...,\alpha_k$ satisfy the nondegeneracy condition (\ref{nondeg}).
\footnote{It is also interesting to remark that in the general $n=2$ case, a standard iterative procedure of Picard type will produce multi-linear expansions where the 
phases are of the form $(d_1-d_2)\xi_1 + (d_2-d_1)\xi_2+...$ as in the Schr\"{o}dinger case. Since one has $d_1-d_2 + d_2 - d_1 =0$, all the corresponding
maximal operators are unbounded.}

\section{Rooted trees and the decomposition of the symbol $\chi_{\xi_1<...<\xi_n}$}

The goal of this section is to carefully decompose the symbol $\chi_{\xi_1<...<\xi_n}$ as a finite sum of several well localized multipliers associated 
with various subregions of the simplex $\Delta_n = \xi_1<...<\xi_n$, which will be best described by using the combinatorial language of $rooted$ $trees$.

To motivate this decomposition procedure we shall briefly revisit the already understood cases of $\Delta_2$ and $\Delta_3$.

\underline{The $\Delta_2$ case.}

This case corresponds to the bi-linear Hilbert transform \cite{laceyt1}. Let us first recall that by a $shifted$ $dyadic$ $interval$ we simply mean
any interval of the form $2^j(k+(0,1) + (-1)^j\alpha)$ for any $k,j\in\Z$ and $\alpha\in \{0, \frac{1}{3}, \frac{2}{3}\}$. Then, for any integer $d$ greated or equal than $1$
a $shifted$ $dyadic$ $quasi-cube$ $of$ $dimension$ $d$ is defined to be any $d$ - dimensional set of the form $Q = Q_1\times...\times Q_d$ having the property that
$|Q_1|\sim...\sim |Q_d|$ where each $Q_i$ is a shifted dyadic interval. Observe as in \cite{mtt-5} that for any arbitrary cube $\widetilde{Q}\subseteq \R^d$
there always exists a shifted dyadic quasi-cube $Q$ so that $\widetilde{Q}\subseteq \frac{7}{10} Q$ and satisfying $l(\widetilde{Q})\sim l(Q)$.
\footnote{$\frac{7}{10}Q$ is the parallelepiped having the same center as $Q$ but $\frac{7}{10}$ times smaller than it. The center of the quasi-cube is defined to be
the $d$ - dimensional point whose $j$th coordinate is the midpoint of the interval $Q_j$. By $l(Q)$ we simply denote the length of the first interval
$|Q_1|$ since they are all of comparable size. Finally, we will write $A\lesssim B$ whenever $A\leq C B$ for some fixed constant $C>0$ and also
$A\sim B$ whenever $A\lesssim B$ and $B\lesssim A$.}
Let us then denote by $\Gamma$ the singularity set

$$\Gamma = \{ (\xi_1, \xi_2)\in \R^2 : \xi_1=\xi_2 \}$$
and consider the collection $\Q$ of all shifted dyadic quasi-cubes $Q$ of dimension $2$ having the property that
$Q\subseteq \Delta_2$ and also satisfying
\footnote{ Here the understanding is that there exists a fixed large constant $C > 0$ so that $C\diam(Q) \leq \dist(Q, \Gamma) \leq 100 C \diam(Q)$. }

\begin{equation}\label{1}
\dist (Q, \Gamma) \sim \diam (Q).
\end{equation}
Since the set of parallelepipeds $\{ \frac{7}{10}Q : Q\in \Q \}$ forms a finitely overlaping cover of $\Delta_2$, by a standard partition of unity we can write the symbol
$\chi_{\xi_1<\xi_2}$ as

\begin{equation}\label {2}
\chi_{\xi_1<\xi_2} = \sum_Q \Phi_Q (\xi_1, \xi_2)
\end{equation}
where $\Phi_Q$ is a bump function adapted to $\frac{8}{10} Q$. By splitting further each $\Phi_Q$ as a double Fourier series in $\xi_1, \xi_2$ we can rewrite
the above expression as

\begin{equation}\label{3}
\sum_{\n\in\Z^2} C_{\n} \sum_Q \Phi_{Q_1, \n, 1}(\xi_1)\Phi_{Q_2, \n, 2}(\xi_2)
\end{equation}
where $(C_{\n})_{\n}$ is a rapidly decreasing sequence and $\Phi_{Q_j, \n, j}$ is a bump function adapted to $\frac{9}{10}Q_j$ uniformly in $\n$ for $j=1,2$, 
(see also \cite{mtt-5}). 

Since $\xi_1\in\frac{9}{10}Q_1$ and $\xi_2\in \frac{9}{10}Q_2$ it follows that $\xi_1 +\xi_2 \in \frac{9}{10}Q_1 + \frac{9}{10}Q_2$ and as a consequence,
one can find a shifted dyadic interval $Q_3$ with the property $\frac{9}{10}Q_1 + \frac{9}{10}Q_2 \subseteq \frac{7}{10} Q_3$ and also satisfying
$|Q_1|\sim |Q_2| \sim |Q_3|$. In particular, there exists a bump function $\Phi_{Q_3, \n, 3}$ adapted to $\frac{9}{10}Q_3$ uniformly in $\n$ such that
$\Phi_{Q_3, \n, 3}\equiv 1$ on $\frac{9}{10}Q_1 + \frac{9}{10}Q_2$. This means that the expression (\ref{3}) can also be written as

\begin{equation}\label{4}
\sum_{\n\in\Z^2} C_{\n} \sum_Q \Phi_{Q_1, \n, 1}(\xi_1)\Phi_{Q_2, \n, 2}(\xi_2)\Phi_{Q_3, \n, 3}(\xi_1+\xi_2),
\end{equation}
where this time $Q$ runs over the corresponding set of shifted dyadic quasi-cubes of dimension $3$.
Generic multipliers of the type $m(\xi_1, \xi_2)= \Phi_1(\xi_1)\Phi_2(\xi_2)\Phi_3(\xi_1+\xi_2)$ are well localized and they allow one to decompose the 
corresponding bi-linear operator $T_m$ nicely \footnote{ In general, if $m(\xi_1, ... , \xi_k)$ is a multiplier, by $T_m$ we denote the $k$-linear operator defined by
$T_m (f_1, ..., f_k)(x) : = \int_{\R^k} m(\xi_1, ... , \xi_k) \widehat{f_1}(\xi_1) ... \widehat{f_k}(\xi_k) e^{2\pi i x( \xi_1+ ... + \xi_k)} d \xi_1 ... d \xi_k$.   }, 
as one can see from the following sequence of equalities.

$$\int_{\R} T_m(f_1, f_2)(x) f_3(x) dx =$$

$$\int_{\R} \left(\int_{\R^2}
\widehat{f_1}(\xi_1)
\widehat{f_2}(\xi_2)
\Phi_1(\xi_1)
\Phi_2(\xi_2)
\Phi_3(\xi_1+\xi_2)
e^{2\pi i x(\xi_1+\xi_2)} 
d\xi_1
d\xi_2\right)
f_3(x) dx =
$$

$$\int_{\R^2}
\widehat{f_1}(\xi_1)
\widehat{f_2}(\xi_2)
\Phi_1(\xi_1)
\Phi_2(\xi_2)
\Phi_3(\xi_1+\xi_2)
\widehat{f_3}(-\xi_1-\xi_2) d\xi_1 d\xi_2 : =
$$

$$\int_{\R^2}
\widehat{f_1}(\xi_1)
\Phi_1(\xi_1)
\widehat{f_2}(\xi_2)
\Phi_2(\xi_2)
\widehat{f_3}(-\xi_1-\xi_2)
\widetilde{\Phi_3}(-\xi_1-\xi_2)
d\xi_1 d\xi_2 = $$

$$
\int_{\lambda_1 + \lambda_2 + \lambda_3 = 0}
(\widehat{f_1\ast \check{\Phi}_1})(\lambda_1)
(\widehat{f_2\ast \check{\Phi}_2})(\lambda_2)
(\widehat{f_3\ast\check{\widetilde{\Phi_3}}})(\lambda_3) d \lambda=
$$

\begin{equation}\label{5}
\int_{\R}
(f_1\ast\check{\Phi}_1)(x)
(f_2\ast\check{\Phi}_2)(x)
(f_3\ast\check{\widetilde{\Phi_3}})(x).
\end{equation}
This also implies that

$$T_m(f_1, f_2)(x) =
[(f_1\ast\check{\Phi}_1)
(f_2\ast\check{\Phi}_2)]\ast
\check{\widetilde{\Phi_3}}.
$$
If one discretizes further the expression (\ref{5}) in the $x$ - variable for each of the similar multipliers appearing in (\ref{4}), one obtains the usual model for the 
bi-linear Hilbert transform \cite{laceyt1}. For reasons that will become clearer later on, we would like to associate to this simplex $\Delta_2$ the simplest
rooted tree having precisely two leaves, as in Figure \ref{fig1}.

\begin{figure}[htbp]\centering
\psfig{figure=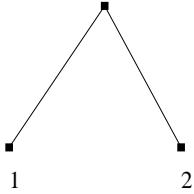, height=1in, width=1in}
\caption{ The rooted tree of the bi-linear Hilbert transform.}
\label{fig1}
\end{figure}

\underline{The $\Delta_3$ case.}

This case corresponds to the ``bi-est'' operator \cite{mtt-4}, \cite{mtt-5}. If $\xi_1<\xi_2<\xi_3$ then clearly there are three possibilities:
either $|\xi_1-\xi_2|<< |\xi_2-\xi_3|$ or $|\xi_2-\xi_3| << |\xi_1-\xi_2|$ or $|\xi_1-\xi_2| \sim |\xi_2-\xi_3|$ each defining different types of regions inside the simplex
$\Delta_3$.
\footnote{$A<<B$ denotes the statement that there exists a large fixed constant $C>0$ so that $CA\leq B$.}

The idea of \cite{mtt-5} was to split the symbol $\chi_{\xi_1<\xi_2<\xi_3}$ as a sum of three distinct symbols

$$\chi_{\xi_1<\xi_2<\xi_3} = m_I + m_{II} + m_{III}$$
well adapted to these regions described above.

To define $m_I$ properly, let us first observe that as before, the symbol $\chi_{\frac{\xi}{2}<\eta}$ can be also decomposed as

\begin{equation}\label{6}
\chi_{\frac{\xi}{2}<\eta} = 
\sum_{\n'\in\Z^2} C'_{\n'} \sum_{Q'} \Phi_{Q'_1, \n', 1}(\xi)\Phi_{Q'_2, \n', 2}(\eta)
\end{equation}
where this time $Q'$ are shifted dyadic quasi-cubes of dimension $2$ having the property that $Q'\subseteq \{ (\xi,\eta)\in\R^2 : \frac{\xi}{2} < \eta \}$
and also that

$$\dist (Q', \Gamma') \sim \diam (Q')$$
where $\Gamma'$ denotes the singularity line

$$\Gamma' = \{ (\xi,\eta)\in\R^2 : \frac{\xi}{2} = \eta \}.
$$
In particular, if $(\xi_1, \xi_2, \xi_3)$ is a fixed point in the simplex $\Delta_3$ which belongs to the first region $|\xi_1-\xi_2|<<|\xi_2-\xi_3|$, one
clearly has not only $\xi_2 <\xi_3$ but also $\frac{\xi_1+\xi_2}{2} < \xi_3$ and this implies that (using (\ref{3}) and (\ref{6}))

$$1 = \chi_{\xi_1<\xi_2}\cdot \chi_{\frac{\xi_1+\xi_2}{2}<\xi_3}(\xi_1, \xi_2, \xi_3) = $$

\begin{equation}\label{7}
\sum_{\n, \n', Q, Q'}
C_{\n} C'_{\n'}
\Phi_{Q_1, \n, 1}(\xi_1)\Phi_{Q_2, \n, 2}(\xi_2)
\Phi_{Q'_1, \n', 1}(\xi_1+\xi_2)\Phi_{Q'_2, \n', 2}(\xi_3)
\end{equation}
It is also not difficult to see that the last expression (\ref{7}) is also equal to

$$
\sum_{\n, \n', Q, Q' : l(Q') >> l(Q)}
C_{\n} C'_{\n'}
\Phi_{Q_1, \n, 1}(\xi_1)\Phi_{Q_2, \n, 2}(\xi_2)
\Phi_{Q'_1, \n', 1}(\xi_1+\xi_2)\Phi_{Q'_2, \n', 2}(\xi_3)
$$ with the implicit constants independent on the fixed point $(\xi_1, \xi_2, \xi_3)$ and dependent only on the corresponding constants defining the first region.

Then, one defines the symbol $m_I$ simply by

\begin{equation}\label{mI}
m_I(\xi_1, \xi_2, \xi_3):= 
\sum_{\n, \n', Q, Q' : l(Q') >> l(Q)}
C_{\n} C_{\n'}
\Phi_{Q_1, \n, 1}(\xi_1)\Phi_{Q_2, \n, 2}(\xi_2)
\Phi_{Q'_1, \n', 1}(\xi_1+\xi_2)\Phi_{Q'_2, \n', 2}(\xi_3)
\end{equation}
for any $(\xi_1, \xi_2, \xi_3)\in\R^3$. Clearly, by construction, $m_I$ is identically equal to $1$ on the first region and is also supported on another larger region
of the same type (defined by different constants) which is contained inside the simplex $\Delta_3$. Similarly, one defines the symbol $m_{II}$ adapted to the second region
$|\xi_2-\xi_3|<< |\xi_1 - \xi_2|$ and in the end one sets

$$m_{III}:= \chi_{\Delta_3} - m_I - m_{II}$$
which clearly is supported inside a region of the third type $|\xi_1-\xi_2| \sim |\xi_2-\xi_3|$.

As before (when we passed to (\ref{4}) from (\ref{3})), one can ``complete'' the expressions in (\ref{mI}) obtaining products of type

\begin{equation}\label{8}
m(\xi_1, \xi_2, \xi_3) =
\Phi_1(\xi_1)
\Phi_2(\xi_2)
\Phi_3(\xi_1+\xi_2)
\Phi'_1(\xi_1+\xi_2) 
\Phi'_2(\xi_3)
\Phi'_3(\xi_1 +\xi_2 +\xi_3).
\end{equation}
The reason for which we prefered to describe the region $|\xi_1-\xi_2|<<|\xi_2 - \xi_3|$ as rather being $|\xi_1-\xi_2|<<|\frac{\xi_1+\xi_2}{2}-\xi_3|$
will be clearer when we calculate the $4$ - linear form associated to the tri-linear operator $T_m$ given by symbols of type (\ref{8}). This time we can write

$$\int_{\R}T_m(f_1, f_2, f_3)(x) f_4(x) dx =$$

$$\int_{\R^3}
\widehat{f_1}(\xi_1)
\widehat{f_2}(\xi_2)
\widehat{f_3}(\xi_3)
\Phi_1(\xi_1)
\Phi_2(\xi_2)
\Phi_3(\xi_1+\xi_2)
\Phi'_1(\xi_1+\xi_2) 
\Phi'_2(\xi_3)
\Phi'_3(\xi_1 +\xi_2 +\xi_3)
\widehat{f_4}(-\xi_1-\xi_2-\xi_3) d\xi :=
$$

$$\int_{\R^3}
\widehat{f_1}(\xi_1)
\widehat{f_2}(\xi_2)
\widehat{f_3}(\xi_3)
\Phi_1(\xi_1)
\Phi_2(\xi_2)
\Phi_3(\xi_1+\xi_2)
\Phi'_1(\xi_1+\xi_2) 
\Phi'_2(\xi_3)
\widetilde{\Phi'_3}(-\xi_1 -\xi_2 -\xi_3)
\widehat{f_4}(-\xi_1-\xi_2-\xi_3) d\xi =
$$

$$\int_{\lambda_1 +\lambda_2 +\lambda_3=0}
\left(
\int_{\xi_1+\xi_2=\lambda_1}
\widehat{f_1}(\xi_1)
\widehat{f_2}(\xi_2)
\Phi_1(\xi_1)
\Phi_2(\xi_2)
d\xi_1 d\xi_2
\right)
\Phi_3(\lambda_1)
\Phi'_1(\lambda_1)\cdot
$$

$$
\Phi'_2(\lambda_2)
\widehat{f_3}(\lambda_2)
\widetilde{\Phi'_3}(\lambda_3)
\widehat{f_4}(\lambda_3) d\lambda =
$$

$$
\int_{\lambda_1 +\lambda_2 +\lambda_3=0}
\left(
\widehat{\left[ [(f_1\ast\check{\Phi}_1)(f_2\ast\check{\Phi}_2)]\ast\check{\Phi}_3\right] \ast\check{\Phi'}_1    }
\right)(\lambda_1)
(\widehat{f_3\ast\check{\Phi'}_2})(\lambda_2)
(\widehat{f_4\ast\check{\widetilde{\Phi'_3}}})(\lambda_3) d\lambda =
$$

$$\int_{\R}
\left(
\left[ [(f_1\ast\check{\Phi}_1)(f_2\ast\check{\Phi}_2)]\ast\check{\Phi}_3\right] \ast\check{\Phi'}_1
\right)(x)
( f_3\ast\check{\Phi'}_2   )(x)
(  f_4\ast\check{\widetilde{\Phi'_3}}   )(x) dx.
$$

As before, if one discretizes further this expression in the $x$ - variable for each of the similar multipliers appearing in (\ref{mI}),
one obtains the discretized model for the ``bi-est'' operator in \cite{mtt-5}. To each of the three regions described above, we associate a rooted tree as in the 
Figure \ref{fig2}.

\begin{figure}[htbp]\centering
\psfig{figure=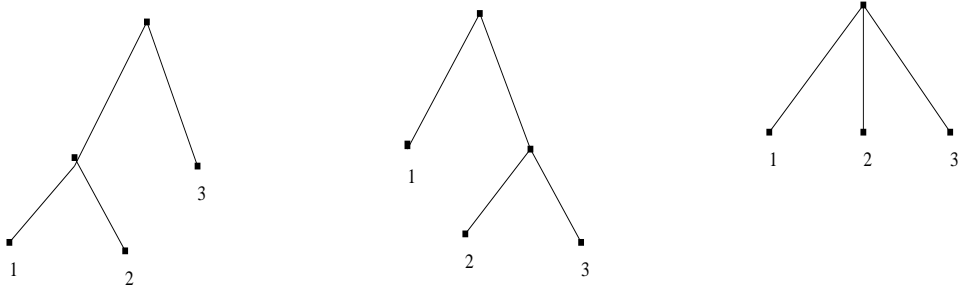, height=1.5in, width=5in}
\caption{The rooted trees of the ``bi-est''.}
\label{fig2}
\end{figure}

\underline{The $\Delta_n$ case.}

We denote generically by $G$ an arbitrary rooted tree. Let us recall that each vertex of $G$
has a $level$. There is precisely one vertex, $the$ $root$, which has level $0$.
All adjacent vertices differ by exactly one level and each vertex at level $i+1$ is adjacent
to exactly one vertex at level $i$. All the vertices adjacent to a fixed vertex $u$ and at
levels below $u$ are called $the$ $sons$ of $u$. Also, all the vertices that are joined to $u$
by a chain of vertices at levels below $u$, are called $descendants$ of $u$. The vertices which
do not have sons are called $leaves$.

We will consider only those rooted trees having precisely $n$ leaves which we lable with
numbers from $1$ to $n$ from the most left one to the most right one. Also, our trees have the
property that every vertex of the tree which is not a leave has at least two sons.
We denote the class of all such rooted trees by $\G_n$.

The maximal possible level in a tree is called $the$ $height$ of the tree. We will alo denote by
$V_G$ the set of all vertices of $G$ which are not leaves.

If $u\in V_G$, we denote by $\I_u$ the collection of all the integers $1\leq i\leq n$
having the property that the leave labeled ``$i$'' is a descendant of $u$. It is not difficult
to see that there exist two integers $1\leq l_u < r_u\leq n$ so that

$$\I_u = \{l_u, l_u+1,...,r_u\}.$$ In case $u$ is a leave labled $i_0$, we simply
set $\I_u:= \{i_0\}$.
Now, if $\xi_1<...<\xi_n$, we denote by $I_1, I_2,...I_{n-1}$ the intervals
$[\xi_1,\xi_2]$, $[\xi_2,\xi_3]$,..., $[\xi_{n-1},\xi_n]$ respectively.

Fix now $G\in \G_n$ and $u\in V_G$. Denote also by $u_1, u_2,...,u_{\#}$ all the sons of $u$.
To this vertex $u$ we associate a region $R_u\subseteq \Delta_n$ defined to be the set
all all vectors $(\xi_1,...,\xi_n)\in\Delta_n$ having the property that

\begin{equation}\label{constraints}
|I_{r_{u_1}}|\sim |I_{r_{u_2}}|\sim...\sim |I_{r_{u_{\#-1}}}| >> |I_l|
\end{equation}
for every $l_u\leq l\leq r_u - 1$ and $l\neq r_{u_1}, r_{u_2},...,r_{u_{\#-1}}$.
\footnote{In other words, the region $R_u$ is the subregion of $\Delta_n$ defined by the constraints in (\ref{constraints}). Note that sometimes it may happen that
some (or all !) of the constraints ``do not make sense'', in which case they should be simply disregarded. For instance, if all the sons of $u$ are leaves 
(as in Figure \ref{fig4}), then there is no $l$ with the property  $l_u\leq l\leq r_u - 1$ and $l\neq r_{u_1}, r_{u_2},...,r_{u_{\#-1}}$. In this case, the constraints
defining the region are only the first ones, namely 
$|I_{r_{u_1}}|\sim |I_{r_{u_2}}|\sim...\sim |I_{r_{u_{\#-1}}}|$. If in addition $u$ has only two sons 
(and they are both leaves as before), then one can see that the sequence of inequalities 
$|I_{r_{u_1}}|\sim |I_{r_{u_2}}|\sim...\sim |I_{r_{u_{\#-1}}}|$   becomes redundant
since $\#-1 = 1$. In this case, there are simply no constraints and as a consequence the corresponding region $R_u$ coincides with the whole simplex $\Delta_n$ }

Then, we define the region $R_G\subseteq \Delta_n$ by

\begin{equation}\label{RG}
R_G:= \bigcap_{u\in V_G} R_u.
\end{equation}
By a region $R( = R_G)$ of $type$ $G$ we simply mean from now on a region defined in this way
for various implicit constants.

For instance, if $G$ is the rooted tree on the left in Figure \ref{fig3} then by a region of 
type $G$
we mean one defined by the inequalities 
$(|I_3|\sim |I_4| >> |I_1|, |I_2|, |I_5|) \cap (|I_1| >> |I_2|)$ and if $G$ is the tree
on the right in Figure \ref{fig3}, then by a region of type $G$ we mean one given by the 
inequalities $(|I_4| >> |I_1|, |I_2|, |I_3|, |I_5|) \cap (|I_2| >> |I_1|, |I_3|) $.

\begin{figure}[htbp]\centering
\psfig{figure=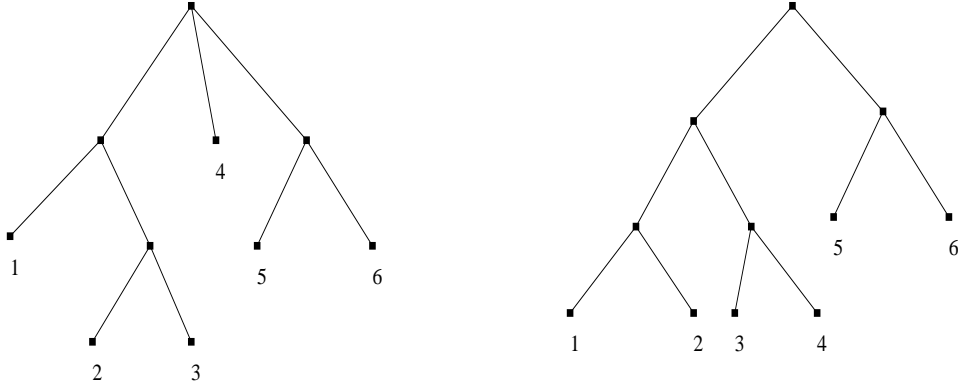, height=2in, width=5in}
\caption{Two rooted trees of height three.}
\label{fig3}
\end{figure}

It is now easy to see that if the implicit constants are chosen carefully, one can decompose
the simplex $\Delta_n$ as

$$\Delta_n = \bigcup_{G\in\G_n} R_G.$$
We would now like to associate to each $G\in \G_n$ a multiplier $m_G$ adapted to such a region
$R_G$ constructed before.
 We first need the following definition.

\begin{definition}\label{Ra}
Let $d\geq 2$ and $\vec{a} = (a_1,...,a_d)$ be a vector with real and strictly positive entries.
Denote by $R_{\vec{a}}$ the region of all vectors $(x_1,...,x_d)\in \R^d$ satisfying

$$a_1x_1 < a_2x_2 <...< a_dx_d$$
and also 

$$|a_1x_1 -a_2x_2| \sim |a_2x_2-a_3x_3| \sim ... \sim |a_{d-1}x_{d-1} - a_d x_d|.$$
A shifted dyadic quasi-cube $Q = (Q_1, Q_2,..., Q_d)$ of dimension $d$ is said to be adapted
to the region $R_{\vec{a}}$ if and only if for every $j=1,...,d-1$ the quasi-cubes 
$Q^j:= (Q_j, Q_{j+1})$ have the property that 

$$ Q^j\subseteq \{(x_j, x_{j+1})\in \R^2 : a_jx_j < a_{j+1}x_{j+1} \}$$
and also that

$$\dist(Q^j, \Gamma^j) \sim \diam (Q^j)$$
where

$$\Gamma^j := \{(x_j, x_{j+1})\in\R^2 : a_jx_j = a_{j+1}x_{j+1}\}.$$ 

\end{definition}
We are now ready to define a standard symbol associated to $G$ as follows. First, for
$u\in V_G$ denote by $u_1,...,u_{\#}$ all its sons and consider the region
$R_{\vec{a}}$ associated to the vector

$$\vec{a} := (\frac{1}{|\I_{u_1}|},...,\frac{1}{|\I_{u_{\#}}|})$$
as in Definition \ref{Ra}.

We denote by $m_u$ any expression of the form

\begin{equation}\label{9}
m_u((\xi_l)_{l\in\I_u}) =
\sum_{Q_u}
\Phi_{Q_u^1, 1}(\sum_{l\in\I_{u_1}} \xi_l)
...
\Phi_{Q_u^{\#}, \#}(\sum_{l\in\I_{u_{\#}}} \xi_l)
\end{equation}
where the sum in taken over shifted dyadic quasi-cubes adapted to the region $R_{\vec{a}}$
above in the sense of Definition \ref{Ra}, and $\Phi_{Q_u^j, j}$ are bumps adapted to
$\frac{9}{10}Q_u^j$ for any $j=1,...,\#$.

The expression (\ref{9}) can also be written as 

$$m_u:= \sum_{k\in\Z} m_u^k
$$
where $m_u^k$ is defined by the same formula (\ref{9}) but with the aditional constraint that
$l(Q_u)\sim 2^k$.

In the end, a multiplier $m_G$ corresponding to the tree $G$ is defined to be
any expression of the form

\begin{equation}\label{mG}
m_G :=
\sum_{\n} C_{\n}
\sum_{(k_1, k_2,...,k_{|V_G|})\in S_G}
\left(
\prod_{l=1}^{|V_G|} m^{k_l}_{v_l, \n}
\right)
\end{equation}
where we assumed that $V_G = \{v_1,...,v_{|V_G|}\}$, the sequence $(C_{\n})_{\n}$ is a rapidly 
decreasing sequence indexed over a countable set, and by $S_G$ we denoted
the collection of all tuples of integers $(k_1,...,k_{|V_G|})$ having the property that
$k_{l'}>> k_{l''}$ as long as the vertices $v_{l'}$ and $v_{l''}$ are adjacent and
$v_{l''}$ is the son of $v_{l'}$.
\footnote{Also, by $m^{k_l}_{v_l, \n}$ we clearly mean an expression of the same type
with the one in (\ref{9}) but with the corresponding bumps $\Phi_{Q_{v_l}^i, i}$ (for which
$l(Q_{v_l})\sim 2^{k_l}$)
being replaced by similar ones $\Phi_{Q_{v_l}^i, \n, i}$ which have the same properties,
uniformly in $\n$.}

The following Lemma will be very helpful.

\begin{lemma}\label{help}
Let $d$, $\vec{a}$, $R_{\vec{a}}$ be as in Definition \ref{Ra}. Then, there exists a symbol
$m_{\vec{a}}$ of the form

\begin{equation}\label{10}
m_{\vec{a}}(x_1,...,x_d) =
\sum_{\n} C_{\n}
\sum_Q
\Phi_{Q_1, \n, 1}(x_1)
...
\Phi_{Q_d, \n, d}(x_d)
\end{equation}
so that $m_{\vec{a}}| R_{\vec{a}} \equiv 1$, where $(C_{\n})_{\n}$ is a rapidly decreasing sequence,
$\Phi_{Q_j, \n, j}$ is a bump adapted to $\frac{9}{10}Q_j$ uniformly in $\n$ and the sum in
(\ref{10}) is taken over $\n$ belonging to a certain countable set and $Q$ belonging to
a certain collection of shifted dyadic quasi-cubes adapted to $R_{\vec{a}}$.
\end{lemma}

\begin{proof}
Using our standard procedure, for each $i=1,...,d-1$ one can decompose as before the symbol
$\chi_{a_ix_i <a_{i+1}x_{i+1}}$ as

\begin{equation}\label{11}
\chi_{a_ix_i <a_{i+1}x_{i+1}} = 
\sum_{\n_i\in \Z^2}C_{\n_i}^i
\sum_{Q^i}
\Phi_{Q^i_1, \n_i, 1}(x_i)
\Phi_{Q^i_2, \n_i, 2}(x_{i+1})
\end{equation}
where as usual $(C^i_{\n_i})_{\n_i}$ is a rapidly decreasing sequence, 
$\Phi_{Q^i_1, \n_i, 1}$, $\Phi_{Q^i_2, \n_i, 2}$ are bumps adapted to 
$\frac{9}{10}Q^i_1$ and $\frac{9}{10}Q^i_2$ respectively and the sum is taken over
shifted dyadic quasi-cubes $Q^i$ of dimension $2$ inside the region
$\{(x_i, x_{i+1})\in\R^2 : a_ix_i < a_{i+1}x_{i+1} \}$ and satisfying

$$\dist (Q^i, \Gamma^i) \sim \diam (Q^i)$$
where $\Gamma^i$ has been defined in Definition \ref{Ra}.

Now, if one takes the product of all the $d-1$ expressions in (\ref{11}) and if one restricts
the summation to those $Q^1,...,Q^{d-1}$ for which $l(Q^1)\sim...\sim l(Q^{d-1})$, one gets
an expression (named ``$m_{\vec{a}}$'') which can clearly be written in the form required by the 
Lemma and which
also satisfies $m_{\vec{a}}| R_{\vec{a}} \equiv 1$ if the implicit constants are chosen
appropriately.

\end{proof}
The following Lemma will also play an important role.

\begin{lemma}\label{mG1}
Let $G\in\G_n$ and $R_G$ be a fixed region of type $G$. Then, there exists a standard symbol
$m_G$ of type $G$ having the property that 

$$m_G|R_G \equiv 1$$
and also that $\supp (m_G)$ is included inside a larger region of the same type $G$.
\end{lemma}

\begin{proof}
We will prove this by induction with respect to the height of the rooted tree. First, let us 
consider the case of a rooted tree $G$ having height $1$ and an arbitrary number of leaves
$L$ for $2\leq L\leq n$  (see Figure \ref{fig4}).

\begin{figure}[htbp]\centering
\psfig{figure=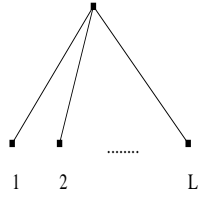, height=1in, width=1in}
\caption{Rooted tree of height $1$.}
\label{fig4}
\end{figure}

The region $R_G$ is then described as being the set
of vectors $(\xi_1,...,\xi_L)$ so that $\xi_1<...<\xi_L$ and satisfying

$$|\xi_1-\xi_2|\sim |\xi_2-\xi_3|\sim ...\sim |\xi_{L-1}-\xi_L|.
$$
The fact that a multiplier $m_G$ of type $G$ having the property that
$m_G|R_G\equiv 1$ exists, is in this case a simple consequence of the previous Lemma \ref{help}.

Let us now assume that our statement holds for any rooted tree $G$ having height smaller
or equal than $h$ and any number of leaves $L$ and we would like to prove that it holds for trees
of height $h+1$ and any number of leaves $L$.

Fix $G$ a rooted tree with height $h+1$ and $L$ leaves, for some $2\leq L\leq n$.
Denote by $u$ the root of it and as usual by $u_1,...,u_{\#}$ the sons of $u$.
If $u_j$ is not a leave, denote also by $G_j$ the corresponding sub-tree of $G$ whose root
is $u_j$. Clearly, all of these sub-trees have heights smaller or equal than $h$.

Fix also $R_G$ a region of type $G$. This region defines the ``projected regions''
$R_{G_j}$ given by

$$R_{G_j}: = \{ (\xi_l)_{l\in\I_{u_j}} : (\xi_l)_{l=1}^L\in R_G \}.$$
By the induction hypothesis, there exist multipliers $m_{G_j}$ of type $G_j$ having the 
property that $m_{G_j}|R_{G_j}\equiv 1$ and so that their supports are also included
 into a slightly larger region of the same corresponding type $G_j$.

Fix now $(\xi_l)_{l=1}^L\in\R_G$. Clearly, since $(\xi_l)_{l\in\I_{u_j}}\in R_{G_j}$
it follows that $m_{G_j}((\xi_l)_{l\in\I_{u_j}}) = 1$.

On the other hand, from the definitions of the regions $R_G$ of type $G$ we also know that

\begin{equation}\label{ineq}
|I_{r_{u_1}}|\sim |I_{r_{u_2}}|\sim...\sim |I_{r_{u_{\#-1}}}| >> |I_l|
\end{equation}
for every $l_u\leq l\leq r_u - 1$ and $l\neq r_{u_1}, r_{u_2},...,r_{u_{\#-1}}$. 

In particular, this implies that the following inequalities hold

\begin{equation}\label{12}
\frac{1}{|\I_{u_1}|}
\sum_{l\in\I_{u_1}} \xi_l <
\frac{1}{|\I_{u_2}|}
\sum_{l\in\I_{u_2}} \xi_l<
...
< \frac{1}{|\I_{u_{\#}}|}
\sum_{l\in\I_{u_{\#}}}\xi_l  
\end{equation}
and also that

\begin{equation}\label{13}
\left|
\frac{1}{|\I_{u_1}|}
\sum_{l\in\I_{u_1}} \xi_l -
\frac{1}{|\I_{u_2}|}
\sum_{l\in\I_{u_2}} \xi_l
\right|\sim ... \sim
\left|
\frac{1}{|\I_{u_{\#-1}}|}
\sum_{l\in\I_{u_{\#-1}}} \xi_l -
\frac{1}{|\I_{u_{\#}}|}
\sum_{l\in\I_{u_{\#}}} \xi_l
\right|.
\end{equation}
As a consequence of Lemma \ref{help}, if we denote as before by
$\vec{a} := (\frac{1}{|\I_{u_1}|},...,\frac{1}{|\I_{u_{\#}}|})$
we know that we can find a multiplier $m_{\vec{a}}$ having the property that
$m_{\vec{a}}|R_{\vec{a}}\equiv 1$. In particular, from (\ref{12}) and (\ref{13})
we see that

$$m_{\vec{a}}(\sum_{l\in\I_{u_1}}\xi_l,...,\sum_{l\in\I_{u_{\#}}}\xi_l) = 1$$
which implies that

\begin{equation}\label{14}
1 = \left(
\prod_j m_{G_j}((\xi_l)_{l\in\I_{u_j}})
\right)
\cdot 
m_{\vec{a}}(\sum_{l\in\I_{u_1}}\xi_l,...,\sum_{l\in\I_{u_{\#}}}\xi_l).
\end{equation}
Denote now by $j_1,...,j_s$ those indices ``$l$'' between $1$ and $\#$ for which
the corresponding son $u_l$ of $u$ is not a leave. It is not difficult to see
that this last expression (\ref{14}) is also equal to

\begin{equation}\label{15}
\sum_{k >> k_{j_1},...,k_{j_s}}
m_{G_{j_1}}^{k_{j_1}}((\xi_l)_{l\in\I_{u_{j_1}}})...
m_{G_{j_s}}^{k_{j_s}}((\xi_l)_{l\in\I_{u_{j_s}}})
m_{\vec{a}}^k(\sum_{l\in\I_{u_1}}\xi_l,...,\sum_{l\in\I_{u_{\#}}}\xi_l),
\end{equation}
where in general, by $m_{\tilde{G}}^{\tilde{k}}$ we denote the multiplier
defined by the same formula (\ref{mG}) but with the additional constraint that the summation
index corresponding to the root of $\tilde{G}$ is kept constant and equal to $\tilde{k}$.
Similarly, by $m_{\vec{a}}^k$ we denote the multiplier defined by the same expression
in Lemma \ref{help} but again with the additional constraint that the summation over $Q$
is restricted to those for which $l(Q)\sim 2^k$.

It is also important to note that the implicit constants in (\ref{15}) are independent on the 
previously fixed vector $(\xi_l)_{l=1}^L\in R_G$
(and dependent only on the constants defining the region). Then, one simply defines the desired
multiplier $m_G$ by the same similar expression in (\ref{15}), and then one can observe
that $m_G$ has all the desired properties.
\end{proof}

Having all these constructions at our disposal, we can actually start describing
the decomposition of the symbol $\chi_{\Delta_n}$. Assume for simplicity that
$\G_n = \{G_1,...,G_N\}$. Clearly, there exist regions $R_{G_1},...,R_{G_N}$ corresponding to 
these rooted trees so that

$$\chi_{\Delta_n} = R_{G_1}\cup...\cup R_{G_N}.$$
Let $m_{G_1}$ be a symbol of type $G_1$ satisfying $m_{G_1}|R_1\equiv 1$ as in Lemma \ref{mG1}
and define

\begin{equation}\label{16}
m^1:= \chi_{\Delta_n} - m_{G_1}.
\end{equation}
Observe that $m^1|R_1\equiv 0$ and as a consequence, we have that

$$\supp (m^1) \subseteq R_2\cup...\cup R_N.$$
Then, consider $m_{G_2}$ a symbol  of type $G_2$ satisfying $m_{G_2}|R_2\equiv 1$
again as in Lemma \ref{mG1} and define

\begin{equation}\label{17}
m^2:= m^1 - m^1\cdot m_{G_2}.
\end{equation}
Observe as before that

$$\supp (m^2) \subseteq \supp (m^1) \subseteq R_2\cup...\cup R_N$$
and since $m^2|R_2\equiv 0$ it follows that actually

$$\supp (m^2) \subseteq R_3\cup...\cup R_N.$$
One can continue in this way and define $m^3,..., m^{N-1}$ recursively and in the end
one observes that $m^N$ defined by

\begin{equation}\label{18}
m^N:= m^{N-1} - m^{N-1}\cdot m_{G_N}
\end{equation}
has the property that $m_N\equiv 0$ since $\supp (m_N)\subseteq R_N$ and
$m_{G_N}|R_N\equiv 1$.

Adding all these equalities (\ref{16}), (\ref{17}), (\ref{18}) together
we obtain that $\chi_{\Delta_n}$ can be written as

\begin{equation}\label{19}
\chi_{\Delta_n} = m_{G_1} + m^1\cdot m_{G_2} + m^2\cdot m_{G_3} + ... + 
m^{N-1}\cdot m_{G_N}.
\end{equation}
Moreover, since $m^1$ has an explicit formula, all the symbols $m^i$ are explicit since they 
have been
defined recursively. It is easy to remark that our symbol $\chi_{\Delta_n}$ can as a result
be written as a finite sum of products of multipliers of type

$$m_{G_{i_1}}\cdot m_{G_{i_2}}\cdot...\cdot m_{G_{i_k}}.$$
However, the next Proposition shows that these product symbols are essentially of the same
type as the standard ones considered above.

Before stating it, we first need to recall the following definition from \cite{mtt-5}.

\begin{definition}
A subset $\Q$ of shifted dyadic quasi-cubes is said to be sparse if and only if for any
two quasi-cubes $Q, Q'\in \Q$ with $Q\neq Q'$ we have $|Q| < |Q'|$ implies
$|C Q| < |Q'|$ and $|Q| = |Q'|$ implies $CQ \cap CQ' = \emptyset$, where $C > 0$ is a fixed 
large constant.
\end{definition}

We would like to assume from now on, without losing the generality, that all our families
of quasi-cubes that implicitly enter the formulas (\ref{mG}) are sparse.

The following simple Lemma will also play an important role later on.

\begin{lemma}\label{fixing}
Let $d\geq 2$, $Q = (Q_1,...,Q_d)$ be a shifted dyadic quasi-cube, $I$ be a shifted dyadic
interval so that $|I| \sim \diam (Q)$ and $0<\alpha<\beta<1$. Consider also bump functions
$\Phi_{Q_1},...,\Phi_{Q_d}, \Phi_I$ adapted to $\alpha Q_1,...,\alpha Q_d$ and
$\alpha I$ respectively. Then, there exists a sequence $(C_{\n})_{\n}$ of rapidly decreasing
complex numbers (independent on $Q$, $I$ !) and bump functions $\widetilde{\Phi_{Q_1, \n}},...,\widetilde{\Phi_{Q_d, \n}}$
uniformly adapted to $\beta Q_1,..., \beta Q_d$ respectively, so that

\begin{equation}\label{fixingg}
\Phi_{Q_1}(x_1)...\Phi_{Q_d}(x_d) \cdot \Phi_I(x_1+...+x_d) =
\sum_{\n} C_{\n} \widetilde{\Phi_{Q_1, \n}}(x_1)...\widetilde{\Phi_{Q_d, \n}}(x_d).
\end{equation}

\end{lemma}

\begin{proof}
First, consider bump functions $\widetilde{\Phi_{Q_1}},...,\widetilde{\Phi_{Q_d}}$
adapted to $\beta Q_1,...,\beta Q_d$ and having the property that

$$\widetilde{\Phi_{Q_j}}\equiv 1$$
on the support of $\Phi_{Q_j}$ for every $j=1,...,d$. In particular, the left hand side
of (\ref{fixingg}) can be rewritten as

$$\widetilde{\Phi_{Q_1}}(x_1)...\widetilde{\Phi_{Q_d}}(x_d)\cdot
[\Phi_{Q_1}(x_1)...\Phi_{Q_d}(x_d) \cdot \Phi_I(x_1+...+x_d)] :=
\widetilde{\Phi_{Q_1}}(x_1)...\widetilde{\Phi_{Q_d}}(x_d)\cdot m(x_1,...,x_d).$$
Then, one just has to write $m$ as a multiple Fourier series in the variables
$x_1,...,x_d$ on $Q$ and to take advantage of the smoothness of $m$.

\end{proof}

\begin{proposition}\label{G}
Let $G_1, G_2 \in \G_n$ and $m_{G_1}$, $m_{G_2}$ symbols associated to $G_1$ and $G_2$ 
respectively. Then, there exists a rooted tree $G\in \G_n$ so that

$$m_{G_1}\cdot m_{G_2} = m_G$$
for a certain symbol $m_G$ of type $G$.
\end{proposition}

\begin{proof}
Fix $G_1, G_2 \in \G_n$. If $m_{G_1}$ and $m_{G_2}$ are symbols of type $G_1$ and $G_2$
respectively, then one can write

\begin{equation}\label{product}
m_{G_1}\cdot m_{G_2} =
(\sum_k m_{G_1}^k) (\sum_{\widetilde{k}} m_{G_2}^{\widetilde{k}}) =
\sum_{k, \widetilde{k}} m_{G_1}^k m_{G_2}^{\widetilde{k}} =
\sum_{k \sim \widetilde{k}} m_{G_1}^k m_{G_2}^{\widetilde{k}}
\end{equation}
since $m_{G_1}^k m_{G_2}^{\widetilde{k}}\equiv 0$ unless $k\sim \widetilde{k}$.
Clearly, since both $k$ and $\widetilde{k}$ run inside sparse sets, for any fixed
$k$ there is a unique $\widetilde{k}$ for which $k\sim \widetilde{k}$. By abuse of notation
we will denote from now on the corresponding $m_{G_2}^{\widetilde{k}}$ simply by
$m_{G_2}^k$.

We will prove by induction over $n$ that there exists $G\in\G_n$ so that

\begin{equation}\label{induction}
m_{G_1}^k\cdot m_{G_2}^k = m_G^k
\end{equation}
for every $k$, for a certain symbol $m_G$ of type $G$.
If we accept for a moment (\ref{induction}), then (\ref{product}) becomes

$$\sum_km_{G_1}^k\cdot m_{G_2}^k = \sum_k m_G^k = m_G$$
which would complete our proof. 

Denote by $u$ the root of $G_1$, by $v$ the root of $G_2$ and by $w$ the root of $G$.
Denote also by $u_1,...,u_\#$ 
the sons of $u$, by $v_1,...,v_{\widetilde{\#}}$ the sons of $v$ and by
$w_1,...,w_{\widetilde{\widetilde{\#}}}$ the sons of $G$. We will in fact prove that
the rooted tree $G$ we are looking for has also the following $refinement$ $property$ 
(with respect to $G_1$ and $G_2$) which says that the sets of indices
$\I_{u_i}$ and $\I_{v_j}$ can each be written as a disjoint union of various sets of indices of
type $\I_{w_l}$ \footnote{ And we add this ``refinement property'' to the induction hypothesis.}.

It remains to prove the inductive claim.

Clearly, the case $n=2$ is completely obvious, since there is only one type of rooted trees
in $\G_2$. Assume now that our statement holds for indices up to $n-1$ and we will prove it for
$n$.

\underline{Case 1}

Let us first assume that we are in the easier case when there exist $r_{u_{i_0}}$ for
$1\leq i_0\leq \#-1$ and $r_{v_{j_0}}$ for $1\leq j_0\leq \widetilde{\#} - 1$ so that

$$r_{u_{i_0}} = r_{v_{j_0}} := l_0.$$
In this case, define $G'_1$ and $G'_2$ to be the minimal subtrees of $G_1$ and $G_2$ respectively
whose leaves are only those indexed from $1$ to $l_0$.
Similarly, define $G''_1$ and $G''_2$ to be the minimal subtrees
whose leaves are those indexed from $l_0 +1$ to $n$. It is not difficult to remark that the roots
 of $G'_1$ and $G''_1$ are either equal to the root
$u$ of $G_1$ or they are sons of $u$. Similarly, the roots of $G''_1$ and $G''_2$ are either 
equal to the root $v$ of $G_2$, or they are sons of $v$.
As a consequence of this fact, we are facing several subcases.

\underline{Case $1_a$} Assume that $G'_1$ and $G''_1$ have the same root with $G_1$ and that 
$G'_2$ and $G''_2$ have the same root with $G_2$.

Then, for a fixed $k$, one can write (using (\ref{mG}))

$$m_{G_1}^k ((\xi_l)_{l=1}^n) = 
\sum _{\n_1} C^1_{\n_1} m_{G_1, \n_1}^k( (\xi_l)_{l=1}^n)
$$
and similarly

$$m_{G_2}^k ((\xi_l)_{l=1}^n) = 
\sum _{\n_2} C^2_{\n_2} m_{G_2, \n_2}^k( (\xi_l)_{l=1}^n).$$

Using (\ref{9}) (in the case when the vertex is either the root of $G_1$ or $G_2$)
and (\ref{mG}) one can further split $m_{G_1, \n_1}^k$ and $m_{G_2, \n_2}^k$ naturally as

$$m_{G_1, \n_1}^k = \sum_{Q_u} m_{G_1, \n_1}^{k, Q_u}$$ and

$$m_{G_2, \n_2}^k = \sum_{Q_v} m_{G_2, \n_2}^{k, Q_v}.$$
Fix now $\n_1, \n_2, Q_u, Q_v$ and consider the corresponding product term

\begin{equation}\label{productt}
m_{G_1, \n_1}^{k, Q_u}\cdot m_{G_2, \n_2}^{k, Q_v}.
\end{equation}
Observe that for every fixed $Q_u$ there are at most $O(1)$ quasi-cubes $Q_v$ for which the above product is not identically equal to zero. One can then rewrite
$m_{G_1, \n_1}^{k, Q_u}$ and $m_{G_2, \n_2}^{k, Q_v}$ more explicitly as 
\footnote{$m_{G'_1, \n_1}^{k}$ is obtained from the formula (\ref{mG}) corresponding to
$m_{G_1, \n_1}^k$, by taking only the sums and products associated to the vertices of $G_1'$ and 
with the extra ``twist'' given by the fact
that the symbol associated to the root $u$ is of the form 
$\sum_{\widetilde{Q}_u}[\Phi^{\n_1}_{\widetilde{Q}_u^1, 1} (\sum_{l\in \I_{u_1}} \xi_l ) ... 
\Phi^{\n_1}_{\widetilde{Q}_u^{i_0}, i_0} (\sum_{l\in \I_{u_{i_0}}} \xi_l )]$ and has the 
property that it is identically equal to $1$ on the support of
$\sum_{Q_u}[\Phi^{\n_1}_{Q_u^1, 1} (\sum_{l\in \I_{u_1}} \xi_l ) ... 
\Phi^{\n_1}_{Q_u^{i_0}, i_0} (\sum_{l\in \I_{u_{i_0}}} \xi_l )]$. Similarly, one defines
$m_{G''_1, \n_1}$, $m_{G'_2, \n_2}$ and $m_{G''_2, \n_2}$. }

\begin{equation}
m_{G_1, \n_1}^{k, Q_u} = 
\left[\Phi^{\n_1}_{Q_u^1, 1} (\sum_{l\in \I_{u_1}} \xi_l ) ... \Phi^{\n_1}_{Q_u^{i_0}, i_0} (\sum_{l\in \I_{u_{i_0}}} \xi_l )\right]\cdot
\left[\Phi^{\n_1}_{Q_u^{i_0 + 1}, i_0 + 1} (\sum_{l\in \I_{u_{i_0 + 1}}} \xi_l ) ... \Phi^{\n_1}_{Q_u^{\#}, \#} (\sum_{l\in \I_{u_{\#}}} \xi_l )\right]\cdot
\end{equation}

$$
m_{G'_1, \n_1}^{k}( (\xi_l)_{l=1}^{l_0}) \cdot m_{G''_1, \n_1}^{k}( (\xi_l)_{l= l_0 + 1}^{n})
$$
and 

\begin{equation}
m_{G_2, \n_2}^{k, Q_v} = 
\left[\Phi^{\n_2}_{Q_v^1, 1} (\sum_{l\in \I_{v_1}} \xi_l ) ... \Phi^{\n_2}_{Q_v^{j_0}, j_0} (\sum_{l\in \I_{u_{j_0}}} \xi_l )\right]\cdot
\left[\Phi^{\n_2}_{Q_v^{j_0 + 1}, j_0 + 1} (\sum_{l\in \I_{v_{j_0 + 1}}} \xi_l ) ... \Phi^{\n_2}_{Q_v^{\widetilde{\#}}, \widetilde{\#}} 
(\sum_{l\in \I_{v_{\widetilde{\#}}}} \xi_l )\right]\cdot
\end{equation}

$$
m_{G'_2, \n_2}^{k}( (\xi_l)_{l=1}^{l_0}) \cdot m_{G''_2, \n_2}^{k}( (\xi_l)_{l= l_0 + 1}^{n}).
$$
In particular, the product (\ref{productt}) can be written as

\begin{equation}\label{productt1}
\left[\Phi^{\n_1}_{Q_u^1, 1} (\sum_{l\in \I_{u_1}} \xi_l ) ... \Phi^{\n_1}_{Q_u^{i_0}, i_0} (\sum_{l\in \I_{u_{i_0}}} \xi_l )\right]\cdot
\left[\Phi^{\n_2}_{Q_v^1, 1} (\sum_{l\in \I_{v_1}} \xi_l ) ... \Phi^{\n_2}_{Q_v^{j_0}, j_0} (\sum_{l\in \I_{u_{j_0}}} \xi_l )\right]\cdot
\end{equation}

$$
m_{G'_1, \n_1}^{k}( (\xi_l)_{l=1}^{l_0})\cdot
m_{G'_2, \n_2}^{k}( (\xi_l)_{l=1}^{l_0})\cdot
$$

$$
\left[\Phi^{\n_1}_{Q_u^{i_0 + 1}, i_0 + 1} (\sum_{l\in \I_{u_{i_0 + 1}}} \xi_l ) ... \Phi^{\n_1}_{Q_u^{\#}, \#} (\sum_{l\in \I_{u_{\#}}} \xi_l )\right]\cdot
\left[\Phi^{\n_2}_{Q_v^{j_0 + 1}, j_0 + 1} (\sum_{l\in \I_{v_{j_0 + 1}}} \xi_l ) ... \Phi^{\n_2}_{Q_v^{\widetilde{\#}}, \widetilde{\#}} 
(\sum_{l\in \I_{v_{\widetilde{\#}}}} \xi_l )\right]\cdot
$$

$$
m_{G''_1, \n_1}^{k}( (\xi_l)_{l= l_0 + 1}^{n})\cdot
m_{G''_2, \n_2}^{k}( (\xi_l)_{l= l_0 + 1}^{n}).
$$
By using the induction hypothesis, there exist two trees $G'$ and $G''$ and symbols associated to them with the property that

$$m_{G'_1, \n_1}^{k}\cdot m_{G'_2, \n_2}^{k} = m^k_{G', \n_1, \n_2}$$
and

$$m_{G''_1, \n_1}^{k}\cdot m_{G''_2, \n_2}^{k} = m^k_{G'', \n_1, \n_2}$$
for every $k\in\Z$. Denote now by $G$ the rooted tree obtained by concatenating $G'$ and $G''$ together and let $w$ denote the root of $G$.

Using these, the expression (\ref{productt1}) becomes

\begin{equation}\label{productt2}
\sum_{Q'_w} \sum_{Q''_w}
\end{equation}

$$
\left[\Phi^{\n_1}_{Q_u^1, 1} (\sum_{l\in \I_{u_1}} \xi_l ) ... \Phi^{\n_1}_{Q_u^{i_0}, i_0} (\sum_{l\in \I_{u_{i_0}}} \xi_l )\right]\cdot
\left[\Phi^{\n_2}_{Q_v^1, 1} (\sum_{l\in \I_{v_1}} \xi_l ) ... \Phi^{\n_2}_{Q_v^{j_0}, j_0} (\sum_{l\in \I_{u_{j_0}}} \xi_l )\right]\cdot
m^{k, Q'_w}_{G', \n_1 , \n_2}\cdot
$$

$$
\left[\Phi^{\n_1}_{Q_u^{i_0 + 1}, i_0 + 1} (\sum_{l\in \I_{u_{i_0 + 1}}} \xi_l ) ... \Phi^{\n_1}_{Q_u^{\#}, \#} (\sum_{l\in \I_{u_{\#}}} \xi_l )\right]\cdot
\left[\Phi^{\n_2}_{Q_v^{j_0 + 1}, j_0 + 1} (\sum_{l\in \I_{v_{j_0 + 1}}} \xi_l ) ... \Phi^{\n_2}_{Q_v^{\widetilde{\#}}, \widetilde{\#}} 
(\sum_{l\in \I_{v_{\widetilde{\#}}}} \xi_l )\right]\cdot
m^{k, Q''_w}_{G'', \n_1 , \n_2}.
$$
Observe now as before that for our fixed $Q_u$ and $Q_v$ there exist at most $O(1)$ quasi-cubes $Q'_w$ and $Q''_w$ for which the previous expression
(\ref{productt2}) does not vanish.

Also, by using the fact that $G'$ and $G''$ have the $refinement$ $property$ ( with respect to $(G'_1, G'_2)$ and $(G''_1, G''_2)$ respectively )
one can successively apply the previous ``fixing'' Lemma \ref{fixing} and rewrite (\ref{productt2}) in the form

$$\sum_{\n} C_{\n} \widetilde{m}^{k, Q'_w}_{G', \n_1 , \n_2, \n}\cdot 
\widetilde{m}^{k, Q''_w}_{G'', \n_1 , \n_2, \n} := $$

$$\sum_{\n} C_{\n}  \widetilde{m}^{k, Q'_w 
\times Q''_w }_{G, \n_1, \n_2, \n}.$$
Summing now over all the previously fixed parameters $\n_1, \n_2, Q_u, Q_v$ means summing 
over $\n_1, \n_2$ and 
$Q'_w \times Q''_w$ and as a consequence, our original product 
$m_{G_1}^k\cdot m_{G_2}^k$ can be clearly written
as $m^k_G$ for a certain multiplier $m_G$ of type $G$.

\underline{Case $1_b$} Assume that the roots of $G'_1, G'_2$ and the roots of $G''_1, G''_2$ 
are all sons of $u$ and $v$ respectively.

It is then not difficult to see that this can only happen if both $u$ and $v$ have precisely 
two sons $u_1, u_2$ and $v_1, v_2$. And this means that
the root of $G'_1$ is $u_1$, the root of $G''_1$ is $u_2$, the root of $G'_2$ is $v_1$ and the 
root of $G''_2$ is $v_2$. Using the same notations as before,
this time one can write
\footnote{ This time by $m_{G'_1, \n_1}$ we denote the symbol obtained from the formula (\ref{mG}) corresponding to $m_{G_1, \n_1}$, by taking
only the sums and products associated to the vertices of $G'_1$; and all the other symbols are defined similarly.}

$$
m_{G_1, \n_1}^{k, Q_u} = 
\left[\Phi^{\n_1}_{Q_u^1, 1} (\sum_{l = 1}^{l_0} \xi_l )\right]\cdot
\left[\Phi^{\n_1}_{Q_u^2, 2} ( \sum_{l = l_0 + 1}^{n} \xi_l )\right]\cdot
\left[ \sum_{k'_1 << k} m^{k'_1}_{G'_1, \n_1} ((\xi_l)_{l=1}^{l_0})      \right]\cdot
\left[ \sum_{k''_1 << k } m^{k''_1}_{G''_1, \n_1} ((\xi_l)_{l=l_0 + 1}^{n})      \right]
$$
and 

$$
m_{G_2, \n_2}^{k, Q_v} = 
\left[\Phi^{\n_2}_{Q_v^1, 1} (\sum_{l = 1}^{l_0} \xi_l )\right]\cdot
\left[\Phi^{\n_2}_{Q_v^2, 2} ( \sum_{l = l_0 + 1}^{n} \xi_l )\right]\cdot
\left[ \sum_{k'_2 << k} m^{k'_2}_{G'_2, \n_2} ((\xi_l)_{l=1}^{l_0})      \right]\cdot
\left[ \sum_{k''_2 << k } m^{k''_2}_{G''_2, \n_2} ((\xi_l)_{l=l_0 + 1}^{n})      \right].
$$
In this case, it is very easy to remark that when one considers the product $m^{k, Q_u}_{G_1, \n_1}\cdot m^{k, Q_v}_{G_2, \n_2}$
the terms of the previous two expressions match each other perfectly (there is no need of the ``fixing'' lemma this time) and the induction hypothesis
can be applied twice, solving the problem.

\underline{Case $1_c$} Assume that we are in a ``mixed case'' when the roots of $G'_1, G'_2$ are 
sons of $u$ and $v$ respectively and the roots of
$G''_1, G''_2$ coincide with $u$ and $v$ respectively. It is not difficult to see that this case 
can be solved by combining the arguments used for the previous two
cases.

Finally, we are left with

\underline{ Case $1_d$} Assume that we are in a ``skewed situation'' now, 
when for instance the root of $G'_1$ is a son of $u$, the top of $G''_1$ is $u$,
the root of $G'_2$ is $v$ and the root of $G''_2$ is a son of $v$. It is easy to see that in 
fact the root of $G'_1$ is $u_1$ while the root of $G''_2$ is 
$v_{\widetilde{\#}}$. As a consequence, the two multipliers $m_{G_1, \n_1}^{k, Q_u}$ and $m_{G_2, \n_2}^{k, Q_v}$ become

$$
m_{G_1, \n_1}^{k, Q_u} = 
\left[\Phi^{\n_1}_{Q_u^1, 1} (\sum_{l = 1}^{l_0} \xi_l )\right]\cdot
\left[\Phi^{\n_1}_{Q_u^2, 2} ( \sum_{l \in \I_{u_2} }\xi_l ) ... \Phi^{\n_1}_{Q_u^{\#}, \#} (\sum_{l\in \I_{u_{\#}}} \xi_l ) \right]\cdot
\left[ \sum_{k'_1 << k} m^{k'_1}_{G'_1, \n_1} ((\xi_l)_{l=1}^{l_0})      \right]\cdot
\left[ m^{k}_{G''_1, \n_1} ((\xi_l)_{l=l_0 + 1}^{n})      \right]
$$
and 

$$
m_{G_2, \n_2}^{k, Q_v} = 
\left[\Phi^{\n_2}_{Q_v^1, 1} (\sum_{l\in \I_{v_{1}}} \xi_l ) ... \Phi^{\n_2}_{Q_v^{\widetilde{\#} - 1}, \widetilde{\#} - 1} 
(\sum_{l\in \I_{v_{\widetilde{\#} - 1}}} \xi_l )\right]\cdot
\left[\Phi^{\n_2}_{Q_v^{\widetilde{\#}}, \widetilde{\#}} 
(\sum_{l = l_0 + 1}^n \xi_l )\right] \cdot
\left[ m^{k}_{G'_2, \n_n} ((\xi_l)_{l=1}^{l_0})\right]\cdot
\left[ \sum_{k''_2 << k } m^{k''_2}_{G''_2, \n_2} ((\xi_l)_{l=l_0 + 1}^{n})\right].
$$
Then, one observes that in order for the product between $m_{G_1, \n_1}^{k, Q_u}$ and $m_{G_2, \n_2}^{k, Q_v}$ to be nonzero, one must have
$k'_1 \sim k$ and also $k''_2 \sim k$. But then, the whole case becomes similar to the previously studied Case $1_a$.
\footnote{ Of course, all the other possible ``skewed'' cases can be treated similarly.}

\underline{Case 2}

Assume now that $r_{u_i} \neq r_{v_j}$ for any
$1\leq i \leq \#-1$ and $1\leq j\leq \widetilde{\#}-1$. In this case we claim that
one can construct two other rooted trees $Ret(G_1)$ and $Ret(G_2)$ (a ``retract'' of
$G_1$ and a ``retract'' of $G_2$) having the property that the pair
$(Ret(G_1), Ret(G_2))$ satisfies the condition of Case 1 and also such that

\begin{equation}\label{retract}
m_{G_1}\cdot m_{G_2} = m_{Ret(G_1)}\cdot m_{Ret(G_2)}
\end{equation}
for certain symbols $m_{Ret(G_1)}$ and $m_{Ret(G_2)}$ of type $Ret(G_1)$ and 
$Ret(G_2)$ respectively. Clearly, (\ref{retract}) allows us to reduce the general Case 2
to the Case 1 discussed before.

The idea of proving the claim is very natural. In order for $m_{G_1}\cdot m_{G_2} $
to be non-zero at a given point $(\xi_l)_{l=1}^n$ one must have 

$$|I_{r_{u_1}}| \sim ... \sim |I_{r_{u_{\#-1}}}| >> |I_l|$$
for any $l\neq r_{u_1},..., r_{u_{\#-1}}$
and also

$$|I_{r_{v_1}}| \sim ... \sim |I_{r_{v_{\widetilde{\#}-1}}}| >> |I_l|$$
for any other $l\neq r_{v_1},..., r_{v_{\widetilde{\#}-1}}$. In particular, one has to have

\begin{equation}\label{!}
|I_{r_{u_1}}| \sim ... \sim |I_{r_{u_{\#-1}}}| \sim
|I_{r_{v_1}}| \sim ... \sim |I_{r_{u_{\widetilde{\#}-1}}}| >> |I_l|
\end{equation}
for any other $l$ different than all these indices. Intuitively, it is clear that all
these conditions should induce many equivalences between the summation indices which appear
in the definitions of $m_{G_1}$ and $m_{G_2}$ in (\ref{mG}) and this should allow one
to simplify the trees.

Denote by $S_1$ the set

\begin{equation}
S_1:= \{ r_{u_1},...,r_{u_{\#-1}} \}
\end{equation}
and by $S_2$ the set

\begin{equation}
S_1:= \{ r_{v_1},...,r_{v_{\widetilde{\#}-1}} \}.
\end{equation}
We will describe the construction of $Ret(G_1)$, the one for $Ret(G_2)$ being similar.

The root of $Ret(G_1)$ will be the same as the one of $G_1$ itself, but the sons will be 
different. They are selected from the former vertices of $G_1$ as follows.

First, we look at all the sons of $u$, namely $u_1,...,u_{\#}$ and select those $u_j$ having
the property that the sets $\widetilde{\I_{u_j}}:= [l_{u_j}, r_{u_j})\cap \N$
do not contain any element from $S_2$, i.e. $\widetilde{\I_{u_j}}\cap S_2 =\emptyset$.

At the second step we are left with those not selected sons of $u$. Consider the sons of all of 
them. If $v$ is such a new son (therefore a grandson of $u$) having the property that
$\widetilde{\I_v}\cap S_2 = \emptyset$ then we select it, if not, we do not. 
Clearly, this selection
procedure ends after a finite number of steps producing the selected vertices $V_{G_1}^{select}$.

As we already mentioned, the vertices in this set will be (by definition) the sons of the top
of $Ret(G_1)$ and then the rest of the tree is constructed by simply ``copying and pasting''
the subtrees of $G_1$ whose roots are these selected vertices in $V_{G_1}^{select}$, see
Figure \ref{fig6} for a particlular case.

\begin{figure}[htbp]\centering
\psfig{figure=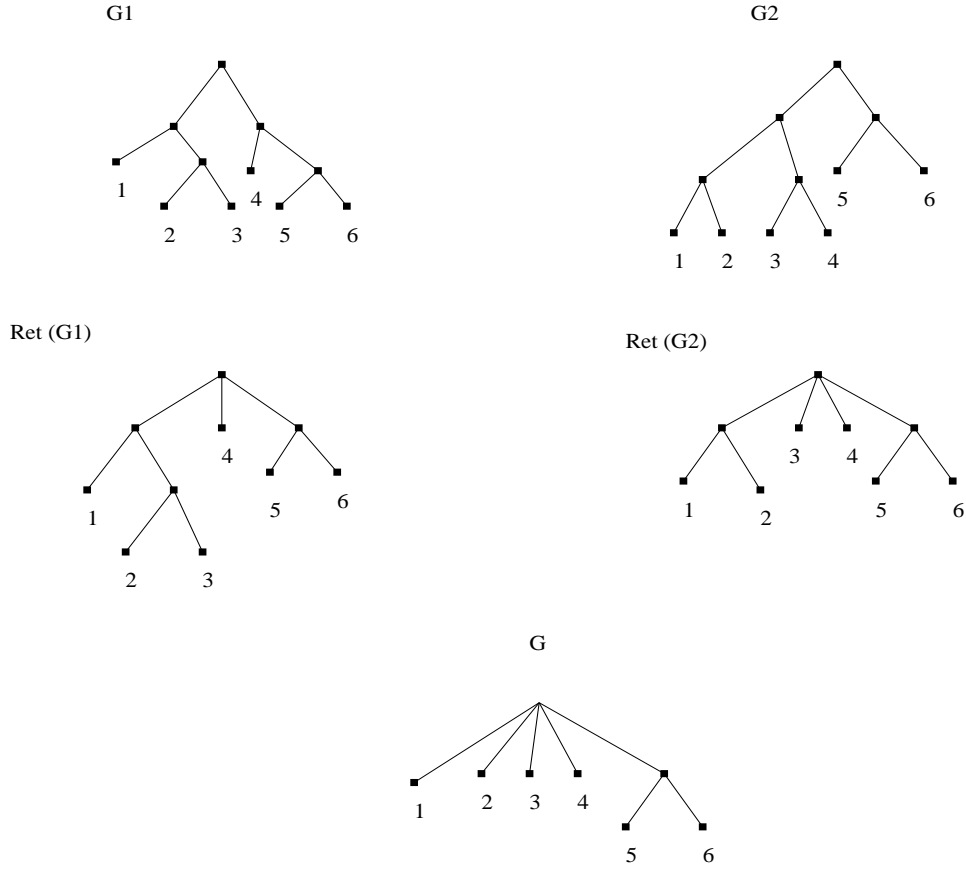, height=4.5in, width=5in}
\caption{Retracts}
\label{fig6}
\end{figure}

It is not difficult to observe that by construction, one has that the sets
$(\I_v)_{v\in V_{G_1}}$ form a partition of $\{1,...,n\}$ and also that

$$S_1\cup S_2 \subseteq \{ r_v : v\in V_{G_1}^{select}, r_v \neq n \}$$
and a similar inclusion holds for $G_2$. This shows that $Ret(G_1)$ and $Ret(G_2)$ are indeed
as in Case 1.

If in the definition of $m_{G_1}$ one takes into account the new constraints induced by
the relations (\ref{!}), one obtains a new formula which is clearly a multiplier of type
$Ret(G_1)$ which we call $m_{Ret(G_1)}$. Similarly, one defines $m_{Ret(G_2)}$. There is only
one technical issue left to be solved. If one looks carefully at the products corresponding
to the top of $Ret(G_1)$ (in the definition (\ref{mG})), one sees not only the standard
expressions of the form

$$\prod_{v\in V_{G_1}^{select}} \Phi_v^k (\sum_{l\in\I_v} \xi_l)$$
but also extra terms of type

$$\Phi_w^k(\sum_{l\in\I_w} \xi_l)$$
coming from various other vertices $w$. However, it is not difficult to see that by construction,
one always has that the sets $\I_w$ can be written as unions of $\I_v$ for various 
$v\in V_{G_1}^{select}$ and as a consequence, the issue is easily solved by repeatedly
applying the previous ``fixing'' Lemma \ref{fixing}.

\end{proof}

All of these show that indeed our initial symbol $\chi_{\xi_1 < ... < \xi_n}$ can be written as

$$\chi_{\xi_1 < ... < \xi_n} = \sum_G m_G$$
for various multipliers $m_G$ of type $G$. The reason for which we like this decomposition
will be clearer in the next section.

\begin{figure}[htbp]\centering
\psfig{figure=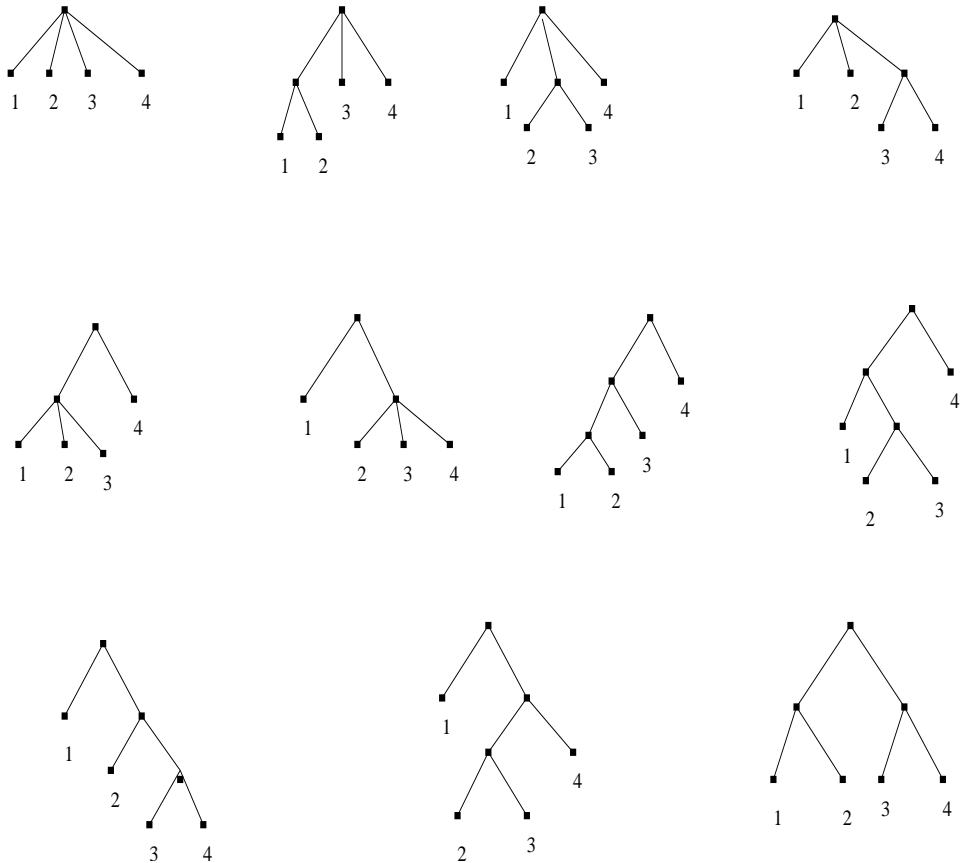, height=4.5in, width=5in}
\caption{The rooted trees of $T_4$.}
\label{fig5}
\end{figure}

\section{Models and the main reduction}

In this section we introduce some discrete model operators deeply related to our original 
operators $T_n$ in (\ref{multi-est}) and show that in order to prove our main
Theorem \ref{mainth} it is enough to prove it for these model operators.

We first need to recall certain definitions from earlier papers \cite{mtt-1} - \cite{mtt-11}.

\begin{definition}
Let $d\geq 1$. A tile is a rectangle $P = I_P \times \omega_P$ of area one where $I_P$ is a 
dyadic interval and $\omega_P$ is a shifted dyadic interval. A vector tile of dimension $d$ is
a $d$-tuple $P = (P_1,...,P_d)$ where each $P_i$ for $1\leq i\leq d$ is a tile and with
$I_{P_1} =...= I_{P_d} ( := I_P )$. We will sometimes refer to the tiles in the $i$ th
position as being $i$-tiles. The intervals $I_P$ are called the time intervals of the tiles
$P$ and the intervals $\omega_P$ are called the frequency intervals of the tiles $P$.
Similarly, the quasi-cube $\omega_{P}:= \omega_{P_1}\times ... \times \omega_{P_d}$
is called the frequency cube of the vector tile $P$.
\end{definition}

\begin{definition}
A set $\vec{\P}$ of vector tiles of dimension $d$ is said to be sparse if and only if the 
collection $\{\omega_{P} : P\in \vec{\P} \}$ of quasi-cubes is sparse.

\end{definition}

\begin{definition}
Let $P$ and $P'$ be tiles. We write $P' < P$ if $I_{P'}\subsetneqq I_P$ and
$3\omega_P \subseteq 3\omega_{P'}$, and $P'\leq P$ if $P' < P$ or $P'=P$. We also write
$P'\lesssim P$ if $I_{P'}\subseteq I_P$ and $C\omega_P \subseteq C\omega_{P'}$ where
$C>0$ is a large fixed constant. Finally, we write $P'\lesssim' P$ if
$P'\lesssim P$ but $P'\nleq P$.
\end{definition}

\begin{definition}
A collection $\vec{\P}$ of vector tiles of dimension $d$ is said to have rank $1$ if one has 
the following properties for all $P, P'\in \vec{\P}$:

(1) If $P\neq P'$ then $P_j\neq P'_j$ for all $1\leq j\leq d$.

(2) If $P'_j\leq P_j$ for some $1\leq j\leq d$, then $P'_i\lesssim P_i$ for all
$1\leq i\leq d$.

(3) If in addition to $P'_j\leq P_j$ one also assumes that $C|I_{P'}| < |I_{P}|$
(for a fixed large constant $C>0$) then we have $P'_i\lesssim' P_i$ for every $i\neq j$.
\end{definition}

Finally, we also recall

\begin{definition}
Let $P$ be a tile. A wave packet adapted to $P$ is a function $\Phi_P$ which has Fourier support
inside the interval $\frac{9}{10} \omega_P$ and satisfies the estimate

\begin{equation}
|\Phi_P(x)| \lesssim |I_P|^{-1/2} (1+ \frac{|x-c_{I_P}|}{|I_P|})^{-m}
\end{equation}
(where $c_{I_P}$ is the center of the interval $I_P$) for all $m>0$, with the implicit
constants depending on $m$.
\end{definition}

Thus, heuristically, $\Phi_P$ is $L^2$-normalized and supported in $P$.

Sometimes we will also use the notation $\widetilde{\chi}_I$ for the bump function defined by

\begin{equation}
\widetilde{\chi}_I(x) := (1 + \frac{\dist (x, I)}{|I|})^{- 10}.
\end{equation}

Having all these definitions at our disposal, we can start the description of our model
operators. They will be associated to arbitrary rooted trees $G\in \G_n$. We will define them
inductively, with respect to their height $h$.

Let $G$ be a rooted tree of height $1$. Then, a discrete model operator of type
$G$, is an $n$-linear operator of the form

\begin{equation}\label{discrete1}
T^G(f_1,...,f_n):= \int_0^1
\sum_{ P \in\vec{\P}_G}
\frac{1}{|I_P|^{\frac{n-1}{2}}}
\langle f_1, \Phi_{P_1}^{1, \alpha}\rangle...
\langle f_n, \Phi_{P_n}^{n, \alpha}\rangle
\Phi_{P_{n+1}}^{n+1, \alpha} d \alpha
\end{equation}
where $\vec{\P}_G$ is an arbitrary finite collection of rank $1$ of vector tiles of dimension
$n+1$ and $\Phi_{P_i}^{i, \alpha}$ are wave packets adapted to the tiles $P_i$ 
for $i=1,...,n+1$, uniformly for $\alpha \in [0, 1]$ .

Similarly, one defines model operators associated to rooted trees of height $1$ having an 
arbitrary number of leaves. Then, if $I$ is a dyadic interval, we define
$T^G_{|I|}$ to be the multi-linear operator defined by

\begin{equation}
T^G_{|I|}(f_1,...,f_n):= \int_0^1
\sum_{P\in\vec{\P}_G : |I_{\vec{P}}|\geq |I|   }
\frac{1}{|I_P|^{\frac{n-1}{2}}}
\langle f_1, \Phi_{P_1}^{1, \alpha}\rangle...
\langle f_n, \Phi_{P_n}^{n, \alpha}\rangle
\Phi_{P_{n+1}}^{n+1, \alpha} d \alpha .
\end{equation}
Suppose now that we know how to define model operators $T^G$ and $T^G_{|I|}$ associated
to rooted trees of height $h-1$ and an arbitrary number of leaves and we will describe
the definition of a model operator associated to a rooted tree of height $h$.

Fix $G$ of height $h$ having an arbitrary number of leaves $L$. Assume that the root of it
has precisely $\#$ sons. Denote by $G_1,...,G_{\#}$ the subtrees of $G$ whose roots
are all these sons. Clearly, all of these trees have height at most $h-1$ and 
by the induction hypothesis we know how to define the model operators
$T^{G_i}((f_l)_{l\in\I_i})$ and also $T^{G_i}_{|I|}((f_l)_{l\in\I_i})$ for 
$1\leq i\leq \#$. By $\I_i$ we simply denoted the set of indices corresponding
to the leaves which are descendents of the root of $G_i$. We also adopt the convention that
if a certain $G_j$ is a leave then $T^G = T^G_{|I|} = id$. Then, by a model operator
of type $G$ we mean an expression of the form

\begin{equation}\label{discreteh}
T^G(f_1,...,f_L):= \int_0^1
\sum_{P\in\vec{\P}_G}
\frac{1}{|I_P|^{\frac{\#-1}{2}}}
\langle T^{G_1}_{|I|}((f_l)_{l\in\I_1}), \Phi_{P_1}^{1, \alpha}\rangle...
\langle T^{G_{\#}}_{|I|}((f_l)_{l\in\I_{\#}}), \Phi_{P_{\#}}^{\#, \alpha}\rangle
\Phi_{P_{\#+1}}^{\#+1, \alpha} d \alpha
\end{equation}
where as before $\vec{\P}_G$ is a finite collection of rank $1$ of vector tiles of dimension
$\#+1$ and $\Phi_{P_i}^{i, \alpha}$ are wave packets adapted to $P_i$ for every $1\leq i\leq \#$,
 uniformly in $\alpha$. Finally, one defines $T^G_{|I|}$ similarly by

\begin{equation}
T^G_{|I|}(f_1,...,f_L):= \int_0^1
\sum_{P\in\vec{\P}_G : |I_P| \geq |I| }
\frac{1}{|I_P|^{\frac{\#-1}{2}}}
\langle T^{G_1}_{|I|}((f_l)_{l\in\I_1}), \Phi_{P_1}^{1, \alpha}\rangle...
\langle T^{G_{\#}}_{|I|}((f_l)_{l\in\I_{\#}}), \Phi_{P_{\#}}^{\#, \alpha}\rangle
\Phi_{P_{\#+1}}^{\#+1, \alpha} d \alpha
\end{equation}

We now claim that our main Theorem \ref{mainth} can be in fact reduced to the following

\begin{theorem}\label{mainthd}
Let $G\in\G_n$, let $\epsilon > 0$ be a small number and let
$\alpha_1,...,\alpha_n\in (\frac{1}{2}-\epsilon, \frac{1}{2} + \epsilon)$.
Then, for every measurable sets $F_1,...,F_n$ of finite measure, 
every functions $f_i$ having the property that $f_i\leq \chi_{F_i}$ for $1\leq i\leq n$,
and every $F$ with $|F|\sim 1$, 
there exists a subset $F'$ of $F$ with $|F'| \sim 1$ such that the following inequalities hold

\begin{equation}\label{restricted}
|\int_{\R} T^G(f_1,...,f_n)(x)\chi_{F'}(x) dx | \lesssim
|F_1|^{\alpha_1}...|F_n|^{\alpha_n}
\end{equation}
where the
implicit constants can be chosen to be independent on all the cardinalities of the finite sets
of vector-tiles which implicitly appear in the definition of $T^G$ and also independent
on the various wave packets considered. 
\end{theorem}

To see why this claim is true, fix $G\in \G_n$ and pick $m_G$ a multiplier of type $G$
as defined in the previous section. We will see in what follows that the corresponding
multi-linear operator $T_{m_G}$ can in fact be written as an weighted average of model operators
of type $T^G$. Since we will prove this inductively, let us first assume that $G$ is a rooted
tree of height $1$.

In particular, $m_G$ is of the form

$$m_G(\xi_1,...,\xi_n) =
\sum_{\n} C_{\n} \sum_{Q}
\Phi_{Q_1, \n, 1}(\xi_1)...
\Phi_{Q_n, \n, n}(\xi_n)
$$
where as usual $(C_{\n})_{\n}$ is a rapidly decreasing sequence and the inner summation runs over 
shifted dyadic quasi-cubes adapted to the region defined by the inequalities
$\xi_1<...<\xi_n$ and $|\xi_1-\xi_2| \sim ...\sim |\xi_{n-1} - \xi_n|$ in the sense
of Definition \ref{Ra}.

Fix $\n$ and consider only the inner sum. As before, it can be ``completed'' and rewritten
as

$$\sum_{Q}
\Phi_{Q_1, \n, 1}(\xi_1)...
\Phi_{Q_n, \n, n}(\xi_n)
\Phi_{Q_{n+1}, \n, n+1}(\xi_1+...+\xi_n)$$
for an appropriate choice of a wave packet $\Phi_{Q_{n+1}, \n, n+1}$. As a consequence, the expression

$$\int_{\R} T_{m_G} (f_1,...,f_n)(x) f_{n+1}(x) dx$$
becomes

\begin{equation}
\sum_{\n}C_{\n}
\sum_Q
\int_{\R^{n+1}}
\Phi_{Q_1,\n, 1}(\xi_1)...
\Phi_{Q_n,\n, n}(\xi_n)
\Phi_{Q_{n+1},\n, n+1}(\xi_1+...+\xi_n)\cdot
\end{equation}

$$
\cdot\widehat{f_1}(\xi_1)...
\widehat{f_n}(\xi_n) e^{2\pi i x(\xi_1+...+\xi_n)}
d \xi_1...
d \xi_n
f_{n+1}(x) 
d x =
$$

$$
\sum_{\n} C_{\n}
\sum_Q
\int_{\R^n}
\widehat{f_1}(\xi_1)\Phi_{Q_1,\n, 1}(\xi_1)
...
\widehat{f_n}(\xi_n)\Phi_{Q_n,\n, n}(\xi_n)\cdot
$$

$$
\cdot\Phi_{Q_{n+1},\n, n+1}(\xi_1+...+\xi_n)
\widehat{f_{n+1}}(-\xi_1-...-\xi_n) 
d \xi_1...
d \xi_n : = 
$$

$$
\sum_{\n }C_{\n}\sum_Q
\int_{\R^n}
\widehat{f_1}(\xi_1)\Phi_{Q_1,\n, 1}(\xi_1)
...
\widehat{f_n}(\xi_n)\Phi_{Q_n,\n, n}(\xi_n)\cdot
$$

$$
\cdot\widetilde{\Phi}_{Q_{n+1},\n, n+1}(-\xi_1-...-\xi_n)
\widehat{f_{n+1}}(-\xi_1-...-\xi_n) 
d \xi_1...
d \xi_n = 
$$

$$
\sum_{\n }C_{\n}
\sum_Q
\int_{\lambda_1+...+\lambda_{n+1} = 0}
(\widehat{f_1\ast \Phi^{\vee}_{Q_1,\n, 1}})(\lambda_1)...
(\widehat{f_n\ast \Phi^{\vee}_{Q_n,\n, n}})(\lambda_n)
(\widehat{f_{n+1}\ast \widetilde{\Phi}^{\vee}_{Q_{n+1},\n, n+1}})(\lambda_{n+1})
d \lambda =
$$

$$
\sum_{\n}C_{\n}\sum_Q
\int_{\R}
(f_1\ast \Phi_{Q_1,\n, 1}^{\vee})(x)...
(f_n\ast \Phi_{Q_n,\n, n}^{\vee})(x)
(f_{n+1}\ast \widetilde{\Phi}_{Q_{n+1},\n, n+1}^{\vee})(x) d x.
$$

Fix now $Q$ with $l(Q)\sim 2^k$ and look at the corresponding inner term in the previous expression. By making the change of variables
$x=2^{-k} y$ that term becomes

$$2^{-k}
\int_{\R}
(f_1\ast \Phi_{Q_1,\n, 1}^{\vee})(2^{-k} y )...
(f_n\ast \Phi_{Q_n,\n, n}^{\vee})(2^{-k} y )
(f_{n+1}\ast \widetilde{\Phi}_{Q_{n+1},\n, n+1}^{\vee})(2^{-k} y ) d y =
$$

$$2^{-k}\int_0^1
\sum_{l\in\Z}
(f_1\ast \Phi_{Q_1,\n, 1}^{\vee})(2^{-k} (l+ \alpha))...
(f_n\ast \Phi_{Q_n,\n, n}^{\vee})(2^{-k} (l + \alpha) )
(f_{n+1}\ast \widetilde{\Phi}_{Q_{n+1},\n, n+1}^{\vee})(2^{-k} (l + \alpha) ) d \alpha.
$$
Now, every generic term of the form $(f_j\ast \Phi_{j, \n, j}^{\vee})(2^{-k} ( l + \alpha ) )$ can be written as

$$
(f_j\ast \Phi_{j, \n, j}^{\vee})(2^{-k} ( l + \alpha ) ) =
\int_{\R}
f_j (y) \Phi_{j, \n, j}^{\vee} ( 2^{-k}l + 2^{-k}\alpha - y) d y :=
\int_{\R} f_j (y) \widetilde{\Phi_{j, \n, j}^{\vee} } ( y - 2^{-k}l - 2^{-k}\alpha ) d y =
$$

$$\int_{\R} f_j(y) 
\overline{\overline{\widetilde{\Phi_{j, \n, j}^{\vee} }} ( y - 2^{-k}l - 2^{-k}\alpha ) }
d y :=
2^{k/2}
\langle f_j, \Phi_{P_j}^{j, \n, \alpha} \rangle
$$
where $P_j$ is the tile $P_j:= 2^{-k}[l, l+1] \times Q_j$ for $j=1,...,n$ and $P_j:= 2^{-k}[l, l+1] \times (- Q_j)$
for $j= n + 1$. It is now clear that putting together all these calculations, $T_{m_G}$ can be written indeed as an weighted average of discrete model operators
of the type described before.

Assume now that this is true for every multiplier $m_G$ associated with trees of height $h-1$ and we want to demonstrate it for symbols $m_G$ associated to trees
of height $h$. Using (\ref{mG}) we can write any such a multiplier as

\begin{equation}
m_G = \sum_{\n}C_{\n}
\sum_Q
\sum_{k_1,...,k_{\#} << l(Q)}
m_{G_1, \n}^{k_1}((\xi_l)_{l\in\I_1})...
m_{G_{\#}, \n}^{k_{\#}}((\xi_l)_{l\in\I_{\#}})
\Phi_{Q_1, \n, 1}(\sum_{l\in\I_1}\xi_l)...
\Phi_{Q_{\#}, \n, \#}(\sum_{l\in\I_{\#}}\xi_l)
\end{equation}
where we implicitly assumed that $G$ has $\#$ sons. \footnote{ As usual, if $u$ is the root of $G$ and $u_1,..., u_{\#}$ are the sons of $u$, we denote
by $G_1,...,G_{\#}$ the subtrees of $G$ whose roots are all these sons. Then, for simplicity, we also denoted by $\I_j$ the previously defined sets of indices $\I_{u_j}$
for $j=1,...,\#$. }

Fix $\n$ and consider only the inner summation where we suppress for simplicity the dependence on $\n$. Fix also $Q$ so that $l(Q)\sim 2^k$ then fix
$k_1,..., k_{\#} << k$ and consider only the corresponding term determined by these fixed indices. As before, the $n+1$-linear form associated with the $n$-linear 
operator given by such a symbol is equal to (after the usual ``completion'')

$$\int_{\R}
\int_{\R^n}
m_{G_1}^{k_1}((\xi_l)_{l\in\I_1})...
m_{G_{\#}}^{k_{\#}}((\xi_l)_{l\in\I_{\#}})
\Phi_{Q_1, 1}(\sum_{l\in\I_1}\xi_l)...
\Phi_{Q_{\#},\#}(\sum_{l\in\I_{\#}}\xi_l)
\Phi_{Q_{\#+1}, \# + 1}(\sum_{l=1}^n \xi_l)\cdot
$$
$$
\cdot\widehat{f_1}(\xi_1)...
\widehat{f_n}(\xi_n) e^{2\pi i x(\xi_1+...+\xi_n)} d\xi_1...d\xi_n
f_{n+1}(x) d x =
$$

$$
\int_{\R^n}
m_{G_1}^{k_1}((\xi_l)_{l\in\I_1})...
m_{G_{\#}}^{k_{\#}}((\xi_l)_{l\in\I_{\#}})
\Phi_{Q_1, 1}(\sum_{l\in\I_1}\xi_l)...
\Phi_{Q_{\#}, \#}(\sum_{l\in\I_{\#}}\xi_l)
\Phi_{Q_{\#+1},\# + 1}(\sum_{l=1}^n \xi_l)\cdot
$$
$$
\cdot\widehat{f_1}(\xi_1)...
\widehat{f_n}(\xi_n) 
\widehat{f_{n+1}}(-\xi_1-...-\xi_n) d\xi_1... d\xi_n := 
$$

$$
\int_{\R^n}
m_{G_1}^{k_1}((\xi_l)_{l\in\I_1})...
m_{G_{\#}}^{k_{\#}}((\xi_l)_{l\in\I_{\#}})
\Phi_{Q_1, 1}(\sum_{l\in\I_1}\xi_l)...
\Phi_{Q_{\#}, \#}(\sum_{l\in\I_{\#}}\xi_l)
\widetilde{\Phi}_{Q_{\#+1}, \# + 1}(-\xi_1-...-\xi_n )\cdot
$$
$$
\cdot\widehat{f_1}(\xi_1)...
\widehat{f_n}(\xi_n) 
\widehat{f_{n+1}}(-\xi_1-...-\xi_n) d\xi_1... d\xi_n = 
$$

$$
\int_{\lambda_1 + ... + \lambda_{\#+1} = 0}
\widehat{T_{m_{G_1}^{k_1}}((f_l)_{l\in\I_1})}(\lambda_1)...
\widehat{T_{m_{G_{\#}}}^{k_{\#}}((f_l)_{l\in\I_{\#}})}(\lambda_{\#})\cdot
\Phi_{Q_1, 1}(\lambda_1)...
\Phi_{Q_{\#}, \#}(\lambda_{\#})\cdot
\widehat{f_{n+1}}(\lambda_{\#+1})\widetilde{\Phi}_{Q_{\#+1}, \#+1}(\lambda_{\#+1}) d\lambda =
$$

$$
\int_{\lambda_1 + ... + \lambda_{\#+1} = 0}
\widehat{(T_{m_{G_1}^{k_1}}((f_l)_{l\in\I_1})\ast \Phi^{\vee}_{Q_1, 1} )   }(\lambda_1)...
\widehat{(T_{m_{G_{\#}}^{k_{\#}}}((f_l)_{l\in\I_{\#}})\ast \Phi^{\vee}_{Q_{\#}, \#} )   }(\lambda_{\#})\cdot
\widehat{(f_{n+1} \ast \widetilde{\Phi}^{\vee}_{Q_{\#+1}, \#+1})}(\lambda_{\#+1}) d\lambda =
$$

$$
\int_{\R}
(T_{m_{G_1}^{k_1}}((f_l)_{l\in\I_1})\ast \Phi^{\vee}_{Q_1, 1} )(x)... 
(T_{m_{G_{\#}}^{k_{\#}}}((f_l)_{l\in\I_{\#}})\ast \Phi^{\vee}_{Q_{\#}, \#} )(x)\cdot
(f_{n+1} \ast \widetilde{\Phi}^{\vee}_{Q_{\#+1}, \#+1})(x) d x.
$$
At this point, one can discretize as usual once again in the $x$ variable, to obtain an average (over $\alpha$) of expressions of type

\begin{equation}
\sum_{\vec{P}}
\frac{1}{|I_{\vec{P}}|^{\frac{\#-1}{2}}}
\langle T_{m_{G_1}^{k_1}} ((f_l)_{l\in\I_1}), \Phi_{P_1}^{1,\alpha} \rangle ...
\langle T_{m_{G_{\#}}^{k_{\#}}} ((f_l)_{l\in\I_{\#}}), \Phi_{P_{\#}}^{\#,\alpha} \rangle \cdot
\langle f_{n+1}, \Phi_{P_{\#+1}}^{\#+1, \alpha} \rangle
\end{equation}
where the sum runs over vector tiles $\vec{P}$ so that $|\omega_{P_j}|\sim 2^k$ for every $j=1,...,\#+1$. Using now the induction hypothesis
and the fact that $k_1,...,k_{\#} << k$, it follows that indeed our operator $T_{m_G}$ can be written as an weighted average of discrete operators of the form
(\ref{discreteh}) as desired. 

More specifically, we have seen that every $T_{m_G}$ can be written as

$$T_{m_G} = \sum_{\n} D_{\n} T^G_{\n}$$
where $(D_{\n})_{\n}$ is a rapidly decreasing sequence indexed over a countable set, while
$T^G_{\n}$ is a discrete operator of the type (\ref{discreteh}). The only difference is that
in its case, the corresponding sum in (\ref{discreteh}) may be infinite. Using now
Theorem \ref{mainthd}, scaling invariance, the interpolation theory from \cite{mtt-1} and a 
standard limiting argument, it follows that each $T^G_{\n}$ is bounded
from $L^2\times ... \times L^2$ into $L^{2/n}$ with bounds which are independent on $\n$.
This shows that $T_{m_G}$ itself satisfies the same estimates, which proves our main
Theorem \ref{mainth}. It is therefore enough to prove Theorem \ref{mainthd} only.

\section{Prof of Theorem \ref{mainthd}}

First, we need to recall several definitions from some of our earlier work 
\cite{mtt-1} - \cite{mtt-8}.
We will also assume from now on that all our collections of vector-tiles are sparse.

\begin{definition}
Let $d\geq 3$ and $\vec{\P}$ be a collection of rank $1$ vector tiles of dimension $d$.
Let also $1\leq j\leq d$. A subcollection $T\subseteq \vec{\P}$ is said to be a
$j$-tree if and only if there exists a vector tile $P_T\in \vec{\P}$ such that

$$P_j \leq P_{T, j}$$
for all $P\in T$, where $P_{T, j}$ is the $j$ th component of $P_T$. The vector tile
$P_T$ is called the top of the tree. We write $I_T$ for $I_{P_T}$ and 
$\omega_{T, j}$ for $\omega_{P_{T, j}}$ respectively.
\end{definition}

Note also that a tree $T$ does not necessarily have to contain its top $P_T$.

\begin{definition}
Using the same notations in the previous definition, two trees $T$ and $T'$ are said to be
strongly $i$-disjoint ($1\leq i\leq d$) if and only if

(1) $P_i \neq P'_i$ for all $P\in T$ and $P'\in T'$.

(2) Whenever $P\in T$ and $P'\in T'$ are such that 
$2\omega_{P_i}\cap 2\omega_{P'_i}\neq \emptyset$ then one has 
$I_{P'}\cap I_T = \emptyset$ and similarly with $T$ and $T'$ reversed.
\end{definition}

Note also that if $T$ and $T'$ are strongly $i$-disjoint, then
$(I_P\times 2\omega_{P_i}) \cap (I_{P'}\times 2\omega_{P'_i}) = \emptyset$ for
all $P\in T$ and $P'\in T'$.

It is also important to point out that if $T$ is an $i$-tree, then for all
$P, P'\in T$ and $j\neq i$, either

$$\omega_{P_j} = \omega_{P'j}$$
or

$$2 \omega_{P_j} \cap 2 \omega_{P'j} = \emptyset.$$

It is now clear from the previous sections that in order to prove our main theorem
we need to be able to estimate generic expressions of the form

\begin{equation}\label{generic}
\sum_{P\in \vec{\P}}\frac{1}{|I_P|^{\frac{d-2}{2}}}
a^1_{P_1} ... a^d_{P_d}
\end{equation}
where $\vec{\P}$ is a finite collection of rank $1$ vector tiles of dimension $d$
and $(a^j_{P_j})_j$ are complex numbers of the form

$$a^j_{P_j} = 
\langle T^{G_j}_{|I_P|} ((f_l)_{l\in \I_j}), \Phi^j_{P_j} \rangle.
$$
The usual way to do this, is by using certain $sizes$ and $energies$ which are very
helpful to describe the local behaviour of expressions of type (\ref{generic}).
We recall first the following definition from \cite{mtt-1}.

\begin{definition}
Let $\vec{\P}$ be a rank $1$ collection of vector tiles of dimension $d$, $1\leq j\leq d$
and let also $(a^j_{P_j})_{P\in\vec{\P}}$ be a sequence of complex numbers. We define
the size of this sequence by

$$\size_j( (a^j_{P_j})_{P\in\vec{\P}}  ) :=
\sup_{T\subseteq \vec{\P}} 
(\frac{1}{|I_T|} \sum_{P\in T} |a^j_{P_j}|^2 )^{1/2}
$$
where $T$ ranges over all trees in $\vec{\P}$ which are $i$-trees for some $i\neq j$.
\end{definition}

The following John-Nirenberg type lemma is also very useful (see for instance
\cite{mtt-1} for a complete proof).

\begin{lemma}\label{jn}
Under the same hypothesis of the previous definition, one has

$$\size_j( (a^j_{P_j})_{P\in\vec{\P}}  )
\sim \sup_{T\in\vec{\P}}\frac{1}{|I_T|}
\| (\sum_{P\in T}\frac{|a^j_{P_j}|^2}{|I_P|}\chi_{I_P} )^{1/2} \|_{1, \infty}
$$
where again $T$ ranges over all trees in $\vec{\P}$ which are $i$-trees for some $i\neq j$.
\end{lemma}

The following lemma is also known (see for instance \cite{mtt-1} for a proof).

\begin{lemma}\label{sizeest}
Let $f$ be a measurable function. Then, one has

$$\size_j( (f, \Phi^j_{P_j})_{P\in\vec{\P}}  ) \lesssim \sup_{P\in\vec{\P}}
\frac{\int_{\R} |f| \widetilde{\chi}^M_{I_P}}{ |I_P|}
$$
for any positive real number $M$, where the implicit constant depends on $M$.
\end{lemma}

Let us also recall the following definition from \cite{mtt-5}.

\begin{definition}
Using the same notations as before, one defines the energy of the sequence
$(a^j_{P_j})_{P\in\vec{\P}}$ by

$$\energy_j((a^j_{P_j})_{P\in\vec{\P}}  ) :=
\sup_{\n\in\Z}\sup_{\T} 2^{\n} (\sum_{T\in\T} |I_T| )^{1/2}
$$
where $\T$ ranges over all collections of strongly $j$-disjoint trees
in $\vec{\P}$ ( which are $i$-trees for some $i\neq j$ )
such that

$$(\frac{1}{|I_T|} \sum_{P\in T} |a^j_{P_j}|^2 )^{1/2}\geq 2^{\n}$$
for all $T\in\T$ and also satisfying

$$(\frac{1}{|I_{T'}|} \sum_{P\in T'} |a^j_{P_j}|^2 )^{1/2}\leq 2^{\n+1}$$
for all sub-trees $T'\subseteq T\in\T$.
\end{definition}

It is also not difficult to observe the following lemma \cite{mtt-5}.

\begin{lemma}\label{CP}
For any sequence $(a^j_{P_j})_{P\in\vec{\P}}$ there exists a collection
$\T$ of strongly $j$-disjoint trees (which are $i$-trees for some $i\neq j$) and
complex numbers $c^j_{P_j}$ for all $P\in \cup_{T\in\T} T$ such that

$$\energy_j((a^j_{P_j})_{P\in\vec{\P}}  ) \sim
\sum_{T\in T}\sum_{P\in T} a^j_{P_j}\overline{c^j_{P_j}}
$$
and such that 

$$\sum_{P\in T'}
|c^j_{P_j} |^2 \lesssim \frac{|I_{T'}|}{\sum_{T\in\T} |I_T|}
$$
for all $T\in \T$ and all subtrees $T'\subseteq T$.
\end{lemma}

The following lemma is also well known (see for instance \cite{mtt-5}).

\begin{lemma}\label{energyest}
For any $f\in L^2(\R)$ one has

\begin{equation}
\energy_j( \langle f, \Phi^j_{P_j} \rangle _{P\in\vec{\P}} )
\lesssim \|f \|_2.
\end{equation}

\end{lemma}

The following lemma will also play an important role when estimating the 
energies of various
general sequences, later on. For a proof of it see \cite{mtt-5}.

\begin{lemma}\label{bessel}
Let $d_1, d_2\geq 3$ and $\vec{\P}, \vec{\Q}$ be rank $1$ collections of vector tiles
of dimensions $d_1$ and $d_2$ respectively. Let also $1\leq i\leq d_1$ and
$1\leq j\leq d_2$. Consider also two sequences of complex numbers 
$(c^i_{P_i})_P$ and $(c^j_{Q_j})_Q$ where $P$ runs inside a collection of strongly
$i$-disjoint trees which are $l$-trees for some $l\neq i$ and $Q$ runs inside a collection of
strongly $j$-disjoint trees which are $l$-trees for some $l\neq j$. Assume also that both of 
these sequences satisfy the conclusion of the previous Lemma \ref{CP}. Then, one has

\begin{equation}\label{ineqq}
\left|
\sum_{P, Q : |I_P| \leq |I_Q|}
c^i_{P_i}c^j_{Q_j} 
\langle \Phi^i_{P_i}, \Phi^j_{P_j} \rangle
\right| \lesssim 1.
\end{equation}

\end{lemma}

In addition to the above lemma, we need also the following result, which will play a crucial role later on.

\begin{lemma}\label{delicatebessel}
Assume that the sequences of complex numbers $(c^i_{P_i})_P$ and $(c^j_{Q_j})_Q$ are precisely
as in the previous Lemma \ref{bessel}. 
Assume in addition that there are two subsets $S_{\vec{\P}}$ and
$S_{\vec{\Q}}$ of the real line, so that $S_{\vec{\P}}\subseteq S_{\vec{\Q}}$ and so that
every $P$ satisfies

$$\frac{\dist(I_P, S_{\vec{\P}})}{|I_P|} \sim 2^{k_1}$$
and every $Q$ satisfies

$$\frac{\dist(I_Q, S_{\vec{\Q}})}{|I_Q|} \sim 2^{k_2}$$
for two fixed numbers $k_1, k_2$ so that $k_2 >> k_1$.

Then, the corresponding estimate for the left hand side of (\ref{ineqq}) can be improved to

\begin{equation}\label{ineqq1}
\left|
\sum_{P, Q : |I_P| \leq |I_Q|}
c^i_{P_i}c^j_{Q_j} 
\langle \Phi^i_{P_i}, \Phi^j_{P_j} \rangle
\right| \lesssim 2^{-M k_2}
\end{equation}
for any positive constant $M$, where the implicit constant depends on $M$.

\end{lemma}

The proof of this lemma is quite delicate and will be presented in the last section of the 
paper.

The main proposition used to estimate expressions of the form (\ref{generic}) is the following.

\begin{proposition}\label{tool}
Let $\vec{\P}$ be a rank $1$ collection of vector tiles of dimension $d\geq 3$. Let also consider
arbitrary sequences of complex numbers $(a^j_{P_j})_P$ for $1\leq j\leq d$. Then, one has
the inequality

$$
\left|
\sum_{P\in \vec{\P}}\frac{1}{|I_P|^{\frac{d-2}{2}}}
a^1_{P_1} ... a^d_{P_d}\right| \lesssim
\prod_{j=1}^d
\size_j( (a^j_{P_j})_{P\in\vec{\P}}  )^{\theta_j}\cdot
\energy_j((a^j_{P_j})_{P\in\vec{\P}}  )^{1 - \theta_j}
$$
for any $0\leq \theta_1, ... , \theta_d < 1$ with $\theta_1 + ... + \theta_d = d-2$
with the implicit constants depending on $(\theta_j)_j$.

\end{proposition}

\begin{proof}

The proof of it is based on the following lemma and its corollary, which have been proven in
\cite{mtt-5}.

\begin{lemma}\label{decc}
Let $1\leq j\leq d$, $\vec{\P'}$ be a subset of $\vec{\P}$, $\n\in \Z$ and assume that

$$
\size_j( (a^j_{P_j})_{P\in\vec{\P'}}  ) \leq 2^{-\n}\energy_j((a^j_{P_j})_{P\in\vec{\P}}  ).
$$
Then, one can decompose $\vec{\P'}$ as $\vec{\P'} = \vec{\P''} \cup \vec{\P'''}$ such that 

$$
\size_j( (a^j_{P_j})_{P\in\vec{\P''}}  ) \leq 2^{-\n-1}\energy_j((a^j_{P_j})_{P\in\vec{\P}}  )
$$
and also such that $\vec{\P'''}$ can be written as a disjoint union of trees in $\T$ with the
property that

$$\sum_{T\in \T} |I_T| \lesssim 2^{2\n}.
$$

\end{lemma}

By iterating the above lemma we immediately obtain the following Corollary.

\begin{corollary}\label{fulldecc}
Fix $1\leq j\leq d$. Then, there exists a partition 

$$\vec{\P} = \bigcup_{\n\in\Z} \vec{\P}^j_{\n}$$
where for each $\n\in\Z$ one has

$$
\size_j( (a^j_{P_j})_{P\in\vec{\P}^j_{\n}}  ) \leq 
\min (2^{-\n}\energy_j((a^j_{P_j})_{P\in\vec{\P}}  ), \size_j( (a^j_{P_j})_{P\in\vec{\P}}  ) ).
$$
Also, we can cover $\vec{\P}^j_{\n}$ by a disjoint union $\T^j_{\n}$ of trees such that

$$\sum_{T\in \T^j_{\n}} |I_T| \lesssim 2^{2\n}.
$$

\end{corollary}

We can now start the actual proof of our Proposition \ref{tool}.

First, let us observe that for every $l$-tree $T$ one can estimate the corresponding
term in the inequality by

$$
|
\sum_{P\in T}\frac{1}{|I_P|^{\frac{d-2}{2}}}
a^1_{P_1} ... a^d_{P_d}| \lesssim
\sum_{P\in T}\frac{1}{|I_P|^{\frac{d-2}{2}}}
|a^1_{P_1}| ... |a^d_{P_d}|\lesssim
$$

$$\prod_{k\neq l_1, l_2}
(\sup_{P\in T}\frac{|a^k_{P_k}|}{|I_P|^{1/2}})\cdot
(\sum_{P\in T} |a^{l_1}_{P_{l_1}}|^2 )^{1/2} \cdot
(\sum_{P\in T} |a^{l_2}_{P_{l_2}}|^2 )^{1/2}
$$
for any $l_1\neq l$ and $l_2\neq l$. But this is clearly smaller than

$$\prod_{j=1}^d 
\size_j( (a^j_{P_j})_{P\in T}  )\cdot |I_T|.
$$
Using this simple ``tree estimate'' and applying $d$ times Corollary \ref{fulldecc},
we can estimate our general left hand side of our inequality by

\begin{equation}\label{expression}
E_1 ... E_d \sum_{\n_1, ... ,\n_d}
2^{-\n_1} ... 2^{-\n_d} \sum_{T\in \T_{\n_1,...,\n_d}} |I_T|
\end{equation}
where $\T_{\n_1,...,\n_d}$ is just the intersection of the collections of trees
$\T^j_{\n_j}$ for $j = 1, ... , d$ provided by the Corollary \ref{fulldecc}
and we denoted for simplicity by $E_j := \energy_j((a^j_{P_j})_{P\in\vec{\P}}  )$
and we will also use the notation $S_j$ for $\size_j((a^j_{P_j})_{P\in\vec{\P}}  )$.

One should also observe that as a consequence of the same corollary, the above summations
run inside the set of integers $\n_1, ... , \n_d$ for which

$$2^{-\n_j} \lesssim \frac{S_j}{E_j}$$
for $j = 1, ... , d$. On the other hand, we also know that

$$ \sum_{T\in \T_{\n_1,...,\n_d}} |I_T|\leq \sum_{T\in \T^j_{\n_j}} |I_T|
\lesssim 2^{2\n_j}$$
for any $j = 1, ... , d$ and as a consequence we can write

\begin{equation}\label{luc}
\sum_{T\in \T_{\n_1,...,\n_d}} |I_T| \lesssim \min (2^{2\n_1}, ... , 2^{2\n_d}).
\end{equation}

To prove that the proposition holds for any $0 < \theta_1, ... , \theta_d < 1$,
one can use instead of (\ref{luc}) the weaker inequality

\begin{equation}\label{lucluc}
\sum_{T\in \T_{\n_1,...,\n_d}} |I_T| \lesssim
2^{2\alpha_1} ... 2^{2\alpha_d}
\end{equation}
for any $0 < \alpha_1, ... , \alpha_d < 1$ so that $\alpha_1 + ... + \alpha_d = 1$.

In particular, this allows us to estimate further our expression by

$$
E_1 ... E_d \sum_{\n_1, ... ,\n_d}
2^{-\n_1} ... 2^{-\n_d}
2^{2\alpha_1} ... 2^{2\alpha_d} =
$$

$$E_1 ... E_d \sum_{\n_1, ... ,\n_d} 2^{-\n_1(1-2\alpha_1)} ... 2^{-\n_d(1-2\alpha_d)}.
$$
Now, assuming that in addition we also have $1- 2\alpha_j > 0$, for every
$j = 1, ... , d$ we obtain the upper bound

$$E_1 ... E_d (\frac{S_1}{E_1})^{(1-2\alpha_1)} ... (\frac{S_d}{E_d})^{(1-2\alpha_d)} =
$$

$$ S_1^{(1-2\alpha_1)} ... S_d^{(1-2\alpha_d)}
E_1^{2\alpha_1} ... E_d^{2\alpha_d}.
$$
Since we observe that $(1-2\alpha_1) + ... + (1-2\alpha_d) = d - 2 (\alpha_1 + ... + \alpha_d)
= d-2$, this proves our assertion.

In the case that one of our $\theta$'s is equal to zero (note that at most one
can be zero !), say $\theta_d = 0$, we can estimate our expression
in (\ref{expression}) by

$$
E_1 ... E_d \sum_{\n_1, ... ,\n_d}
2^{-\n_1} ... 2^{-\n_d}
\min (2^{2\n_1}, ... , 2^{2\n_d}) =
$$

$$E_1 ... E_d \sum_{\n_1, ... ,\n_d}
2^{-\n_1} ... 2^{-\n_{d-1}}
\min (2^{\n_d}, 2^{-\n_d} \min (2^{2\n_1}, ... , 2^{2\n_{d-1}})).
$$

Now, if we fix $\n_1, ... , \n_{d-1}$ and first sum over $\n_d$ using the elementary inequality

\begin{equation}
\sum_{\n\in \Z}
\min (2^{\n}, 2^{-\n} a ) \lesssim a^{1/2},
\end{equation}
we obtain the bound

$$
E_1 ... E_d \sum_{\n_1, ... ,\n_{d-1}}
2^{-\n_1} ... 2^{-\n_{d-1}}
\min (2^{\n_1}, ... , 2^{\n_{d-1}}) \lesssim
$$

$$E_1 ... E_d \sum_{\n_1, ... ,\n_{d-1}}
2^{-\n_1} ... 2^{-\n_{d-1}}
2^{\alpha_1 \n_1} ... 2^{\alpha_{d-1} \n_{d-1}}
$$
for every $0 < \alpha_1, ... , \alpha_{d-1} < 1$ with the property that
$\alpha_1 + ... + \alpha_{d-1} = 1$.

After summing the above expression, we obtain the upper bound

$$E_1 ... E_d (\frac{S_1}{E_1})^{1 - \alpha_1} ... 
(\frac{S_{d-1}}{E_{d-1}})^{1 - \alpha_{d-1}} = $$

$$S_1^{1 - \alpha_1} ... S_{d-1}^{1 - \alpha_{d-1}} 
E_1^{\alpha_1} ... E_{d-1}^{\alpha_{d-1}} E_d$$
which coincides with the desired extimate.

\end{proof}

Having Proposition \ref{tool} at our disposal, we can now start the proof of
Theorem \ref{mainthd}. Fix $F_1, ... , F_n$ and $F$ arbitrary measurable sets
of finite measure and $f_1\leq \chi_{F_1}, ... , f_n \leq \chi_{F_n}$. Assume also that
$|F| \sim 1$. Our goal is to construct a subset $F'\subseteq F$ with $|F'| \sim |F|$, so that 
the corresponding inequality in Theorem \ref{mainthd} holds.

First, for every rooted tree $G$ we will construct inductively an exceptional set
$\Omega_G$ as follows. Assume that $G$ has height $1$ and an arbitrary number of leaves $L$.
Then, the exceptional set $\Omega_G$ is defined by
\footnote{M is the classical Hardy-Littlewood maximal operator}

\begin{equation}
\Omega_G : = \{ x : M(\chi_{F_1})(x) > C |F_1|\} \cup ... 
\cup \{ x : M(\chi_{F_L})(x) > C |F_L| \}
\end{equation}
where $C > 0$ is big enough to guarantee that $|\Omega_G| << 1$.

Suppose now that we know how to construct such exceptional sets for arbitrary rooted trees
of height smaller or equal than $h$ and we describe the construction of $\Omega_G$ in the case
of an arbitrary rooted tree of height $h + 1$. Fix such a tree $G$. Assume that the root of it
is $u$ and that the sons of $u$ are $u_1, ... , u_{\#}$. Denote also, as usual, by $G_i$
the subtree of $G$ whose root is $u_i$ for $1\leq i \leq \#$.

Clearly, either $G_i$ is a leave or it is a tree of a strictly smaller height.
Fix $1\leq i \leq \#$ so that $G_i$ is not a leave. Fix also $(k_v)_{v\in V_{G_i}}$
a vector indexed over the vertices of $G_i$ whose entries are all positive integers.
Denote by $T_{G_i}^{(k_v)_v }$ the model operator defined by the same formula as
$T^{G_i}$, but where the implicit sums run over the sets $\vec{\P}_{G_v}^{k_v}$
instead of $\vec{\P}_{G_v}$, where by $\vec{\P}_{G_v}^{k_v}$ we denote the collection of all
vector-tiles $P\in \vec{\P}_{G_v}$ having the property that
\footnote{As usual, $G_v$ is the rooted sub-tree whose root is the vertex $v$ and 
$\Omega_{G_v}$ is the exceptional set associated to $G_v$ which exists by the induction
hypothesis.  }

$$1 + \frac{\dist (I_P, \Omega^c_{G_v})}{ |I_P|} \sim 2^{k_v}.$$
Denote also by $\widetilde{\Omega}_{G_i}$ the set

\begin{equation}
\widetilde{\Omega}_{G_i} : =
\bigcup_{(k_v)_v \in \Z^{|V_{G_i}|}}
\left\{ x : M( T_{G_i}^{(k_v)_v }((f_l)_{l\in \I_i}) )(x) > 
C (\prod_{ v\in V_{G_i}} 2^{k_v}) \| T_{G_i}^{(k_v)_v }((f_l)_{l\in \I_i})   \|_2
\right\}.
\end{equation}

Clearly, $|\widetilde{\Omega}_{G_i}| << 1$ if $C > 0$ is a big enough constant.
Similarly, one defines $\widetilde{\Omega}_{G_j}$ for any other index $1\leq j\leq \#$
for which $G_j$ is not a leave.

In case $G_j$ is a leave, then instead, we define $\widetilde{\Omega}_{G_j}$ by

$$\widetilde{\Omega}_{G_j} : = \{ x : M(\chi_{F_j})(x) > C |F_j| \}.$$
In the end, we define the exceptional set associated to $G$ by

\begin{equation}
\Omega_G : = 
\left(\bigcup_{j=1}^{\#}\widetilde{\Omega}_{G_j}\right)\bigcup
\left( \bigcup_{j : G_j \neq \text{leave}}\Omega_{G_j}\right). 
\end{equation}
Then, we simply define $F'$ by

$$F' : = F\setminus \Omega_G$$
and we observe that indeed, $|F'|\sim 1$ if all the constants $C$ involved are large enough.

We are therefore left with estimating the following expression

\begin{equation}\label{luc2}
\int_0^1\sum_{P\in\vec{\P}_G}
\frac{1}{|I_P|^{\frac{\#-1}{2}}}
\langle T^{G_1}_{|I_P|} ((f_l)_{l\in \I_1}), \Phi^{1, \alpha}1_{P_1} \rangle \cdot...\cdot
\langle T^{G_{\#}}_{|I_P|} ((f_l)_{l\in \I_{\#}}), \Phi^{\#, \alpha}_{P_{\#}}\rangle \cdot
\langle \chi_{F'}, \Phi^{\# + 1, \alpha}_{P_{\# + 1}}\rangle d \alpha.
\end{equation}
Fix now $(k_v)_{v\in V_G}$ positive integers. We will assume from now on that all the implicit 
inner sums in (\ref{luc2}) are taken over collections of the form $\vec{\P}^{k_v}_{G_v}$
for every $v\in V_G$. We will estimate the corresponding term under these restrictions and in 
the end we will sum over all the vectors $(k_v)_{v\in V_G}$.

By applying Proposition \ref{tool} we can estimate (\ref{luc2}) by

\begin{equation}
\sup_{0\leq \alpha \leq 1}
[\size_1 ( (\langle T^{G_1}_{|I_P|} ((f_l)_{l\in \I_1}), \Phi^{1, \alpha}_{P_1} \rangle   )_P)^
{\theta_1}]
\cdot
\sup_{0\leq \alpha \leq 1}
[\energy_1 ( (\langle T^{G_1}_{|I_P|} ((f_l)_{l\in \I_1}), \Phi^{1, \alpha}_{P_1} \rangle   )_P)
^{1- \theta_1}]\cdot ... \cdot
\end{equation}

$$
\sup_{0\leq \alpha \leq 1}
[\size_{\#}( (\langle T^{G_{\#}}_{|I_P|} ((f_l)_{l\in \I_{\#}}), \Phi^{\#, \alpha}_{P_{\#}}
\rangle  )_P)
^{\theta_{\#}}]
\cdot
\sup_{0\leq \alpha \leq 1}
[\energy_{\#}( (\langle T^{G_{\#}}_{|I_P|} ((f_l)_{l\in \I_{\#}}), \Phi^{\#, \alpha}_{P_{\#}}
\rangle  )_P)
^{1- \theta_{\#}}]
\cdot
$$

$$
\sup_{0\leq \alpha \leq 1}
[\size_{\# + 1}( (\langle \chi_{F'}, \Phi^{\# + 1, \alpha}_{P_{\# + 1}}\rangle )_P)
^{\theta_{\# + 1}}]\cdot
\sup_{0\leq \alpha \leq 1}
[\energy_{\# + 1}( (\langle \chi_{E'}, \Phi^{\# + 1, \alpha}_{P_{\# + 1}}\rangle )_P)
^{1- \theta_{\# + 1}}],
$$
for every positive numbers $0\leq \theta_1, ... , \theta_{\# + 1} < 1$ so that
$\theta_1 + ... + \theta_{\# + 1} = \# - 1$.

We write for simplicity the previous expression as

\begin{equation}\label{simplu}
[S_1^{\theta_1} E_1^{1 - \theta_1}]\cdot ... \cdot
[S_{\#}^{\theta_{\#}} E_{\#}^{1- \theta_{\# + 1} }]\cdot
[S_{\# + 1}^{\theta_{\# + 1}} E_{\# + 1}^{1- \theta_{\# + 1}}].
\end{equation}

\underline{Estimates for $[S_1^{\theta_1} E_1^{1 - \theta_1}]$. }

We concentrate now on estimating the term $[S_1^{\theta_1} E_1^{1 - \theta_1}]$.
It will be later on clear that in exactly the same way one can estimate every other term
of type $[S_j^{\theta_j} E_j^{1 - \theta_j}]$ for $1\leq j\leq \#$.

To fix the notations, we also assume that the sons of $u_1$ (which is the root of $G_1$)
are $u_1^1, ... , u_1^{\#_1}$. Denote also by $G_1^i$ the subtree of $G_1$ whose root
is $u_1^i$ for $1\leq i\leq \#_1$.

\underline{Estimates for $E_1$.}

Fix $\alpha$ for which the suppremum in the definition of $E_1$ is attained and consider the 
corresponding
expression. We will also suppress for simplicity the dependence on $\alpha$ in the next 
formulas since its presence is irrelevant.

By duality, we know that there exists a sequence of complex numbers $(C^1_{P_1})_{P_1}$
as in Lemma \ref{CP} so that

$$E_1 \sim \sum_P
\langle T^{G_1}_{|I_P|} ((f_l)_{l\in \I_1}), \Phi^1_{P_1} \rangle \overline{C^1_{P_1}} =
$$

$$ \sum_P
\left\langle
\int_0^1 \sum_{Q : |I_Q|\geq |I_P| }
\frac{1}{|I_Q|^{\frac{\#_1 - 1}{2}}}
\langle T^{G^1_1}_{|I_Q|}((f_l)_{l\in \I^1_1}), \Phi^{1, \alpha}_{Q_1} \rangle ...
\langle T^{G^{\#_1}_1}_{|I_Q|}((f_l)_{l\in \I^{\#_1}_1}), \Phi^{\#_1, \alpha}_{Q_{\#_1}} \rangle
\Phi^{\#_1 + 1, \alpha}_{Q_{\#_1 + 1}} d \alpha , \Phi^1_{P_1} \right \rangle
\overline{C^1_{P_1}} =
$$

\begin{equation}\label{luc3}
\int_0^1
\sum_Q
\frac{1}{|I_Q|^{\frac{\#_1 - 1}{2}}}
\langle T^{G^1_1}_{|I_Q|}((f_l)_{l\in \I^1_1}), \Phi^{1, \alpha}_{Q_1} \rangle ...
\langle T^{G^{\#_1}_1}_{|I_Q|}((f_l)_{l\in \I^{\#_1}_1}), \Phi^{\#_1, \alpha}_{Q_{\#_1}} \rangle
\cdot
\langle \sum_{P : |I_P|\leq |I_Q|} C^1_{P_1}\Phi^1_{P_1} , \Phi^{\#_1 + 1, \alpha}_{Q_{\#_1 + 1}}
\rangle d \alpha.
\end{equation}
By applying the same Proposition \ref{tool}, we can estimate this further by

\begin{equation}
[(S^1_1)^{\beta_1} (E^1_1)^{1 - \beta_1}]\cdot ... 
[(S^{\#_1}_1)^{\beta_{\#_1}} (E^{\#_1}_1)^{1 - \beta_{\#_1}}]\cdot 1 \cdot
E_1^{\#_1 + 1},
\end{equation}
for any $0 < \beta_1, ... , \beta_{\#_1} < 1$ so that 
$\beta_1 + ... + \beta_{\#_1} = \#_1 - 1$. By $S^j_1$ and $E^j_1$ we denoted the expressions

$$S^j_1 : = \sup_{ 0 \leq \alpha \leq 1}
[\size_j ( (\langle T^{G^j_1}_{|I_Q|} ((f_l)_{l\in \I^j_1}) 
\Phi^{j, \alpha}_{Q_j} \rangle   )_Q)]
$$
and 

$$E^j_1 : = \sup_{ 0 \leq \alpha \leq 1}
[\energy_j ( (\langle T^{G^j_1}_{|I_Q|} ((f_l)_{l\in \I^j_1}), 
\Phi^{j, \alpha}_{Q_j} \rangle   )_Q)],
$$
for any $1\leq j\leq \#_1$. Clearly, $E_1^{\#_1 + 1}$ is defined similarly and corresponds
to the last term in (\ref{luc3}).

To estimate $E_1^{\#_1 + 1}$, we fix as before the index $\alpha$ for which the
suppremum is attained and consider the corresponding expression. We suppress again the 
dependence on $\alpha$ for simplicity, since it is irrelevant to the argument.

By using Lemma \ref{CP} there exists a sequence of complex numbers
$(C^{\#_1 + 1}_{Q_{\# + 1}})_Q$ having the property that

$$E_1^{\#_1 + 1} \sim \sum_{P, Q : |I_P| \leq |I_Q| }
C^1_{P_1}\overline { C^{\#_1 + 1}_{Q_{\# + 1}} }
\langle \Phi^1_{P_1} , \Phi^{\#_1 + 1}_{Q_{\#_1 + 1}}\rangle.
$$
Let us also recall that the summation above runs over $P\in \P_G^{k_u}$ and
$Q\in \P_{G_1}^{k_{u_1}}$. On the other hand, by construction, we also know that
$\Omega_{G_1} \subseteq \Omega_G$ and so $\Omega_G^c \subseteq \Omega_{G_1}^c$.

In particular, this means that by using the previous Lemmas \ref{bessel}
and \ref{delicatebessel} we have that $E_1^{\#_1 + 1}$ is $O(1)$ in general, but in the case
when $k_{u_1} >> k_u$ we have that $E_1^{\#_1 + 1} \lesssim 2^{-M k_{u_1}}$ for arbitrary
constants $M > 0$ (with the implicit constant depending on $M$).

\underline{Estimates for $S_1$.}

Fix $\alpha$ for which the suppremum in the definition of $S_1$ is attained and consider the 
corresponding expression with $\alpha$ suppressed. We have

$$S_1 \leq
\size_1 (\langle T^{G_1}((f_l)_{l\in \I_1}), \Phi^1_{P_1} \rangle _P ) + 
\size_1 (\langle T^{G_1, \ast}_{|I_P|}((f_l)_{l\in \I_1}), \Phi^1_{P_1} \rangle _P )
: = I + II
$$
where $T^{G_1, \ast}_{|I_P|}$ is defined as being the sum over those vector-tiles with the 
property $|I_Q| < |I_P|$.

To estimate $I$, from the definition of the exceptional sets, it is easy to see that one has

$$ I \lesssim
(\prod_{v\in V_{G_1}} 2^{k_v})\cdot 2^{k_u} \cdot
\|T^{G_1}((f_l)_{l\in \I_1})\|_2.
$$

To estimate $\|T^{G_1}((f_l)_{l\in \I_1})\|_2$ we pick $g\in L^2$, $\|g\|_2 = 1$
so that this term becomes equivalent with

\begin{equation}
\int_0^1
\sum_Q
\frac{1}{|I_Q|^{\frac{\#_1 - 1}{2}}}
\langle T^{G^1_1}_{|I_Q|}((f_l)_{l\in \I^1_1}), \Phi^{1, \alpha}_{Q_1} \rangle ...
\langle T^{G^{\#_1}_1}_{|I_Q|}((f_l)_{l\in \I^{\#_1}_1}), \Phi^{\#_1, \alpha}_{Q_{\#_1}} \rangle
\cdot
\langle g, \Phi^{\#_1 + 1, \alpha}_{Q_{\#_1 + 1}}\rangle
d \alpha
\end{equation}
and as before by using the same Proposition \ref{tool} together with Lemma \ref{energyest}
we get an estimate of the form

$$
[(S^1_1)^{\beta_1} (E^1_1)^{1 - \beta_1}]\cdot ... 
[(S^{\#_1}_1)^{\beta_{\#_1}} (E^{\#_1}_1)^{1 - \beta_{\#_1}}]\cdot 1 \cdot
1,
$$
for every $\beta_1, ... , \beta_{\#_1}$ exactly as before.

To estimate $II$, pick $T$ a tree where the corresponding suppremum is attained.
Then, observe that $II$ becomes equivalent with

$$\frac{1}{|I_T|}
\|(\sum_{P\in T}
\frac{|\langle T^{G_1}_T((f_l)_{l\in \I_1}), \Phi^1_{P_1} \rangle |^2}{|I_P|}
\chi_{I_P} )^{1/2} \|_{1, \infty}
$$
where $T^{G_1}_T$ is defined to be the corresponding sum over vector-tiles $Q$
having the property that there exists a $P\in T$ so that $|\omega_{P_1} | < 
|\omega_{Q_{\#_1 + 1}} |$ and
$\omega_{P_1} \cap \omega_{Q_{\#_1 + 1}} \neq \emptyset$. We denote this set of vector-tiles by
$\Q_T$. By using Lemma \ref{sizeest} we see that the above term is smaller than

$$\frac{1}{|I_T|}
\| T^{G_1}_T((f_l)_{l\in \I_1})    \|_{L^1(\widetilde{\chi}_{I_T}^M)} = 
$$

$$ \frac{1}{|I_T|}
\int_0^1\sum_{Q\in \Q_T}
\frac{1}{|I_Q|^{\frac{\#_1 - 1}{2}}}
\langle T^{G^1_1}_{|I_Q|}((f_l)_{l\in \I^1_1}), \Phi^{1, \alpha}_{Q_1} \rangle ...
\langle T^{G^{\#_1}_1}_{|I_Q|}((f_l)_{l\in \I^{\#_1}_1}), \Phi^{\#_1, \alpha}_{Q_{\#_1}} \rangle
\cdot
\langle h \widetilde{\chi}_{I_T}^M, \Phi^{\#_1 + 1, \alpha}_{Q_{\#_1 + 1}}\rangle
d \alpha
$$
for some well chosen function $h\in L^{\infty}$ with $\|h\|_{\infty} = 1$.

This can be also written as

$$ \frac{1}{|I_T|}
\sum_{m=1}^{\infty}
\int_0^1\sum_{Q\in \Q^m_T}
\frac{1}{|I_Q|^{\frac{\#_1 - 1}{2}}}
\langle T^{G^1_1}_{|I_Q|}((f_l)_{l\in \I^1_1}), \Phi^{1, \alpha}_{Q_1} \rangle ...
\langle T^{G^{\#_1}_1}_{|I_Q|}((f_l)_{l\in \I^{\#_1}_1}), \Phi^{\#_1, \alpha}_{Q_{\#_1}} \rangle
\cdot
\langle h \widetilde{\chi}_{I_T}^M, \Phi^{\#_1 + 1, \alpha}_{Q_{\#_1 + 1}}\rangle
d \alpha
$$
where $\Q_T^m$ denotes the set of all vector-tiles $Q\in \Q_T$ having the property that

$$m-1 \leq \frac{\dist ( I_T, I_Q)}{|I_T|} \leq m.$$
Since it is not difficult to see that all these $\Q_T^m$ sets are trees, we deduce that the above expression can be estimated
by

\begin{equation}
\frac{1}{|I_T|}
\sum_{m=1}^{\infty}
S_1^1\cdot ... \cdot S_1^{\#_1}\cdot
\sup_{0\leq \alpha \leq 1}
\sup_{Q\in \Q^m_T}
\frac{| \langle h \widetilde{\chi}_{I_T}^M, \Phi^{\#_1 + 1, \alpha}_{Q_{\#_1 + 1}}\rangle  |}
{|I_Q|^{1/2}}\cdot |I_T| \lesssim
\end{equation}

$$
\frac{1}{|I_T|}
\sum_{m=1}^{\infty}
\frac{1}{m^{10}} S_1^1\cdot ... \cdot S_1^{\#_1} |I_T|
\lesssim S_1^1\cdot ... \cdot S_1^{\#_1}.
$$
Putting all these estimates together, we can finish our original estimate for
$S_1^{\theta_1}\cdot E_1^{1-\theta_1}$ as follows:

$$S_1^{\theta_1}\cdot E_1^{1-\theta_1} \leq
(I + II)^{\theta_1}\cdot E_1^{1-\theta_1} \leq
I^{\theta_1}\cdot E_1^{1-\theta_1} + II^{\theta_1}\cdot E_1^{1-\theta_1}\leq
$$

$$
C\cdot
[(S^1_1)^{\beta_1}\cdot (E_1^1)^{1-\beta_1}]\cdot ... 
\cdot [(S_1^{\#_1})^{\beta_{\#_1}}\cdot (E_1^{\#_1})^{1-\beta_{\#_1}}] =
$$

$$C\cdot [S_1^1\cdot ... \cdot S_1^{\#_1}]^{\theta_1}\cdot
[(S^1_1)^{\beta_1}\cdot (E_1^1)^{1-\beta_1}\cdot ... 
\cdot (S_1^{\#_1})^{\beta_{\#_1}}\cdot (E_1^{\#_1})^{1-\beta_{\#_1}}]^{1-\theta_1} +
$$

$$
C\cdot
[(S^1_1)^{\beta_1}\cdot (E_1^1)^{1-\beta_1}]\cdot ... 
\cdot [(S_1^{\#_1})^{\beta_{\#_1}}\cdot (E_1^{\#_1})^{1-\beta_{\#_1}}] + 
$$

$$
C\cdot 
[(S_1^1)^{\theta_1 + \beta_1 (1-\theta_1)}\cdot (E_1^1)^{(1-\beta_1)(1-\theta_1)}]
\cdot ... \cdot
[(S_1^{\#_1})^{\theta_1 + \beta_{\#_1} (1-\theta_1)}\cdot (E_1^{\#_1})^{(1-\beta_{\#_1})
(1-\theta_1)}]
: =
$$

$$
C\cdot
[(S^1_1)^{\beta_1}\cdot (E_1^1)^{1-\beta_1}]\cdot ... 
\cdot [(S_1^{\#_1})^{\beta_{\#_1}}\cdot (E_1^{\#_1})^{1-\beta_{\#_1}}] + 
$$

$$
C\cdot
[(S^1_1)^{\gamma_1}\cdot (E_1^1)^{1-\gamma_1}]\cdot ... 
\cdot [(S_1^{\#_1})^{\gamma_{\#_1}}\cdot (E_1^{\#_1})^{1-\gamma_{\#_1}}],
$$
for some numbers $0 < \gamma_1, ... , \gamma_{\#_1} < 1$.

In other words we showed that the expression 
$[S_1^{\theta_1}\cdot E_1^{1 - \theta_1}]$ which corresponds to the subtree $G_1$
(whose root is $u_1$), can be estimated by a finite sum of products of similar expressions
involving terms which correspond to the sons of $u_1$, namely $u_1^1, ... , u_1^{\#_1}$.
It is not difficult to observe that one can estimate in exactly the same way all the other 
expressions of the form
$[S_j^{\theta_j}\cdot E_j^{1 - \theta_j}]$ in (\ref{simplu}) for $1\leq j\leq \#$.
Also, it is important to observe that the constants $C$ above depend as we have seen on the 
integers $(k_v)_{v\in V_{G_1}}$ and also on $k_u$. 

Clearly, one can then iterate this
procedure further, eventualy arriving at estimating the expressions corresponding to the leaves
of the rooted tree $G$. It is then easy to see that in the case when the sequence
$(a_P)_P$ is indeed of the form $(\langle f_j, \Phi_P \rangle )_P$, one clearly has

$$S := \size ( (\langle f_j, \Phi_P \rangle )_P) \lesssim 2^{k_v} \min ( |F_j|, 1)
\lesssim 2^{k_v} |F_j|^{\beta}$$
for every $ 0 < \beta < 1$, where $v$ is the vertex whose son is the leave indexed ``$j$'',
while

$$E := \energy ( (\langle f_j, \Phi_P \rangle )_P) \lesssim |F_j|^{1/2}.$$
In particular, this implies that any product of the form $S^{\theta}\cdot E^{1-\theta}$
for some $0 < \theta < 1$, becomes smaller than

\begin{equation}
|F_j|^{\beta \theta} \cdot |F_j|^{(1-\theta)/2} = |F_j|^{1/2 + \theta (\beta - 1/2)}
\end{equation}
and clearly, the exponent $1/2 + \theta (\beta - 1/2)$ can be made arbitarily close
to $\frac{1}{2}$ by taking $\beta$ close to $\frac{1}{2}$ (no matter which $\theta$ we face).

Putting all these estimates together and also using the fact that the size of the last sequence
$ (\langle \chi_{F'}, \Phi^{\# + 1}_{\# + 1} \rangle )_P$ in (\ref{luc2}) is smaller
than $2^{-M k_u}$ (for any $M$ arbitrarily large, with the implicit constants depending on it)
while the energy of it is $O(1)$, we obtain au upper bound of the form

\begin{equation}
C\cdot |F_1|^{\alpha_1}\cdot ... \cdot |F_n|^{\alpha_n}
\end{equation}
for $\alpha_1, ... , \alpha_n$ arbitraily close to $\frac{1}{2}$ (as required in
Theorem \ref{mainthd}), where the constant $C$ depends on all the integers
$(k_v)_{v\in V_G}$ fixed before.

However, it is not difficult to see that $C$ is actually of the form

$$ C_n \prod_{v\in V_G} 2^{C_s k_v}\cdot 2^{C_e k_v}$$
where $C_s$ is the constant coming from estimating the various sizes which can be chosen
to be dependent only on $n$, while $C_e$ is the constant coming from estimating the various 
energies which is either zero, or it is of the form $-M$ for an arbitrarily big
$M$, as we have seen.

In order to see that this big geometric series is convergent when we sum over all the 
positive integers $(k_v)_{v\in V_G}$, it is enough to observe that every time $v$
and $w$ are adjacent and say $w$ is the son of $v$, we have always three possibilities:
either $k_w << k_v$ or $k_w \sim k_v$ or $k_w >> k_v$. In the first two cases one should first 
sum over $k_w$ and get a bound of the form $2^{C k_v}$, while in the third one should also sum
over $k_w$ first since this time the corresponding energy estimate comes with a factor of type 
$2^{-M k_w}$ and in this case we also obtain an upper bound of the same type $2^{ C k_v}$.

Hence, if one starts summing from the vertices having the highest levels
(those vertices whose sons are the leaves of the tree) and continues until one reaches
the root of the tree, one sees that the geometric sum is indeed convergent.
This ends the proof of Theorem \ref{mainthd}.

We are therefore left with proving Lemma \ref{delicatebessel}.

\section{Proof of Lemma \ref{delicatebessel}}

Fix $\vec{\P}, \vec{\Q}$ and $S_{\vec{\P}}\subseteq S_{\vec{\Q}}$, fix also
$k_2 >> k_1$ and to simplify the notation suppress from now on the dependence on the indices
$i, j$ which appear in Lemma \ref{delicatebessel}. We would therefore like to prove that

\begin{equation}\label{last}
\left|
\sum_{P, Q : |I_P| \leq |I_Q|} C_P C_Q
\langle \Phi_P, \Phi_Q \rangle \right |
\lesssim 2^{ - M k_2}
\end{equation}
for any $M > 0$ with the implicit constants depending on $M$. We also know from hypothesis that
the collections $\vec{\P}$ and $\vec{\Q}$ can be written as unions of strongly disjoint
trees which we call $\T$ and $\T'$ respectively. We also denote by $S$ and $S'$ the expressions

$$ S : = \sum_{T\in \T} |I_T|$$
and 

$$ S' := \sum_{T\in \T'} |I_T|.$$
We should note from the very beginning the crucial fact that for any $Q_1$, $Q_2$ with
$|I_{Q_1}| \neq |I_{Q_2}|$ one has $I_{Q_1} \cap I_{Q_2} = \emptyset$ and also that
$2^{k_2 - 5} I_Q \cap I_P = \emptyset$ for any $P$ and $Q$. It is also important to note that since $k_2 >> k_1$ one has that all the trees in $\T'$ are ``one-tile trees''
and as a consequence we have $S' = \sum_Q |I_Q|$.

Using these, the left hand side of the above inequality (\ref{last}) can be estimated by

$$\sum _{P, Q}
|C_P| |C_Q| |\langle \Phi_P, \Phi_Q \rangle| \lesssim
\sum_{Q}
|C_Q| \left(\sum_{P : |I_P| \leq |I_Q|, \omega_P \cap \omega_Q \neq \emptyset}|C_P||\langle \Phi_P, \Phi_Q \rangle|\right)\lesssim
$$

$$\frac{1}{(S')^{1/2}}\sum_Q |I_Q|^{1/2}\left(\sum_{P : |I_P| \leq |I_Q|, \omega_P \cap \omega_Q \neq \emptyset}|C_P||\langle \Phi_P, \Phi_Q \rangle|\right)\lesssim$$

$$\frac{1}{(S')^{1/2}}\sum_Q |I_Q|^{1/2}\left(\sum_{P : |I_P| \leq |I_Q|, \omega_P \cap \omega_Q \neq \emptyset}\frac{|I_P|^{1/2}}{S^{1/2}}
|\langle \Phi_P, \Phi_Q \rangle|\right)\lesssim$$

$$
\frac{1}{(S')^{1/2}} \frac{1}{S^{1/2}}
\sum_Q
\left(\sum_{P : |I_P| \leq |I_Q|, \omega_P \cap \omega_Q \neq \emptyset} |\langle \widetilde{\chi}^{2N}_{I_P}, \widetilde{\chi}^{2N}_{I_Q} \rangle|\right),$$
for any $N > 0$ with the implicit constants depending on $N$.

Using now the fact that all the $P$ tiles are disjoint together with the previous observation that $2^{k_2 - 5} I_Q \cap I_P = \emptyset$ one can estimate the previous 
expression further by

$$2^{-N k_2}\frac{1}{(S')^{1/2}} \frac{1}{S^{1/2}}
\sum_Q 
\sum_{P : |I_P| \leq |I_Q|, \omega_P \cap \omega_Q \neq \emptyset}
\left( 1 + \frac{\dist (I_P, I_Q)}{|I_Q|} \right) ^{-N} |I_P| \lesssim$$

$$
2^{-N k_2}\frac{1}{(S')^{1/2}} \frac{1}{S^{1/2}}\sum_Q |I_Q| = 2^{-N_{k_2}}\frac{1}{(S')^{1/2}} \frac{1}{S^{1/2}} S' =
2^{-N k_2}\frac{(S')^{1/2}}{S^{1/2}}.$$

Let us assume now that $N = 10 M$ where $M$ is the generic number in the hypothesis.
If we knew that 

$$\frac{(S')^{1/2}}{S^{1/2}}\lesssim 2^{ 9 M k_2}$$
then we would be done.

We are therefore left with understanding the opposite case when one has

\begin{equation}\label{100}
\frac{S^{1/2}}{(S')^{1/2}}\lesssim 2^{ - 9 M k_2}
\end{equation}

In this case, we write the left hand side of (\ref{last}) as

$$
\sum_{P, Q : |I_P| \leq |I_Q|} C_P C_Q
\langle \Phi_P, \Phi_Q \rangle = 
- \sum_{P, Q : |I_P| > |I_Q|} C_P C_Q
\langle \Phi_P, \Phi_Q \rangle + \sum_{P, Q } C_P C_Q
\langle \Phi_P, \Phi_Q \rangle =
$$

\begin{equation}\label{101}
- \sum_{P, Q : |I_P| > |I_Q| } C_P C_Q
\langle \Phi_P, \Phi_Q \rangle +
\langle \sum_P C_P \Phi_P, \sum_Q C_Q \Phi_Q \rangle .
\end{equation}

To estimate the first term in (\ref{101}), we can write

$$
\left|\sum_{P, Q : |I_P| > |I_Q|} C_P C_Q \langle \Phi_P, \Phi_Q \rangle\right| =
\left|\sum_{P, Q : |I_P| > |I_Q|, \omega_P \cap \omega_Q \neq \emptyset } C_P C_Q
\langle \Phi_P, \Phi_Q \rangle\right| = 
$$

$$
\left|
\sum_T \sum_{P\in T} C_P
\left(\sum_{P : |I_P| > |I_Q|, \omega_P \cap \omega_Q \neq \emptyset} C_Q \langle \Phi_P, \Phi_Q \rangle\right) 
\right| \leq 
$$

\begin{equation}\label{1011}
\left|
\sum_T \sum_{P\in T} C_P
\left(\sum_{Q : |I_P| > |I_Q|, \omega_P \cap \omega_Q \neq \emptyset, I_T \cap I_Q  \neq \emptyset} C_Q \langle \Phi_P, \Phi_Q \rangle\right)
\right| +
\end{equation}

$$
\left|
\sum_T \sum_{P\in T} C_P
\left(\sum_{Q : |I_P| > |I_Q|, \omega_P \cap \omega_Q \neq \emptyset, I_T \cap I_Q  = \emptyset} C_Q \langle \Phi_P, \Phi_Q \rangle\right)
\right|.
$$

It is not difficult to see that the first term in (\ref{1011}) can be written as

\begin{equation}\label{1012}
\sum_T \sum_{ P\in T} C_P \langle h_T, \Phi_Q \rangle
\end{equation}
where 

$$ h_T : = \sum_{Q\in \Q_T} C_Q \Phi_Q$$
and $\Q_T$ is defined to be the set of all tiles $Q$ with $I_Q\subseteq I_T$ for which there exists $P\in T$ so that $|I_P| > |I_Q|$
and $\omega_P \cap \omega_Q \neq \emptyset$. It is also not difficult to observe that all the intervals $I_Q$ for $Q\in \Q_T$ are disjoint.
In particular, (\ref{1012}) is smaller than

$$\sum_{T} (\sum_{P\in T} |C_P|^2 )^{1/2} (\sum_{P\in T} |\langle h_T, \Phi_Q \rangle |^2 )^{1/2} \lesssim
\frac{1}{S^{1/2}} \sum_T |I_T|^{1/2} \| h_T\|_2 \lesssim
\frac{1}{S^{1/2}} \sum_T |I_T|^{1/2} (\sum_{Q\in \Q_T} |C_Q|^2 )^{1/2} \lesssim
$$

$$
\frac{1}{S^{1/2}} \frac{1}{(S')^{1/2}} \sum_T |I_T|^{1/2} (\sum_{Q\in \Q_T} |I_Q| )^{1/2} \lesssim
\frac{1}{S^{1/2}} \frac{1}{(S')^{1/2}} \sum_T |I_T| = 
\frac{1}{S^{1/2}} \frac{1}{(S')^{1/2}} S = \frac{S^{1/2}}{(S')^{1/2}} \lesssim 2^{ - 9 M k_2}$$
by using  (\ref{100}).

Then, the second term in (\ref{1011}) can be estimated by

$$\sum_T
(\sum_{P\in T} |C_P|^2 )^{1/2}
\left(\sum_{P\in T}
\left(\sum_{Q : |I_P| > |I_Q|, \omega_P \cap \omega_Q \neq \emptyset, I_T \cap I_Q  = \emptyset} |C_Q| |\langle \Phi_P, \Phi_Q |\rangle\right) ^2 \right)^{1/2} \lesssim
$$

\begin{equation}\label{1013}
\frac{1}{S^{1/2}} \frac{1}{(S')^{1/2}}
\sum_T |I_T|^{1/2}
\left(\sum_{P\in T}
\left(\sum_{Q : |I_P| > |I_Q|, \omega_P \cap \omega_Q \neq \emptyset, I_T \cap I_Q  = \emptyset} |I_Q|^{1/2} |\langle \Phi_P, \Phi_Q \rangle |\right) ^2 \right)^{1/2}.
\end{equation}

Fix now $T$ and $P\in T$ and look at the corresponding inner term in (\ref{1013}). It can be estimated by

$$|I_P|^{- 1/2}
\sum_{Q : |I_P| > |I_Q|, \omega_P \cap \omega_Q \neq \emptyset, I_T \cap I_Q  = \emptyset}
\langle \widetilde{\chi}^N_{I_P}, \widetilde{\chi}^N_{I_Q}\rangle \lesssim
$$

$$|I_P|^{- 1/2}
\sum_{Q : |I_P| > |I_Q|, \omega_P \cap \omega_Q \neq \emptyset, I_T \cap I_Q  = \emptyset}
\left( 1 + \frac{\dist (I_Q, I_P)}{|I_P|}\right)^{- N} |I_Q| \lesssim
|I_P|^{- 1/2} 
\left( 1 + \frac{\dist (I_P, I_T^c)}{|I_P|}\right)^{- N} |I_P| =
$$

$$
|I_P|^{1/2} 
\left( 1 + \frac{\dist (I_P, I_T^c)}{|I_P|}\right)^{- N}
$$
since the time intervals $I_Q$ which contribute to the above sum are all disjoint.

In particular, (\ref{1013}) becomes smaller than

$$
\frac{1}{S^{1/2}}
\frac{1}{(S')^{1/2}}
\sum_T |I_T|^{1/2}
\left(\sum_{P\in T} |I_P|
\left( 1 + \frac{\dist (I_P, I_T^c)}{|I_P|}\right)^{- 2 N}\right) ^{1/2} =
$$

$$
\frac{1}{S^{1/2}}
\frac{1}{(S')^{1/2}}
\sum_T |I_T|^{1/2}
\left(\sum_{k=0}^{\infty} \sum_{P\in T : |I_P| = 2^{- k} |I_T| } 2^{- k} |I_T|
\left( 1 + \frac{\dist (I_P, I_T^c)}{|I_P|}\right)^{- 2 N}\right) ^{1/2} =
$$

$$
\frac{1}{S^{1/2}}
\frac{1}{(S')^{1/2}}
\sum_T |I_T|^{1/2}
\left(\sum_{k=0}^{\infty} 2^{- k} \sum_{P\in T : |I_P| = 2^{- k} |I_T| }
\left( 1 + \frac{\dist (I_P, I_T^c)}{|I_P|}\right)^{- 2 N}\right) ^{1/2} |I_T|^{1/2} \lesssim
$$

$$
\frac{1}{S^{1/2}}
\frac{1}{(S')^{1/2}}
\sum_T |I_T| = \frac{1}{S^{1/2}} \frac{1}{(S')^{1/2}} S =
\frac{S^{1/2}}{ (S')^{1/2}} \lesssim 2^{ - 9 M k_2},
$$
as before.

In conclusion, we are left with estimating the second term in (\ref{101}) namely the expresssion

$$|\langle \sum_P C_P \Phi_P, \sum_Q C_Q \Phi_Q \rangle |.
$$
One should first observe that by simply applying Cauchy-Schwartz, we get a bound of the form

$$\|\sum_P C_P \Phi_P\|_2 \cdot \|\sum_Q C_Q \Phi_Q\|_2$$
and this is $O(1)$ by a result from \cite{mtt-5}. Unfortunately, this is not enough 
since we need this time 
to get an extra factor of type $2^{- M k_2}$. We need to introduce a few notations and 
definitions to proceed further. 

If $I$ is an arbitrary dyadic interval, we say that a smooth function $\widetilde{\Phi}_I$
is a $relaxed$ $bump$ adapted to $I$ if and only if one has

\begin{equation}\label{relax}
\left|\frac{d^l}{d x^l} [ \widetilde{\Phi}_I (x) ]\right|\lesssim
|I|^{- l} 
\left( 1 + \frac{\dist (x, I)}{|I|}\right)^{-10}
\end{equation}
for any $1\leq l\leq 10$.

Then, if $Q$ is an arbitrary tile $Q = I_Q \times \omega_Q$ we say that 
$\widetilde{\Phi}_Q$ is a $relaxed$ $wave$ $packet$ adapted to $Q$ if and only if
$\widetilde{\Phi}_Q (x) = \widetilde{\Phi}_{I_Q} (x)\cdot 2^{2\pi i x\xi_Q}$ where
$\xi_Q$ is the center of the frequency interval $\omega_Q$ and $\widetilde{\Phi}_{I_Q}$
is any $relaxed$ $bump$ adapted to the interval $I_Q$. Note that this time
we do not assume arbitrary decay and also we do not assume that the Fourier transform of 
$\widetilde{\Phi}_Q$ has compact support.

The following Lemma will be useful.

\begin{lemma}\label{scalarproduct}
Let $Q_1$ and $Q_2$ be two tiles so that $|I_{Q_1}|\geq |I_{Q_2}|$. Then,
if $\widetilde{\Phi}_{Q_1}$ and $\widetilde{\Phi}_{Q_2}$ are relaxed wave packets
adapted to $Q_1$ and $Q_2$ respectively, one has the estimate

$$\left|
\langle \widetilde{\Phi}_{Q_1}, \widetilde{\Phi}_{Q_2}\rangle \right|\lesssim
\left( 1 + \frac{\dist (\omega_{Q_1}, \omega_{Q_2})}{|\omega_{Q_2}|}\right)^{ - 10}
\cdot
\int_{\R}\widetilde{\chi}_{I_{Q_1}}(x) \widetilde{\chi}_{I_{Q_2}}(x) dx 
\lesssim
$$

$$
\left( 1 + \frac{\dist (\omega_{Q_1}, \omega_{Q_2})}{|\omega_{Q_2}|}\right)^{ - 10}
\cdot
\left( 1 + \frac{\dist (I_{Q_1}, I_{Q_2})}{|I_{Q_1}|}\right)^{ - 10}\cdot
|I_{Q_2}|.$$

\end{lemma}

\begin{proof} Clearly, if both $\widetilde{\Phi}_{Q_1}$ and $\widetilde{\Phi}_{Q_2}$ would be
``real wave packets'' (therefore compactly supported in frequency) then the first factor
$\left( 1 + \frac{\dist (\omega_{Q_1}, \omega_{Q_2})}{|\omega_{Q_2}|}\right)$
has to be equal to $1$, otherwise the scalar product would be zero. In that case, 
the estimate simply becomes the usual estimate of the 
scalar product of two bump functions. Since this is not the case, one has instead
to take advantage of the oscillation of $\widetilde{\Phi}_{Q_2}$ in the situation when
$\left( 1 + \frac{\dist (\omega_{Q_1}, \omega_{Q_2})}{|\omega_{Q_2}|}\right)$ is a big number
by the usual integration by parts argument which should be performed ten times.

The straightforward details are left to the reader.

\end{proof}

The following lemma will also be important.

\begin{lemma}\label{L2}
Consider for each $Q\in \vec{\Q}$ a relaxed $L^2$-normalized wave packet 
$\widetilde{\Phi}_{Q}$ adapted to the tile $Q$. Then, if $(C_Q)_Q$ is a sequence of
complex numbers as before, then one has

$$\|\sum_{Q} C_Q \widetilde{\Phi}_{Q}\|_2 \lesssim 1.
$$

\end{lemma}

\begin{proof} First of all, let us recall that since any $Q$ has the property that

$$\frac{\dist (I_Q, S_{\vec{Q}})}{|I_Q|} \sim 2^{ k_2}$$
and $k_2$ is a large positive integer, one has that every time $Q$ and $Q'$ are so that
$|I_Q|\neq |I_{Q'}|$ then one must have $I_Q \cap I_{Q'} = \emptyset$. This also implies
that our collection of tiles can contain only one-tile trees and so
$S' = \sum_{T\in \T'} |I_T| = \sum_{Q} |I_Q|$.

Using this, one can write

$$\|\sum_{Q} C_Q \widetilde{\Phi}_{Q}\|_2^2 =
|\langle \sum_{Q} C_Q \widetilde{\Phi}_{Q}, \sum_{Q'} C_{Q'} \widetilde{\Phi}_{Q'}
\rangle | = 
$$

$$|\sum_{Q, Q'}
C_Q C_{Q'}
\langle \widetilde{\Phi}_{Q}, \widetilde{\Phi}_{Q'}\rangle \lesssim
\sum_{Q, Q'}
\frac{|I_Q|^{1/2}}{(S')^{1/2}}
\frac{|I_{Q'}|^{1/2}}{(S')^{1/2}}
|\langle \widetilde{\Phi}_{Q}, \widetilde{\Phi}_{Q'}\rangle  | : = 
$$

$$
\frac{1}{(S')^{1/2}}
\frac{1}{(S')^{1/2}}
\sum_{Q, Q'}
|\langle \widetilde{\chi}^{\infty}_{Q}, \widetilde{\chi}^{\infty}_{Q'}\rangle| \lesssim
\frac{1}{S'}\sum_Q \left (
\sum_{Q': |I_{Q'}|\leq |I_Q|}|\langle \widetilde{\chi}^{\infty}_{Q}, \widetilde{\chi}^{\infty}_{Q'}\rangle|
\right).
$$
Using now our previous Lemma \ref{scalarproduct} together with the observations made at the 
beginning of the proof, it is not difficult to see that the last 
expression is smaller than

$$\frac{1}{S'}\sum_Q |I_Q| = 1$$
which ends the proof.

\end{proof}

Coming back now to our expression
$\langle \sum_P C_P \Phi_P, \sum_Q C_Q \Phi_Q \rangle $ we will do the following.
For each $Q$, split the corresponding $\Phi_Q$ as

$$\Phi_Q = \sum_{l\in \Z} \Phi_Q^l$$
where each $\Phi_Q^l$ is defined to be the old function $\Phi_Q$ multiplied by a  cut-off
bump function supported on an interval of comparable length with $I_Q$ but $l$ units
of lenght $|I_Q|$ away from $I_Q$. Since $\Phi_Q$ is a Schwartz function, we can further write
$\Phi_Q$ as

$$\Phi_Q = \sum_{l\in \Z} \frac{1}{( 1 + |l| )^N}\widetilde{\Phi}_Q^l.$$
As a consequence, our expression splits as

$$\sum_{l\in \Z} \frac{1}{( 1 + |l| )^N}
\langle \sum_Q C_Q \widetilde{\Phi}^l_Q, \sum_P C_P \Phi_P \rangle
:= I + II
$$
where

$$I := \sum_{l\in \Z : |l|\leq 2^{k_2 - 5}} \frac{1}{( 1 + |l| )^N}
\langle \sum_Q C_Q \widetilde{\Phi}^l_Q, \sum_P C_P \Phi_P \rangle
$$
and

$$II := \sum_{l\in \Z : |l| > 2^{k_2 - 5}} \frac{1}{( 1 + |l| )^N}
\langle \sum_Q C_Q \widetilde{\Phi}^l_Q, \sum_P C_P \Phi_P \rangle.
$$
Now, it is not difficult to remark that for each fixed $l$ with $|l| > 2^{k_2 - 5}$,
the function $\widetilde{\Phi}^l_Q$ is also a relaxed wave packet adapted to
$Q$. In particular, this implies that one can simply apply Cauchy-Schwartz
for term II, together with Lemma \ref{L2} to bound it by $2^{- M k_2}$ as desired.

It is therefore enough to estimate term $I$. Fix $l\in \Z$ so that
$|l|\leq 2^{k_2 - 5}$ and consider the expression

\begin{equation}\label{102}
\langle \sum_Q C_Q \widetilde{\Phi}_Q^l, \sum_P C_P \Phi_P \rangle.
\end{equation}
This time, we will take advantage of the fact that the functions $\widetilde{\Phi}_Q^l$
are all compactly supported.

First of all, let us denote by $\I$ the collection of all dyadic intervals $I$ for which there 
exists a $Q$ with $I_Q = I$. As we already remarked, since every $Q$ has the property that

$$2^{k_2} \leq \frac{\dist (I_Q, S_{\vec{Q}})}{|I_Q|} \leq 2^{k_2 + 1}$$
all these intervals $I$ are disjoint. Denote also by $\I^l$ the collection of all
dyadic intervals $I^l$ defined by $I^l : = I + l|I|$ for some $I\in \I$. It is also not difficult
and important to observe that the intervals in $\I^l$ have bounded overlap. 
Also, for any $Q$ we denote by $I_Q^l$ the interval $I_Q^l : = I_Q + l |I_Q|$. Clearly, each function $\widetilde{\Phi}_Q^l$ is supported on a certain fixed
enlargement (with a factor of $3$ say) of the interval $I_Q^l$. We will use these
observations and notations later on. 

Using these, we can estimate (\ref{102}) by

$$\langle \sum_Q C_Q \widetilde{\Phi}_Q^l, \sum_P C_P \Phi_P \rangle =
\sum_{P, Q} C_P C_Q \langle \widetilde{\Phi}_Q^l, \Phi_P \rangle = $$

$$\sum_T \sum_{P\in T} C_P \langle \sum_Q C_Q \widetilde{\Phi}_Q^l, \Phi_P \rangle =$$

$$\sum_T \sum_{P\in T} C_P \langle \sum_{Q : I^l_Q \subseteq I_T} C_Q \widetilde{\Phi}_Q^l, 
\Phi_P \rangle +
\sum_T \sum_{P\in T} C_P \langle \sum_{Q : I^l_Q \subseteq I_T^c} C_Q \widetilde{\Phi}_Q^l, 
\Phi_P \rangle : = \alpha + \beta.$$

\underline{Estimates for $\beta$.}

To understand $\beta$ we write

$$\beta \lesssim
\sum_T \left( \sum_{P\in T} |C_P|^2 \right)^{1/2}
\left(\sum_{P\in T}
|\langle \sum_{Q : I^l_Q \subseteq I_T^c} C_Q \widetilde{\Phi}_Q^l, \Phi_P \rangle |^2
\right)^{1/2}\lesssim
$$

\begin{equation}
\frac{1}{S^{1/2}}
\sum_T |I_T|^2 \left(\sum_{P\in T}
|\langle \sum_{Q : I^l_Q \subseteq I_T^c} C_Q \widetilde{\Phi}_Q^l, \Phi_P \rangle |^2
\right)^{1/2}.
\end{equation}
We will show next that 

\begin{equation}\label{104}
\left(\sum_{P\in T}
|\langle \sum_{Q : I^l_Q \subseteq I_T^c} C_Q \widetilde{\Phi}_Q^l, \Phi_P \rangle |^2
\right)^{1/2} \lesssim \frac{1}{(S')^{1/2}} |I_T|^{1/2}.
\end{equation}
If (\ref{104}) were true, then the estimate on $\beta$ could be completed as follows

$$\beta \lesssim 
\frac{1}{S^{1/2}}
\sum_T |I_T|^{1/2}\cdot \frac{1}{(S')^{1/2}} |I_T|^{1/2} =
\frac{1}{S^{1/2}}\frac{1}{(S')^{1/2}} \sum_T |I_T| =
\frac{1}{S^{1/2}}\frac{1}{(S')^{1/2}} S = 
\frac{S^{1/2} }{(S')^{1/2} }\lesssim 2^{- 9 M k_2}
$$
as desired. It is therefore enough to show (\ref{104}).

Fix $P\in T$. Then, the corresponding inner term in (\ref{104}) can be estimated by

$$|\langle \sum_{Q : I^l_Q \subseteq I_T^c} C_Q \widetilde{\Phi}_Q^l, \Phi_P \rangle |
\lesssim \sum_{Q : I^l_Q \subseteq I_T^c} |C_Q| 
|\langle \Phi_P, \widetilde{\Phi}_Q^l \rangle |\lesssim
$$

$$
\frac{1}{(S')^{1/2}}
\sum_{Q : I^l_Q \subseteq I_T^c} |I_Q|^{1/2} 
|\langle \Phi_P, \widetilde{\Phi}_Q^l \rangle | = 
\frac{1}{(S')^{1/2}}
\sum_{Q : I^l_Q \subseteq I_T^c} |I_P|^{- 1/2} 
|\langle \Phi_P^{\infty}, \widetilde{\Phi}_Q^{l, \infty} \rangle | =
$$

$$
\frac{1}{(S')^{1/2}}
|I_P|^{- 1/2} 
\sum_{Q : I^l_Q \subseteq I_T^c} 
|\langle \Phi_P^{\infty}, \widetilde{\Phi}_Q^{l, \infty} \rangle |,
$$
where $\Phi_P^{\infty} : = |I_P|^{1/2} \Phi_P$ and 
$\widetilde{\Phi}_Q^{l, \infty} : = |I_Q|^{1/2} \widetilde{\Phi}_Q^{l}$
and they are both $L^{\infty}$ normalized functions.

We claim now that 

\begin{equation}\label{105}
\sum_{Q : I^l_Q \subseteq I_T^c}
|\langle \Phi_P^{\infty}, \widetilde{\Phi}_Q^{l, \infty} \rangle | \lesssim
\left( 1 + \frac{\dist (I_P, I_T^c)}{|I_P|}\right)^{- m} |I_P|
\end{equation}
for any positive integer $m$, with the implicit constants depending on it.

If we assume the claim, the corresponding last term can be estimated further by

$$\frac{1}{(S')^{1/2}}
|I_P|^{1/2}\left( 1 + \frac{\dist (I_P, I_T^c)}{|I_P|}\right)^{- m}
$$
and as a consequence, the left hand side of (\ref{104}) becomes smaller than

$$\frac{1}{(S')^{1/2}}
\left(\sum_{P\in T}
\left( 1 + \frac{\dist (I_P, I_T^c)}{|I_P|}\right)^{- 2 m} |I_P| \right)^{1/2}
$$
and this as we have already seen before is smaller than
$\frac{1}{(S')^{1/2}} |I_T|^{1/2}$
as desired.

It is therefore enough to prove the previous claim (\ref{105}).

Split the left hand side of it as

$$
\sum_{Q : I^l_Q \subseteq I_T^c, |I^l_Q|\leq |I_P|}
|\langle \Phi_P^{\infty}, \widetilde{\Phi}_Q^{l, \infty} \rangle | + 
\sum_{Q : I^l_Q \subseteq I_T^c, |I^l_Q| > |I_P|}
|\langle \Phi_P^{\infty}, \widetilde{\Phi}_Q^{l, \infty} \rangle | : =
C_1 + C_2.
$$

To estimate $C_1$, let us assume that our collection $\I^l$ defined before can be listed as

$$\I^l = \{ I_1, ... , I_K \}.$$
Using this and also Lemma \ref{scalarproduct} one can write

$$
C_1 = \sum_{j=1}^K
\sum_{Q : I^l_Q \subseteq I_T^c, |I^l_Q|\leq |I_P|, I^l_Q = I_j}
|\langle \Phi_P^{\infty}, \widetilde{\Phi}_Q^{l, \infty} \rangle |
\lesssim \sum_{j=1}^K
\left( 1 + \frac{\dist (I_P, I_j)}{ |I_P|}\right)^{ - m} |I_j|$$
and the last sum is clearly smaller than

$$
\left( 1 + \frac{\dist (I_P, I_T^c)}{ |I_P|}\right)^{ - m} |I_P|
$$
as required by (\ref{105}).

To estimate $C_2$, this time we can write

$$
\sum_{Q : I^l_Q \subseteq I_T^c, |I^l_Q| > |I_P|}
|\langle \Phi_P^{\infty}, \widetilde{\Phi}_Q^{l, \infty} \rangle | =
\sum_{\l = 1}^{\infty}
\sum_{Q : I^l_Q \subseteq I_T^c, |I^l_Q| > |I_P|, |I^l_Q| = 2^{\l} |I_P|}
|\langle \Phi_P^{\infty}, \widetilde{\Phi}_Q^{l, \infty} \rangle |.
$$ 
Fix $\l$ and look at the corresponding inner sum. It is not difficult to remark that
for every $Q$ as there, one has $2^{k_2 + \l- 5} I_P \cap I_Q = \emptyset$. Tacking also into 
account the fact that all the functions $\widetilde{\Phi}_Q^{l, \infty}$ have compact support
and applying carefully several times Lemma \ref{scalarproduct} one obtains for $C_2$ 
the upper bound

$$2^{\l} \frac{1}{2^{(k_2 + \l) m}}
\left(1 + \frac{\dist (I_P, I_T^c)}{ |I_P|}\right)^{ - m} |I_P|
$$
which is fine, since it is an expression summable over $\l$.
This ends the discussion on $\beta$, we start now estimating term $\alpha$.

\underline{Estimates for $\alpha$.}

We now write

$$\alpha =
\sum_T \sum_{P\in T} C_P
\langle \sum_{Q: I^l_Q\subseteq I_T } C_Q \widetilde{\Phi}_Q^l, \Phi_P\rangle =
\sum_T \sum_{P\in T, Q : I^l_Q\subseteq I_T } C_P C_Q
\langle \widetilde{\Phi}_Q^l, \Phi_P\rangle.
$$
Fix $T$. We will show that

\begin{equation}\label{106}
\left|\sum_{P\in T, Q : I^l_Q\subseteq I_T } C_P C_Q
\langle \widetilde{\Phi}_Q^l, \Phi_P\rangle\right|
\lesssim \frac{1}{S^{1/2}}
\frac{1}{(S')^{1/2}} |I_T|.
\end{equation}
If we accept for a moment (\ref{106}) then $\alpha$ becomes smaller than

$$
\frac{1}{S^{1/2}}
\frac{1}{(S')^{1/2}} \sum_T |I_T| = \frac{1}{S^{1/2}}\frac{1}{(S')^{1/2}} S =
\frac{S^{1/2}}{(S')^{1/2}} \lesssim 2^{- 9 M k_2}
$$
which would be the desired upper bound.
We are therefore left with understanding (\ref{106}).

We split the left hand side of it as

$$
\left|\sum_{P\in T, Q : I^l_Q\subseteq I_T, |I^l_Q|\leq |I_P| } C_P C_Q
\langle \widetilde{\Phi}_Q^l, \Phi_P\rangle\right| + 
\left|\sum_{P\in T, Q : I^l_Q\subseteq I_T, |I^l_Q| > |I_P| } C_P C_Q
\langle \widetilde{\Phi}_Q^l, \Phi_P\rangle\right| : = A + B.$$

\underline{Estimates for $B$.}

Pick an index $1\leq j\leq K$ with the property that $I_j\subseteq I_T$ and consider
the corresponding sum

$$
\left|\sum_{P\in T, Q : I^l_Q\subseteq I_T, |I^l_Q| > |I_P|, I^l_Q = I_j } C_P C_Q
\langle \widetilde{\Phi}_Q^l, \Phi_P\rangle\right|.
$$
It can be further majorized by

$$\sum_{\l}
\left|\sum_{P\in T, Q : I^l_Q\subseteq I_T, |I^l_Q| > |I_P|, I^l_Q = I_j, |I^l_Q| = 2^{\l}|I_P| } C_P C_Q
\langle \widetilde{\Phi}_Q^l, \Phi_P\rangle\right|.
$$
Arguing as before (by using Lemma \ref{scalarproduct}) we deduce that the previous expression
can be estimated by

$$
\frac{1}{S^{1/2}}
\frac{1}{(S')^{1/2}}
2^{\l}\sum_{P\in T: |I_j| = 2^{\l}|I_P|}
\left( 1 + \frac{\dist (I_j, I_P)}{|I_P|}\right)^{- m} |I_j| \lesssim
\frac{1}{S^{1/2}}
\frac{1}{(S')^{1/2}} 2^{\l}
\frac{1}{2^{(k_2 + \l) m}} |I_j|.
$$
Now the above expression is clearly summable over $\l$ and the new bound is also summable
over $j$. In the end we clearly obain an upper bound of the form 
$$\frac{1}{S^{1/2}}
\frac{1}{(S')^{1/2}} |I_T|
$$
since all the intervals $I_j$ which contribute, are disjoint and included in $I_T$.

\underline{ Estimates for $A$.}

To estimate $A$ we simply have to come back to the original wave packets, which are compactly
suppoted in frequency. Since for every $Q$ the function $\widehat{\widetilde{\Phi}_Q^l  }$
is a bump adapted to the interval $\omega_Q$, we can split it accordingly as

$$\widetilde{\Phi}_Q^l = \sum_{\l}
\frac{1}{(1 + |\l|)^{10}}
\Phi_Q^{l, \l}
$$
where clearly, $\Phi_Q^{l, \l}$ is a wave packet adapted to the tile
$Q^{l, \l} : = I^l_Q \times \omega_Q^{\l}$ where $\omega_Q^{\l}$ is the interval defined by
$\omega_Q^{\l} : = \omega_Q + \l |\omega_Q|$.

As a consequence, our term $A$ can be majorized by

$$
\sum_{\l}
\frac{1}{(1 + |\l|)^{10}}
\left|\sum_{P\in T, Q : I^l_Q\subseteq I_T, |I^l_Q|\leq |I_P| } C_P C_Q
\langle \Phi_Q^{l, \l}, \Phi_P\rangle\right|.
$$
It is important to observe now that for each fixed $\l$ the corresponding set of tiles
$Q^{l,\l}$ is still strongly disjoint. Also, since both $\Phi_Q^{l, \l}$ and $\Phi_P$
have compact Fourier support, it is clear that $\langle \Phi_Q^{l, \l}, \Phi_P\rangle = 0$
unless $\omega_Q^{\l}\cap \omega_P \neq \emptyset$.

Fix $\l$ now. The inner sum above can also be written as

$$\left|\sum_{P\in T, Q : I^l_Q\subseteq I_T, \omega_Q^{\l}\cap \omega_P \neq \emptyset,
|\omega^{\l}_Q| > |\omega_P| } C_P C_Q
\langle \Phi_Q^{l, \l}, \Phi_P\rangle\right|.
$$
Denote also as before by $\Q_T$ the set of all tiles $Q$ with $I_Q^l \subseteq I_T$ for which there exists a tile $P\in T$
so that $\omega_Q^{\l}\cap \omega_P \neq \emptyset$ and $|\omega^{\l}_Q| > |\omega_P|$.
Note also that those $Q$ inside $\Q_T$ must have disjoint $I^l_Q$ intervals.
It is also not difficult to see that the above expression can also be written as

\begin{equation}\label{107}
|\sum_{P\in T} C_P \langle h_T, \Phi_P \rangle |,
\end{equation}
where $h_T := \sum_{Q\in \Q_T} C_Q \Phi_Q^{l, \l}$.
Then, (\ref{107}) is smaller than

\begin{equation}\label{108}
\left(\sum_{P\in T} |C_P|^2\right)^{1/2}
\left(\sum_{P\in T}|\langle h_T, \Phi_P \rangle|^2\right)^{1/2} \lesssim
\frac{|I_T|^{1/2}}{S^{1/2}}\cdot \|h_T\|_2.
\end{equation}
Now, one also has

$$\|h_T\|_2 \lesssim
(\sum_{Q\in\Q_T} |C_Q|^2 )^{1/2} \lesssim
\frac{1}{(S')^{1/2}} (\sum_{Q\in\Q_T} |I_Q|)^{1/2} =
$$

$$
\frac{1}{(S')^{1/2}} (\sum_{Q\in\Q_T} |I^l_Q|)^{1/2}
\lesssim \frac{1}{(S')^{1/2}} |I_T|^{1/2}
$$

Using this in (\ref{108}) finishes the proof.


\begin{thebibliography}{99}

\bibitem{ablowitzsegur}Ablowitz, M. and Segur, H. {\it Solitons and the inverse scattering 
transform},
SIAM Studies in Mathematics, [1981].





\bibitem{carleson} Carleson, L.
{\it On convergence and growth of partial sums of Fourier series},
Acta Math. vol. 116, 135-157, [1966].





\bibitem{ck1} Christ, M. and Kiselev, A.
{\it WKB asymptotic behaviour of almost all generalized eigenfunctions of one-dimensional
Schr\"{o}dinger operators}, J.Funct.Anal., vol. 179, 426-447, [2001].

\bibitem{ck2} Christ, M. and Kiselev, A.
{\it Maximal functions associated to filtrations}, J.Funct.Anal., vol. 179, 409-425, [2001].





\bibitem{coifmanm6} Coifman, R. R and Meyer, Y. 
{\it Ondelettes et op\'erateurs III, Op\'erateurs multilin\'eaires},
Actualit\'es Math\'ematiques, Hermann, Paris, [1991].




\bibitem{fefferman1} Fefferman, C. {\it Pointwise convergence of Fourier series}, Ann. of Math. vol. 98, 551--571 [1973].

\bibitem{laceyt1}Lacey, M. and Thiele, C. {\it $L^p$ estimates on the bilinear Hilbert transform for $2<p<\infty$},
Ann. of Math., vol. 146, 693-724, [1997].

\bibitem{laceyt2}Lacey, M. and Thiele, C. {\it On Calder\'{o}n's conjecture},
Ann. of Math., vol. 149, 475-496, [1997].



\bibitem{laceyt3}Lacey, M. and Thiele, C. {\it A proof of the boundedness of Carleson operator}, Math. Res. Lett., vol 7, 361-370, [2000].


\bibitem{lm}Li, X. and Muscalu, C. {\it Gereralizations of the Carleson-Hunt theorem I. The classical singularity case}, Amer. J. Math., vol. 129,
983-1019, [2007].


\bibitem{mtt-1} Muscalu, C., Tao, T., and Thiele, C. {\it Multi-linear operators given by singular multipliers},
 J.Amer.Math.Soc., vol.15, 469-496, [2002].

\bibitem{mtt-4} Muscalu, C., Tao, T., and Thiele, C. {\it $L^p$ estimates for the biest I. The Walsh case}, 
Math.Ann., vol. 329, 401-426, [2004].

\bibitem{mtt-5} Muscalu, C., Tao, T., and Thiele, C. {\it $L^p$ estimates for the biest II. The Fourier case},
Math.Ann., vol. 329, 427-461, [2004].

\bibitem{mtt-6} Muscalu, C., Tao, T., and Thiele, C. {\it A counterexample to a multilinear endpoint question of Christ and Kiselev}, Math.Res.Lett., vol. 10, 237-246, [2003].

\bibitem{mtt-7} Muscalu, C., Tao, T., and Thiele, C. {\it A discrete model for the bi-Carleson operator},
Geom.Funct.Anal., vol. 12, 1324-1364, [2002].

\bibitem{mtt-8} Muscalu, C., Tao, T., and Thiele, C. {\it The bi-Carleson operator},
Geom.Funct.Anal., vol. 16, 230-277, [2006].

\bibitem{mtt-11} Muscalu, C., Tao, T. and Thiele, C. {\it  A Carleson theorem for a Canor group model of the scattering transform}, Nonlinearity, vol. 16, 216-256, [2003].

\bibitem{mptt1} Muscalu, C., Pipher, J., Tao, T. and Thiele, C. {\it Bi-parameter paraproducts}, Acta Math., vol. 193, 269-296, [2004].

\bibitem{mptt2} Muscalu, C., Pipher, J., Tao, T. and Thiele, C. 
{\it Multi-parameter paraproducts}, Revista Mat. Iberoamericana, vol. 22, 963-977, [2006].


\bibitem{simon} Simon, B. {\it Bounded eigenfunctions and absolutely continuous spectra for one-dimensional Schr\"{o}dinger operators},
Proc. Amer. Math. Soc., vol. 124, 3361-3369, [1996].



\bibitem{stein}Stein, E. {\it Harmonic Analysis: Real Variable Methods,
Orthogonality, and Oscillatory Integrals.} Princeton University Press,
Princeton, [1993].





\end{thebibliography}
\end{document}